\documentclass[onefignum,onetabnum]{siamart171218}

\usepackage{lipsum}
\usepackage{amsfonts}
\usepackage{graphicx}
\usepackage{epstopdf}
\usepackage{bm}
\usepackage{amssymb}
\usepackage{algpseudocode}
\usepackage{amsmath}
\usepackage{lineno,hyperref,dsfont}
\usepackage{amscd}
\usepackage{mathrsfs}
\usepackage{stmaryrd}
\usepackage{color}
\usepackage{ifthen}
\usepackage{subfigure}
\usepackage{epsfig,wrapfig,epic}
\usepackage{url}
\usepackage{dsfont}
\usepackage{marvosym}
\usepackage[none]{hyphenat}
\ifpdf
  \DeclareGraphicsExtensions{.eps,.pdf,.png,.jpg}
\else
  \DeclareGraphicsExtensions{.eps}
\fi

\newcommand{\rank}{\textsf{rank}}

\newcommand{\Symbol}{\mathcal{S}}
\newcommand{\Jac}{\mathcal{J}}
\newcommand{\Sing}{\mathfrak{S}}
\newcommand{\LD}{\textsc{ld}}
\newcommand{\LCD}{\textsc{lcd}}

\newcommand{\ODE}{\sc{ode}}
\newcommand{\DAE}{\sc{dae}}

\newcommand{\ILP}{\sc{ilp}}
\newcommand{\randpoint}{\ensuremath{\mathfrak{a}}}

\newcommand{\R}{\mathbb{R}}

\newcommand{\ie}{{i}.{e}.}
\newcommand{\eg}{{e}.{g}. }
\newcommand{\etc}{{etc}.}

\newtheorem{example}{Example}[section]
\newtheorem{remark}{Remark}[section]
\newtheorem{prop}{Proposition}[section]

\newtheorem{define}{Definition}[section]

\def\ctotDer{\textbf{D}}

\def\proof{\textsc{Proof.}}
\def\foorp{\hfill$\square$}

\newsiamremark{hypothesis}{Hypothesis}
\crefname{hypothesis}{Hypothesis}{Hypotheses}
\newsiamthm{claim}{Claim}

\usepackage{algorithm}
\usepackage{amsopn}

\makeatletter
\newenvironment{breakablealgorithm}
{% \begin{breakablealgorithm}
	\begin{center}
		\refstepcounter{algorithm}% New algorithm
		\hrule height.8pt depth0pt \kern2pt% \@fs@pre for \@fs@ruled
		\renewcommand{\caption}[2][\relax]{% Make a new \caption
			{\raggedright\textbf{\ALG@name~\thealgorithm} ##2\par}%
			\ifx\relax##1\relax % #1 is \relax
			\addcontentsline{loa}{algorithm}{\protect\numberline{\thealgorithm}##2}%
			\else % #1 is not \relax
			\addcontentsline{loa}{algorithm}{\protect\numberline{\thealgorithm}##1}%
			\fi
			\kern2pt\hrule\kern2pt
		}
	}{% \end{breakablealgorithm}
		\kern2pt\hrule\relax% \@fs@post for \@fs@ruled
	\end{center}
}
\makeatother

\headers{Index Reduction for Degenerated {\DAE}s}{W. Yang, W. Wu, and G. Reid}

\title{Index Reduction for Degenerated Differential-Algebraic Equations by Embedding and Real Numerical Algebraic Geometry \thanks{Submitted to the editors DATE.
\funding{This work is partially supported by the projects of Chongqing (2020000036, 2021000263,cstc2020yszx-jcyjX0005) and the National Natural Science Foundation of China (11771421)}}}

\author{
{Wenqiang Yang}
\thanks{Chongqing Key Laboratory of Automated Reasoning and Cognition, Chongqing Institute of Green and Intelligent Technology, Chinese Academy of Sciences. \Letter  Corresponding Author: wuwenyuan@cigit.ac.cn.}
\and{}\thanks{Chongqing School, University of Chinese Academy of Sciences.}
\and{Wenyuan Wu\footnotemark[2] \Letter}
\and{Greg Reid\footnotemark[4]}\thanks{Mathematics Department, University of Western Ontario.}}

\makeatletter
\newcommand*{\addFileDependency}[1]{% argument=file name and extension
  \typeout{(#1)}% latexmk will find this if $recorder=0 (however, in that case, it will ignore #1 if it is a .aux or .pdf file etc and it exists! if it doesn't exist, it will appear in the list of dependents regardless)
  \@addtofilelist{#1}% if you want it to appear in \listfiles, not really necessary and latexmk doesn't use this
  \IfFileExists{#1}{}{\typeout{No file #1.}}% latexmk will find this message if #1 doesn't exist (yet)
}
\makeatother

\newcommand*{\myexternaldocument}[1]{%
    \externaldocument{#1}%
    \addFileDependency{#1.tex}%
    \addFileDependency{#1.aux}%
}
\ifpdf
\hypersetup{
  pdftitle={Index Reduction for Degenerated Differential-Algebraic Equations by Embedding},
  pdfauthor={W. Yang, W. Wu, and G. Reid}
}
\fi

\myexternaldocument{ex_supplement}

\begin{document}
\maketitle
% REQUIRED
\begin{abstract}
\sloppy{}
To find consistent initial data points (witness points) for a system of differential-algebraic equations, requires the identification of its
missing (hidden) constraints arising from differentiation of the system.
An efficient class of so-called structural methods exploiting a dependency graph for this task was initiated by Pantiledes.
The choice of graph is related to the choice of ordering that minimize the solution dimension of the problem.
However, this method may fail. More complete methods rely on differential-algebraic geometry but suffer from other issues (e.g. high complexity and instability on approximate data).
In this paper we give a new class of efficient structural methods combined with new tools from numerical real algebraic geometry that has much improved completeness properties.
Existing structural methods may fail for a system of differential-algebraic equations if its Jacobian matrix after differentiation is still singular due to symbolic cancellation or numerical degeneration.
Existing structural methods can only handle degenerated cases caused by symbolic cancellation. However, if a system has parameters, then its parametric Jacobian matrix may be still singular after application of the structural method for certain values of the parameters. This case is called numerical degeneration.

In this paper, for polynomially nonlinear systems of differential-algebraic equations, numerical methods are given to solve both degenerated cases using numerical real algebraic geometry. First, we introduce a witness point method, which produces at least one witness point on every constraint component (manifold).
This can help to ensure constant rank and detection of degeneration on all components of such systems.
Secondly, we present a Constant Rank Embedding Lemma, then based on this lemma we propose an Index Reduction by Embedding (IRE) method which can construct an equivalent system with a full rank Jacobian matrix. Thirdly, IRE leads to a global
structural differentiation method, to solve degenerated differential-algebraic equations on all components numerically.
Application examples of models from circuits, mechanics, and motion are used to demonstrate our method and its advantages.
\end{abstract}

% REQUIRED

% REQUIRED

\section{Introduction}\label{sec:intro}
\sloppy{}

Systems of differential-algebraic equations are widely used to model and simulate dynamical systems such as mechanical systems,
electrical circuits, and chemical reaction plants {\cite{Ilchmann15}}.
We will often use the abbreviation
{\DAE}\footnote{A system of differential-algebraic equations will be denoted by {\DAE} while {\DAE}s will denote several such systems.}
for a system of differential-algebraic equations.
The name arose since such systems usually contain differential equations with derivatives and algebraic
equations without derivatives.
It was initially believed that any such {\DAE} could be easily
converted by coordinate changes and eliminations to a traditional explicit {\ODE} - the so-called
underlying {\ODE}\footnote{A system of explicit ordinary differential equations,
in solved form for their highest derivatives, will be denoted by {\ODE}.}.
However, they are now recognized as common and natural in applications, to the extent
that several modeling environments have them as their central object.
Indeed, the explicit underlying {\ODE} may be too expensive to un-cover, and modern codes for numerical
solution of a {\DAE} have made this unnecessary in most cases.  Even the idea there is a unique underlying {\ODE} is misleading since the
{\DAE} may yield several component manifolds with different behavior and different underlying {\ODE} on each component.
%Throughout {\DAE} will be used as abbreviations for systems of differential equations.
%Similarly {\ODE} will be used as abbreviations for systems of differential equations.

{\DAE}s are a subset of the set of general systems of partial differential equations.  The great geometer Cartan conjectured but was not able to prove that
after a finite number of prolongations (differentiations) of such systems, they would become involutive, and a local existence and uniqueness
theorem could be stated for their solutions.  Another greater geometer Kuranishi eventually proved this result, albeit under conditions that are difficult
to render explicitly {\cite{Kuranishi57}}. The number of prolongations (differentiations) to
uncover the underlying {\ODE} (i.e. the differential index) corresponds to this result for {\DAE}s; and it is equivalent to obtaining
all the constraints on initial data for existence and uniqueness of solutions. We note that the differential index and underlying
{\ODE} may be different on different components of the {\DAE}.

One may try to numerically solve a {\DAE} directly without reducing its index by prolongation.  However, properly posing initial values for a
{\DAE} requires that they satisfy the missing constraints, and hence implicitly requires knowledge of the prolonged form of the {\DAE}.
This direct approach is prone to order reduction, instability, inaccuracy and the tendency for the approximate solution to drift off the constraints that increases with the index.
It is usually only feasible for low index problems {\cite{LIU201593,Skvortsov12,LIU2007748,Awawdeh09,Pryce98}}.
For references related to differential and \textit{perturbation index} see Hairer and Wanner {\cite{Hairer91}}, Campbell and Gear {\cite{Campbell95}}, and
Reid, Lin and Wittkopf {\cite{Reid01}}.

Note that the name {\DAE} misleadingly suggests that a {\DAE} can be partitioned into differential equations and non-differential equations (algebraic equations) where the latter
are regarded as constraints. Consider the {\DAE}
\begin{equation}\label{point}
  u' u''  +  u u' + x = 0,  ((u')^2 + u^2 + x^2 - 1)((u')^2 + u^2 + x^2 - 4 ) = 0
\end{equation}
where $u$ is a unknown function of $x$.
Then $((u')^2 + u^2 + x^2 - 1)((u')^2 + u^2 + x^2 - 4 ) = 0$ is a constraint even though it contains derivatives.
Geometrically there are $2$ constraint components (spheres of radius $1$ and $2$).

In contrast to the above direct approach, indirect and widely used approaches first use
index reduction only by differentiation \cite{Pryce01,Pryce98,Pantelides88,Gear88,Fritzson14} followed
by consistent initial point determination \cite{Brenan95,Taihei19,Pantelides88}.

In this paper we make contributions to such indirect approaches.
In particular for polynomially nonlinear {\DAE}s we apply a new efficient prolongation method to reduce their index,
which implicitly gives the hidden constraint components of initial data, then determine consistent initial points using new methods from real numerical algebraic geometry.
\subsection{Previous Work}\label{ssec:pre-work}

\subsubsection{Consistent Point}
Finding at least one consistent point on each constraint component of a {\DAE}, is an important problem as it
determines the initial point for a numerical solution \cite{Taihei19,Pantelides88}.
Commonly used methods to obtain such consistent initial points are the approximation method \cite{Leimkuhler91} and the transformation method \cite{Vieira20011,Brown98}. The approximation method starts with a guess for an initial point and then iteratively refines it {\cite{Shampine02}}. A good guess is critical for convergence of the iterative method {\cite{Shampine02}}.

We note that most treatments assume there is just one constraint component, and indeed
that equations whose set of solutions correspond to this constraint component (constraint equations) can be
explicitly found.  Geometrically constraint components are projections from the higher dimensional space with the
derivatives regarded as indeterminates (the so-called Jet space of the {\DAE}).  For linear {\DAE}s and polynomially nonlinear {\DAE}s with rational coefficients elimination algorithms are known for explicitly rendering equations for
their constraints.  But no algorithms are known for the general case of analytic {\DAE}s, though there are some known
for subclasses of analytic {\DAE}s.

\subsubsection{Index Reduction}
Indices are used to measure how far a {\DAE} is from a {\DAE} which includes its missing constraints, or is implicitly equivalent to an {\ODE}. The \textit{Kronecker index} {\cite{Gear83,Lamour13}} is applicable to constant coefficient linear {\DAE}. The \textit{tractability index} {\cite{Lamour13,Roswitha02,GUZEL06}} and \textit{strangeness index} {\cite{KUNKEL94}} applies to linear variable coefficient {\DAE}. Further, the \textit{tractability index} can be extended to a \textit{topological index} {\cite{Tischendorf98}} in some applications, and the strangeness index also can be extended to non-square {\DAE} {\cite{Kunkel06}}.  The \textit{perturbation index} {\cite{Campbell95}} is defined in terms of perturbations of nonlinear autonomous {\DAE}.  The \textit{differential index} {\cite{Campbell95,Campbell1995}} is the minimum number of differentiation times required to transform a {\DAE} system into an explicit {\ODE} system, and is used in our paper.

In order to solve a {\DAE} accurately, index reduction is an essential and important operation. The accurate numerical solution of a high index ($\geq 2$) {\DAE} is difficult to obtained directly {\cite{Shampine02}}. Thus, we emphasize the need to convert a higher index {\DAE} to a low ($\leq 1$) index {\DAE}. After sufficient differentiation, all the time derivatives of the existing differential variables can be replaced by new variables to realize the index reduction {\cite{Pantelides88}}.
Gear {\cite{Gear88}} proposed a method by repeatedly finding algebraic equations and dealing with them by differential processing until the system becomes an {\ODE}. However, these methods are notoriously hard for large and non-linear systems.
The arguments in {\cite{Pantelides88}} and {\cite{Gear88}} depend on liberal use of the implicit function theorem for
analytic functions under tacitly assumed unstated rank conditions.  The general finite termination of prolongation of analytic systems of partial differential equations yielding in finite steps involutive systems for which an existence and uniqueness theorem can be given, was conjectured by Cartan.
Kuranishi {\cite{Kuranishi57}} eventually proved this famous and difficult result that had eluded Cartan.

For polynomially nonlinear {\DAE} with rational coefficients, there are symbolic differential-elimination algorithms that reduce index of {\DAE}, but these algorithms are often unstable when applied to approximate {\DAE} and also have high worst case complexity.
Fortunately there are some efficient methods based on bipartite graph preprocessing that can sometimes reduce the differential index. These methods have been implemented in {\DAE} simulation packages such as Dymola, Open-Modelica,
MapleSim {\cite{Fritzson14}}, Simulink and Isam{\DAE} {\cite{Caillaud20}}.  Such methods originated with work by
Pantelides {\cite{Pantelides88}} who presented a graph-based preprocessing method that can sometimes by prolongation reduce a {\DAE} to involutive (index 0 or 1) form containing the underlying {\ODE} that decides consistent initial data for
numerical solutions.  Crucially Pantelide's method and its later developments have proven to be successful often enough in applications that they have become a standard part of the software environments mentioned above. Such developments include Mattsson-S\"{o}derlind's (MS) Method {\cite{Mattsson93}} which employs an amending method to introduce new variables to replace dummy derivatives {\cite{McKenzie17}}. Pryce {\cite{Pryce01}} further generalized it to a more direct and widely applicable method by solving an assignment problem. Zolfaghari, Taylor and  Spiteri {\cite{ZOLF2021}} further extended Pryce method to the application of integro-differential–algebraic equations.

\subsubsection{ Improved Structural Methods}

 Despite the success of structural analysis by index reduction, the methods may fail for a {\DAE} if its Jacobian after differentiation is singular, and it is essential to develop improved structure methods.

 Campbell {\cite{Campbell93}} proposed a direct method, which can regularize a {\DAE} in theory by sufficiently differentiating the {\DAE} and simplifying it with an elimination method. But the symbolic elimination process can be very complex and inefficient for nonlinear {\DAE}.

Linear {\DAE} with constant coefficients can be transformed into the canonical form of Weierstra$\beta$, the Kronecker index determined and then the {\DAE} can be solved directly {\cite{Gerdts11}}.  This transformation is neither a strict equivalence transformation \cite{Iwata03} nor a unimodular transformation \cite{Murota95}. Murota {\cite{Murota95}} proposed a general framework ``combinatorial relaxation" algorithm to compute the degree of a certain determinant based on its upper bound, which is defined in terms of the maximum weight of a perfect matching in an associated graph. Iwata {\cite{Iwata03}} improved the combinatorial relaxation algorithm by an equivalence transformation with constant matrices, reducing computational complexity. X. Wu {\cite{Wu13}} applied the modified combinatorial relaxation algorithm to analyze the resulting error behavior. In particular he
gave an algorithm to detect and correct the error from structural index reduction by matrix pencils.

Compared with the method of X. Wu et al. \cite{Wu13}, the LC-method of Tan et al. \cite{Tan17} also considers equations and their derivatives, with better results for some nonlinear {\DAE}s. Unfortunately, although this method may guarantee a global equivalence transformation, it can only be used by its norm space. The ES-method \cite{Tan17} uses new variables to seek the solution in a projection of a higher dimensional space, and it can be considered as a supplement for the LC-method. The LC-method replaces equations, while the ES-method replaces variables.
If the global equivalence transformation of both methods or nether of them is guaranteed, then the LC-method can be used, otherwise, the ES-method can be used. Both of the above two methods can also be extended and applied to some integro-differential–algebraic equations {\cite{ZOLF2021}}.
%The incapability of the LC-method and the ES-method is again nonlinear in the derivatives of highest order.
The substitution method \cite{Taihei19} aims to modify non-linear {\DAE}, and it is a local equivalence method. Like the method of Campbell, it avoids excessive elimination through targeted variable selection. For {\DAE} with high non-linearity, this method is usually too complex to be applied. In order to avoid the complexity of elimination calculations, the augmentation method \cite{Taihei19} adopts the principle similar to the ES-method, and is also a local equivalence method.

\subsection{Problem Statement}
Mathematical models of circular motion in kinematics, mechanical structures and chemical processes etc., often provide polynomially nonlinear {\DAE}s. Consequently they may have more than one solution component (see Example {\ref{ex:3}}). For global information about solutions, we need at least one consistent initial point on
each component.  Such consistent initial points are hard to obtain for a polynomial system with many variables by using symbolic computation, {\eg} by using Groebner Bases \cite{Geddes1992} or Triangular Decomposition \cite{Collins1982}.
Further, Newton iterative solvers usually require a starting point sufficiently close to a solution.
To obtain such global information is one of the two main goals in this paper.

Moreover, the success of structural analysis methods for {\DAE} depends on the regular Jacobian assumption after index reduction.
In many cases, this assumption is valid. However, we will present counterexamples from real applications.

Such cases are called ``\textbf{degeneration}" cases, which means the Jacobian matrix is singular on a whole component, including two types: \textbf{symbolic cancellation} (see Example \ref{ex:2}) and \textbf{numerical degeneration} (see Example \ref{ex:3}).

\begin{remark}
In fact, singularities can occur only at special points along a solution. For example in Equation (\ref{point}), when $u$ is increasing, $u'$ will gradually change to $0$, leading to a singular Jacobian for the equation. We will not consider this kind of problem in this paper.
\end{remark}

%\begin{example}\label{ex:2} Symbolic Cancellation: \begin{description} \item[(1)] Consider the following {\DAE} --- structurally singularity (not perfect matching){\cite{Taihei19}}: \[ \left\{ \begin{array}{rcl} x_{1}^2&= 0\\ (x_{2}-1)^2 &=0 \end{array} \right . \] \item[(2)] Consider the following {\DAE} --- structurally singularity (perfect matching){\cite{Tan17}}: \[ \left\{ \begin{array}{rcl} (xy){'}-x{'}y-xy{'}+2x+y-3&= 0\\ x+y-2 &=0 \end{array} \right . \] \item[(3)] Consider the following {\DAE} --- identically singularity{\cite{Taihei19}}: \[ \left\{ \begin{array}{rcl} \dot{x}_{1}+\dot{x}_{2}+{x}_{3}&= 0\\ \dot{x}_{1}+\dot{x}_{2} &=0\\ {x}_{2}+\dot{x}_{3}&= 0 \end{array} \right . \] \end{description} \end{example}

\begin{example}\label{ex:2} Symbolic Cancellation:
 Consider the following {\DAE} {\cite{Taihei19}}: \[ \left\{ \begin{array}{rcl} \dot{x}_{1}+\dot{x}_{2}+{x}_{3}&= 0\\ \dot{x}_{1}+\dot{x}_{2} &=0\\ {x}_{2}+\dot{x}_{3}&= 0 \end{array} \right . \]
Symbolic cancellation occurs when the determinant of the Jacobian matrix of the {\DAE} is identically zero.
This case can be regularized by a number of methods: a combinatorial relaxation method, a linear combination (LC) method, and an expression substitution (ES) method {\cite{Tan17}}, a substitution method and an augmentation method {\cite{Taihei19}}.
\end{example}

%Symbolic cancellation is usually represented by structurally and identically singularity, and the determinant of its Jacobian matrix is identically zero. Furthermore, some cases of {\DAE} with symbolic cancellation, whose bipartite graph have perfect match, can be modified. For instance, structurally singular failure is due to not doing simplifications, this failure can be fixed by cancellation{\cite{Nedialkov05}}. For identically singularity, whose Jacobian matrix is obviously equal to zero, can be well regularized by combinatorial relaxation, linear combination (LC) method, expression substitution (ES) method{\cite{Tan17}}, substitution method and augmentation method{\cite{Taihei19}}.

Unfortunately, there is little research on failure caused by numerical degeneration. This could happen for a parametric {\DAE} model with a non-zero determinant, where parameters take some specific values, and the determinant equals zero after substituting any initial value on a component defined by the constraints.

\begin{figure}[htpb]
  \centering
  \includegraphics[height=5cm,width=8cm]{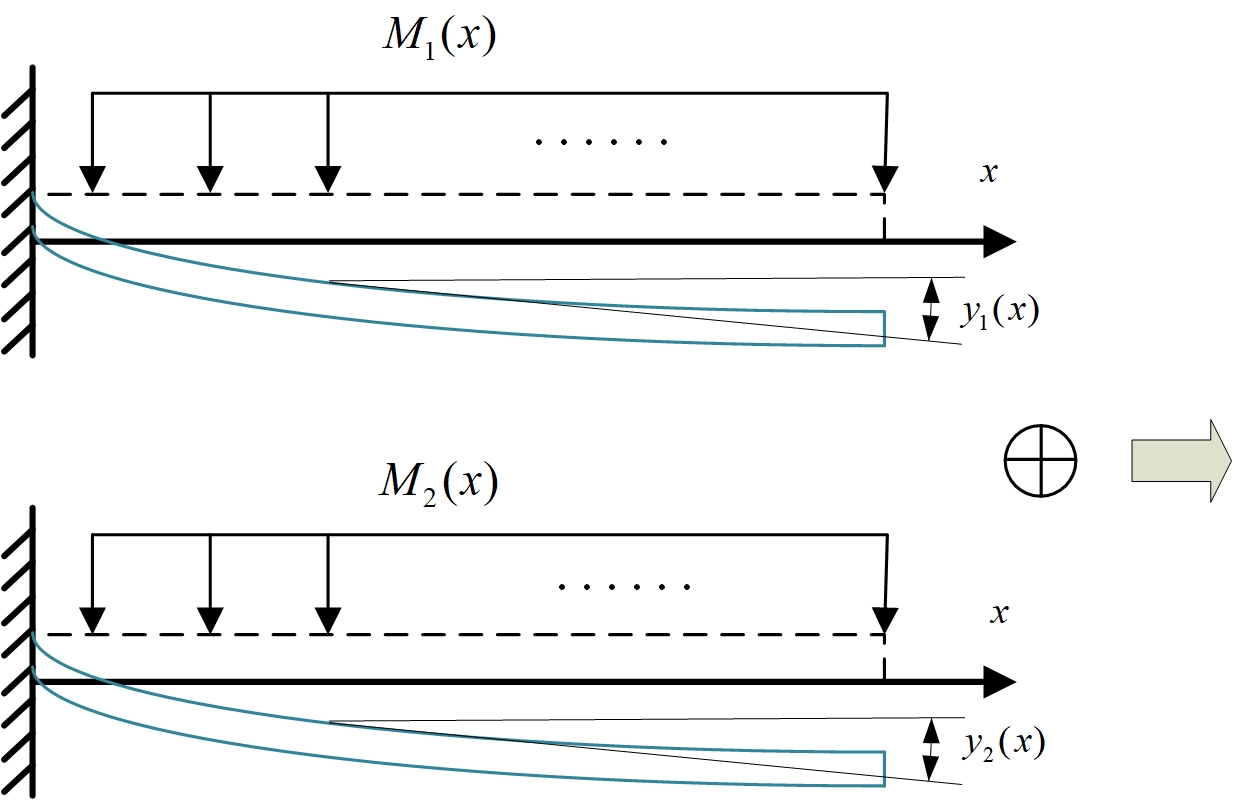}
   \caption{The Superposition Deformation of Beams under Collinear Bending Moments }
  \label{fig:bending}
\end{figure}

\begin{example}\label{ex:3}Numerical Degeneration:

In the bending deformation of a beam described in {\cite{Rans}}, the deformation of any moment acting at a point $x$ satisfies ${\frac {\rm d^{2}}{{\rm d}x^{2}}}y(x)=\frac {M(x)}{E(x)I(x)}$. When two collinear bending moments act on the beam at the same time, the deformation is the superposition of the effects of these moments. Suppose that the elastic deformation energy of the bending moment at each point on the beam is linear in $\lambda$, then the deformation variables $y_1(x)$, $y_2(x)$ satisfy:
\[\left\{\begin{array}{rcc}
{\frac {\rm d^{2}}{{\rm d}x^{2}}}y_{1}(x)+{\frac {\rm d^{2}}{{\rm d}x^{2}}}y_{2}(x)+\frac{1}{5}\cdot(1-\sin(x))+y_{1}(x)&=&0\\
\lambda\cdot y_{1}(x)^{2}-y_{2}(x)^{2}&=&0
\end{array}\right.\]

In this example, the determinant of the Jacobian matrix is $-2(y_{2}+\lambda y_{1})$. When the parameter $\lambda$ is equal to $1$, the constraint becomes $y_{1}^{2}-y_{2}^{2}=(y_{1}+y_{2})(y_{1}-y_{2})=0$. In the view of physics, this means that the elastic deformation energy of each point on the beam is the same. Obviously, two consistent initial points can be selected from two different components, respectively. If the point is on the component $y_1-y_2=0$, then Pryce's structural method works well. But for any initial point on the component $y_1+y_2=0$, we always encounter a singular Jacobian.
Note that, this ``bad" initial value usually can only be obtained approximately, since there is no explicit expression for the roots of general degree $5$ polynomials.
Furthermore, such choices of the parameter values if exist, often satisfy a nonlinear system and are only obtained numerically. Therefore, we call this case \textbf{numerical degeneration}.
\end{example}

Such a degeneration is of potentially great importance in designing control parameters in architecture, aviation and biochemistry.
So a key question deserving further study both theoretically and computationally which is
addressed in our paper is how to identify and solve such degenerated systems.

%\subsection{Main Contributions}\label{ssec:contribute}

%In order to deal with degradation problem, we propose two essential methods in this paper. One is to provide a witness points method, which can detect potentially initial points on all components of polynomially {\DAE}.  The other is to propose an IRE method which can restore full rank Jacobian matrix of degeneration without algebraic elimination. Finally, an improved structure method based on IRE is introduced to solve the {\DAE} whether it is degraded or not.
%\subsection{Motivation}\label{ssec:motiv}

\vspace{0.3cm}

In summary, we aim to solve the following problems in this paper:
\begin{itemize}
  \item  To detect at least one initial point on each real component of a polynomially  nonlinear {\DAE}.
  \item  To propose a global structural method to restore full rank Jacobian matrices without algebraic elimination.
\end{itemize}

\section{Preliminaries}\label{s:pre}
\sloppy{}

In what follows we will use algorithmic aspects of the formal (Jet) theory of differential equations {\cite{Reid01,Seiler2010}}.  Jet theory enables two inter-related views of differential equations to be algorithmically and correctly manipulated.  One view is in terms of the maps as algebraic (non-differential) equations, and the other view is in terms of local solutions of the differential equations.

Let $\mathbb{I}$ be a nonempty sub-interval of $\R$. Let $ t \in \mathbb{I} \subset \R$
and suppose
$\bm{x},\bm{x}^{(1)},...,\bm{x}^{(\ell)}$ are vectors in $\R^n$, where $\ell$ is a fixed positive integer.
Here we consider maps
$\bm{F}: \mathbb{I}\times \R^{\ell n+n}\rightarrow \R^m$ which are polynomially nonlinear
in $\bm{x},\bm{x}^{(1)},...,\bm{x}^{(\ell)}$ and real analytic in $t$, where possibly $m \not = n$.

An algebraic solution of this system is a point $(t,\bm{x},\bm{x}^{(1)},...,\bm{x}^{(\ell)}) \in \mathbb{I} \times \R^{\ell n+n}$
that satisfies $\bm{F}(t,\bm{x},\bm{x}^{(1)},...,\bm{x}^{(\ell)}) = 0$.
A differential solution $\bm{x}(t)$ is a solution for $t$ in some subinterval of $\mathbb{I}$ that satisfies the
differential equations $\bm{F}(t,\bm{x}(t),\bm{x}^{(1)}(t),...,\bm{x}^{(\ell)}(t)) = 0$ where
now $\bm{x}^{(k)}(t)$ denotes the $k$-order derivative of $\bm{x}(t)$.
Sometimes we will consider the system as an algebraic system and sometimes a differential system,
which should be clear from the context.

In particular, we consider systems which are not solved for their highest derivatives, and regard such a system as a
{\DAE}.
The differential-algebraic systems we consider have Jet form 
\begin{equation}
\bm{F}(t,\bm{x},\bm{x}^{(1)},...,\bm{x}^{(\ell)}) = 0
\end{equation}

Let $\ctotDer$ be the formal total derivative operator with respect to independent variable $t$:
\begin{equation}\label{eq:Dt}
\ctotDer = \frac{\partial}{\partial t} +  \sum_{k=0}^{\infty} \bm{x}^{(k+1)} \frac{\partial }{\partial \bm{x}^{(k)}}
\end{equation}

Regarding $\bm{F}$ in its algebraic (jet) form a single \textbf{prolongation} of $\bm{F}$ is the differentiation of each $F_i$ with respect to $t$, in which $F_i$ is the $i$-th equation of $\bm{F}$, and it is denoted by
\begin{equation}\label{eq:F_expan}
\bm{F}^{(1)} = \ctotDer \bm{F} \cup \ctotDer^0 \bm{F} =  \{\ctotDer F_1,..., \ctotDer F_n \} \cup \bm{F}
\end{equation}

It easily follows that the prolongation of $\bm{F}$ is a linear system with respect to the ``new" dependent variable $\bm{x}^{(\ell+1)}$.
Thus, we can rewrite
\begin{equation}\label{eq:linearStructure}
\ctotDer \bm{F} = \Symbol(t,\bm{x},\bm{x}^{(1)},...,\bm{x}^{(\ell)}) \cdot \bm{x}^{(\ell+1)} + G(t,\bm{x},\bm{x}^{(1)},...,\bm{x}^{(\ell)})
\end{equation}
where $\Symbol$ is an $n \times n$ matrix called the ``\textbf{symbol matrix}" and $\bm{x}^{(\ell+1)}$ is a column vector and $G$ contains all the remaining terms. Note that $\Symbol$ is also the Jacobian matrix of $\bm{F}$ with respect to its highest order derivative $\bm{x}^{(\ell+1)}$.

If we specify the prolongation order for $F_i$ to be $c_i$, then $c_{i}\geq 0$, for $i = 1, \dots, n$. For notational brevity, we will write $ (c_1,...,c_n) = \bm{c} \geq 0$. Then the prolongation of $\bm{F}$  up to the order  $\bm{c}$ is
\begin{equation}\label{eq:DefFc}
\bm{F}^{(\bm{c})}= \{F_1, \ctotDer F_1,..., \ctotDer^{c_1}F_1\}\cup \cdots \cup  \{F_n, \ctotDer F_n,..., \ctotDer^{c_n}F_n\} =\ctotDer^{\bm{c}}\bm{F}
\end{equation}

If $\bm{c}>\bm{0}$, then $\bm{F}^{(\bm{c})}$ also has linear structure similar to (\ref{eq:linearStructure}). The number of equations of $\bm{F}^{(\bm{c})}$ is $n + \sum_{i=1}^{n} c_i$.

\subsection{Structural Prolongation Methods for DAE}

In \cite{Pryce01}, Pryce reinterpreted Pantelides' algorithm \cite{Pantelides88} as an assignment problem that reveals structural information about {\DAE}s. This structural method for square {\DAE}s is a special
case with  roots in the work of Jacobi
\cite{Ollivier09} and yields a local existence and uniqueness result.
The most important feature of Pryce's method is that it finds all the local constraints for
a large class of square {\DAE}s only using prolongation.
A generalization of this construction to partial differential-algebraic equations can be found in \cite{WRI09}.

Suppose that the $k$-th order of derivative of $x_j$ occurs in $F_i$, then the partial derivative $\partial F_i/ \partial x_j^{(k)}$ is not identically zero. The \textit{leading derivative} of an equation or a system $\bm{F}$ with respect to
$x_j$ is denoted by $\LD(F,x_j)$ and is the highest order of derivative such that some $F_i\in \bm{F}$ depends on $x_j^{(k)}$ for some $k\in \mathds{Z^{+}}$.
We define the \textit{leading class} derivatives of a system $\bm{F}$ by
$$ \LCD(\bm{F}):=\{\LD(\bm{F},x_j):  1 \leq j \leq n\} \; $$

Then we obtain an $n \times n$ matrix
$\bm \sigma=(\sigma_{i,j})_{1 \leq i \leq n,1 \leq j \leq n}$ which is called the \textit{signature matrix}
 of $\bm{F}$ by Pryce \cite{Pryce01}:
\begin{equation}\label{signature}
    (\sigma_{i,j})(\bm{F}):= \left\{%
\begin{array}{ll}
    \hbox{ the order of $\LD(F_i,x_j)$;} \\
    -\infty, \;\; \hbox{otherwise.} \\
\end{array}%
\right.
\end{equation}

Suppose that the highest order derivative of $x_j$ appearing in $\bm{F}^{(\bm{c})}$, defined in Equation (\ref{eq:DefFc}), is
$d_j$. From the definition of $\sigma_{i,j}$, clearly $d_j$ is
the largest of $c_i + \sigma_{ij}$ for $i=1,...,n$, which implies that
\begin{equation}\label{dc}
    d_j - c_i \geq \sigma_{ij}, \hbox{~ for all~} i,j.
\end{equation}

Obviously, there are at most $n + \sum d_j$ derivatives
and $n + \sum c_i$ equations in $\bm{F}^{(\bm{c})}$. %We can embed $\Pro^c_t R$ into an $m + \sum d_j$
%dimensional space. %If each equation drops the dimension of the
%zero set of $\Pro^c_t R$ by one, then
The dimension of $\bm{F}^{(\bm{c})}$ usually is $\sum d_j - \sum c_i$. Roughly
speaking, finding all the constraints is equivalent to minimizing
the dimension of $\bm{F}^{(\bm{c})}$. This can be formulated as an
integer linear programming ({\ILP}) problem in the variables $\bm{c} =
(c_1, ..., c_n )$ and $\bm{d} = (d_1, ..., d_n )$:
\begin{equation}\label{LPP}
   \delta(\bm{F}) \left\|%
\begin{array}{l}
    \hbox{Minimize~~} \delta = \sum d_j - \sum c_i,  \\
    \hbox{~~~~where~~} d_j -c_i \geq \sigma_{ij},  \\
    ~~~~~~~~~~~~~~c_i \geq 0 \\
\end{array}%
\right.
\end{equation}
Let $\delta(\bm{F})$ be the optimal value of the problem (\ref{LPP}).

\begin{remark}\label{rem:dae}

When a {\DAE} has no redundant equations, the optimal value $\delta(\bm{F})$ can be regarded as degree of freedom (DOF)
measure for the {\DAE}, and it also equals the number of variables of $\bm{F}$ minus the number of equations of $\bm{F}$.  In this paper, we usually only consider cases without redundant equations in theoretical derivation.  Some cases with redundant equations will be addressed in Section \ref{ssec:sqare}.  We will also show the computational performance of our approach in our experiments.
\end{remark}

After we obtain the number of prolongation steps $c_i$ for each
equation $F_i$ by applying an {\ILP} solver to Equation (\ref{LPP}), we can construct the partially prolonged system $\bm{F}^{(\bm{c})}$ using $\bm{c}$.
We note that $\bm{F}^{(\bm{c})}$ has a favorable block
triangular structure enabling us to compute consistent initial values
more efficiently.

Without loss of generality, we
assume $c_1 \geq c_2 \geq \cdots \geq c_n$, and let $k_c = c_1$,
which is closely related to the \textit{index} of system $\bm{F}$ (see
\cite{Pryce01}). The $r$-th order derivative of $F_j$
with respect to $t$ is denoted by $F_j^{(r)}$. Then we can
partition $\bm{F}^{(\bm{c})}$ into $k_c + 1$ parts (see Table $1$),
for $0\leq p \in\mathds{Z} \leq k_c$ given by
\begin{equation}\label{eq:B_i}
    \bm{B}_p:= \{ F_j^{(p+c_j - k_c )} : 1\leq j \leq n, p+c_j - k_c
    \geq 0 \} .
\end{equation}

\begin{table}\label{table:triB}
\begin{center}
\begin{tabular}{|c|c|c|c|c|}
  \hline
  $\bm{B}_0$ & $\bm{B}_1$ & $\cdots$ & $\bm{B}_{k_c-1}$ & $\bm{B}_{k_c}$ \\
  \hline
  $F^{(0)}_1$ & $F^{(1)}_1$  & $\cdots$      & $F^{(c_1-1)}_1$  & $F^{(c_1)}_1$  \\
            & $F^{(0)}_2$  & $\cdots$      & $F^{(c_2-1)}_2$  & $F^{(c_2)}_2$ \\
            &            & $\vdots$      & $\vdots$         & $\vdots$  \\
            &            & $F^{(0)}_n$   &  $\cdots$        & $F^{(c_n)}_n$ \\
  \hline
\end{tabular}
\caption {The triangular block structure of $\bm{F}^{(\bm{c})}$ for the case of $c_{p}=c_{p+1}+1$; For $0\leq p < k_c$, $\bm{B}_i$ has fewer jet variables than $\bm{B}_{p+1}$.}
\end{center}
\end{table}

Here, we call $\bm{B}_{k_c}$ the \textbf{top block} of $\bm{F}^{(\bm{c})}$ and $ \bm{F}^{(\bm{c}-1)} = \{\bm{B}_0,...,\bm{B}_{k_c-1}\}$ the \textbf{constraints}.

Similarly, let $k_d = \max(d_j)$ and we can partition all the variables into $k_d+1$ parts:
\begin{equation}\label{eq:U_i}
    \bm{X}^{(q)}:= \{ x_j^{(q+d_j - k_d )} : 1\leq j \leq n, q+ d_j - k_d  \geq 0 \} .
\end{equation}

For each $\bm{B}_i, 0\leq i \leq k_c$, we define the Jacobian Matrix
\begin{equation}\label{Jac}
    \bm{\Jac}_{i} :=
\left( \frac{\partial \bm{B}_{i}}{\partial \bm{X}^{(i+k_d-k_c)}} \right).
\end{equation}

So $\bm{\Jac}_{k_c}$ is the Jacobian Matrix of the top block in the table, and it is a square matrix.

\begin{prop}\label{prop:fullrank}
Let $ \{\bm{\Jac}_i\}$ be the set of Jacobian
matrices of $\{\bm{B}_i \}$. For any $0 \leq i < j \leq k_c$, $\bm{\Jac}_{i}$
is a sub-matrix of $\bm{\Jac}_{j}$. Moreover, if $\bm{\Jac}_{k_c}$ has full
rank, then any $\bm{\Jac}_{i}$ also has full rank.
\end{prop}

See \cite{WRI09} for the proof.

Suppose $(t^*,\bm{X}^*)$ is a point satisfying the  constraints $\{\bm{B}_0,...,\bm{B}_{k_c-1}\}$ and
$\bm{\Jac}_{k_c}$ has full rank at this point. Then Pryce's structural method has successfully finished the index reduction.
However, it fails if $\bm{\Jac}_{k_c}$ is still singular, {\ie} $\bm{\Jac}_{k_c}$ is degenerated.

Obviously, the definition of optimal value $\delta(\bm{F})$ is limited to square systems, and we need to extend the definition for non-square systems  $\bm{F}^{(\bm{c})}$.

\begin{define}\label{define_delta1}
Let a {\DAE} $\bm{F}$  consist of two blocks $\bm{A}$ and $\bm{B}$,  where $\bm{F}$ contains $p$ equations and $n$  dependent variables $p\geq n$, and the signature matrix of $\bm{A}$ be an $n\times n$ square matrix. So $\bm{B}$ contains the remaining $(p-n)$ equations. Let $\delta(\bm{A})$ be the optimal value of the {\ILP} of $\bm{A}$'s signature matrix. We define  $\delta(\bm{F}) = \delta(\bm{A}) - \#eqns(\bm{B})$, where $\#eqns(\bm{B})$ is the number of equations in $\bm{B}$.  Meanwhile, $\delta(\bm{F})$  also equals the $DOF$ \cite{Tan17} of $\bm{F}$, which equals the number of dependent variables minus the number of equations in the prolongation of $\bm{F}$.
\end{define}

In the case of a square signature matrix of a {\DAE} $\bm{F}$, we have $\#eqns=0$, and the extended definition of $\delta(\bm F)$ is equivalent to the original definition.

\begin{prop}\label{prop:extend_delta}
Let $(\bm{c},\bm{d})$ be the optimal solution of Problem (\ref{LPP}) for a given  DAE $\bm{F}$. Then $\delta (\bm{F})=\delta (\bm{F}^{(\bm{c})})=\sum d_j - \sum c_i$.

\end{prop}

\begin{proof}
For a prolonged {\DAE} system $\bm{F}^{(\bm{c})}=\{\bm{B}_{k_c},\bm{F}^{(\bm{c}-1)}\}$, the signature matrix of the top block $\bm{B}_{k_c}$ is square.

 We construct a pair $(\hat{\bm{c}},\hat{\bm{d})}$, for $i=1,\cdots,n$ and $ j=1,\cdots,n$, $\hat{c}_{i}=0$ and $\hat{d}_{j}= d_{j}$. Since $(\bm{c},\bm{d})$ is the optimal solution for $\bm{F}$, and $\bm{B}_{k_c}$ is the top block of $\bm{F}^{(\bm{c})}$, it follows that $(\hat{\bm{c}},\hat{\bm{d})}$ is the optimal solution of $\bm{B}_{k_c}$, $\delta(\bm{B}_{k_c})= \sum d_j$.

 By Definition \ref{define_delta1}, and using $\#eqns(\bm{F}^{(\bm{c}-1)})=\sum c_i$, we obtain $$\delta (\bm{F}^{(\bm{c})})= \delta(\bm{B}_{k_c}) -\#eqns(\bm{F}^{(\bm{c}-1)})=\sum d_j - \sum c_i =\delta (\bm{F}) \,.$$
\foorp
\end{proof}

\subsection{Framework for Improved Structural Methods}\label{ssec:improved}

Many improved structural methods have been proposed to regularize the Jacobian matrices of {\DAE}s. See {\cite{Gerdts11,Iwata03,Murota95,Campbell93,Wu13}} for methods for linear {\DAE}.
For non-linear {\DAE}s, improved structural methods are based on a combinatorial relaxation framework \cite{Taihei19} with the following steps:
\begin{description}
  \item[Phase $1$.] Compute the solution ($\bm{c}$,$\bm{d}$) of {\ILP} problem $\delta(\bm{F})$. If there is no solution, the {\DAE} do not admit perfect matching, and the algorithm ends with failure.
  \item[Phase $2$.] Determine whether $\bm{\Jac}_{k_c}$ is identically singular or not. If not, the method returns $\bm{F}^{(\bm{c})}$ and halts.
  \item[Phase $3$.] Construct an new {\DAE} $\hat{\bm{F}}$, such that its solution space in $\bm{x}$ dimension is the same as {\DAE} $\bm{F}$ and $0\leq\delta(\hat{\bm{F}})<\delta(\bm{F})$. Then go to Phase $1$.
\end{description}

\begin{remark}
The key part of an improved structural method is to exploit different regularization method for $\hat{\bm{F}}$ in Phase $3$. In this paper, our global structural differentiation method mainly focuses on this phase.
\end{remark}

Phase $2$ above is only to check for symbolic cancellation.
As pointed out in Example \ref{ex:3}, $\det \bm{\Jac}_{k_c}$ may not be identically zero, but $\det \bm{\Jac}_{k_c}=0$ at any consistent initial point of $Z(\bm{F}^{(\bm{c})})$ --- the zero set of $\bm{F}^{(\bm{c})}$. Since $\bm{F}$ is a polynomial system in $\{\bm{x},\bm{x}^{(1)},...,\bm{x}^{(\ell)} \}$,
$\bm{F}^{(\bm{c})}$ can be considered as a polynomial system in the variables $\{\bm{X}^{(0)},...,\bm{X}^{(k_d)}\}$. In the language of algebraic geometry,
it means that $\det \bm{\Jac}_{k_c} \in \sqrt[\R]{\langle \bm{F}^{(\bm{c})} \rangle }$ or equivalently $Z_{\R}(\bm{F}^{(\bm{c})}) \subseteq Z_{\R}(\bm{\Jac}_{k_c})$.

In the rest of the paper, we usually suppress the subscript in $\bm{\Jac}_{k_c}$ so it becomes $\bm{\Jac}$ unless the subscript is needed.

\begin{example}\label{ex:4}
Consider the following {\DAE} with dependent variables $x \left( t \right)$ and $y \left( t \right)$:
\begin{equation}
\bm{F} = \{2\,y {\frac {{\rm d^{2}}x}{{\rm d}t^{2}}}  - x
  {\frac {{\rm d^{2}}y }{{\rm d}t^{2}}} +2x
 \left({\frac {{\rm d}x}{{\rm d}t}} \right)^{2} - {\frac {{\rm d} x }{{\rm d}t}} +\sin \left( t \right) ,y  -
  x^2   \}.
\end{equation}
Applying the structural method yields $\bm{c}=(0,2)$ and $\bm{d}=(2,2)$. Then
$$\bm{F}^{(\bm{c})}= [\{2y x_{tt}-x y_{tt} +2x {x_t}^{2}-x_t+\sin(t), y_{tt}-2x_t^{2}-2xx_{tt}\}, \{-2x x_t+y_t\}, \{-x^2+y\}]$$ and the Jacobian matrix of the top block is
  $\bm{\Jac} =
      \left(
        \begin{array}{cc}
          2y & -x \\
          -2x & 1 \\
        \end{array}
      \right)$.

Although the determinant of the Jacobian $2y-2x^2$ is not identically zero, it must equal zero at any initial point, since
the determinant belongs to the polynomial ideal generated by the constraints, i.e.
$2y-2x^2 \in \langle -x^2+y  \rangle $.
\end{example}

Checking if a polynomial belongs to an ideal can be done by a standard ideal membership test using a Gr\"{o}bner basis of the ideal.
In general, it is challenging to compute the Jacobian determinant and the associated Gr\"{o}bner basis if the system is quite large.
See the text \cite{CoxLittleOshea07} for more details about polynomial ideals, varieties and Gr\"{o}bner bases.
Algorithmic algebraic geometry exploits Gr\"{o}bner bases and related techniques to compute features of solutions
of general polynomial systems with exact (e.g. rational) coefficients.
Numerical versions of these algorithms, where exact numbers are replaced with approximate numbers have largely
been expensive and often unstable.

In this paper, we propose a numerical approach based on real algebraic geometry to detect such degenerated cases
without using determinants or Gr\"{o}bner bases.
It exploits a new generation of algorithms using a fundamentally different and  more thoroughly
numerical approach, centered around the concept of witness points on solution components and is discussed in the next section.

\subsection{Numerical Real Algebraic Geometry}

%Numerical algebraic geometry~\cite{SVW05,Hauenstein2017} was pioneered by Sommese, Wampler, Verschelde and others (see \cite{BHSW13,SommeseWampler05} for references and background).   They first considered the more accessible characterization of complex solution components of each possible dimension, by slicing the solution set with appropriate random planes, that intersected the solution components in complex points called \textit{witness points}.  The complex points are computed by Homotopy continuation solvers.  For example, a one-dimensional circle, $x^2 +y^2 - 1 = 0$ in $\mathbb{C}^2$ is intersected by a random line in two such witness points, but this method obviously fails for $(x,y) \in \mathbb{R}^2$ since a real line may miss the circle.

%Instead the method in \cite{WuReid13,WRF17} yields real witness points as critical points of the distance from a random hyperplane to the real variety. The reader can easily see this yield two real witness points for the circle example.

%An alternative numerical approach where the witness points are critical points of the distance from a random point to the real variety has been developed in \cite{Hauenstein12}.
%The works \cite{Hauenstein12,WuReid13,WRF17} use Lagrange multipliers to set up the critical point problem in the following way.

%For a random point $\randpoint \in \R^n$, there is at least one  point on each connected component of $V_{\R}(f)$ with minimal distance to $\randpoint$ satisfying the following problem:

Numerical algebraic geometry~\cite{SVW05,Hauenstein2017} was pioneered by Sommese, Wampler, Verschelde and others (see \cite{BHSW13,SommeseWampler05} for references and background).  The approach is built on \textit{witness points} which arise by slicing the complex variety with appropriate random planes of complementary dimension. These complex witness points can be efficiently computed by homotopy continuation solvers \cite{Lee2008}, and are theoretically guaranteed to compute at least one such point on each solution component.

For the real case, the methods in \cite{WuReid13,WRF17} yield real witness points as critical points of the distance from a random hyperplane to the real variety.
Alternatively,  the real witness points can be considered as critical points of the distance from a random point to the real variety \cite{Hauenstein12}.

%More precisely, for a random point $\randpoint \in \R^n$, there is at least one point on each connected component of $V_{\R}(f)$ with minimal distance to $\randpoint$ satisfying the following problem:
%
%\begin{eqnarray}\label{eq:opt2}
% \min \;  \sum_{i=1}^n (x_i-\randpoint_i)^2 \\
% s.t. \hspace{1cm} f(x) = 0   \nonumber.
%\end{eqnarray}
%
%This optimization problem can be formulated as a square system by using Lagrange multipliers. Its real solutions obtained by
%the Homotopy continuation method are called real witness points of $V_{\R}(f)$, here $V_{\R}(f)=\{x\in \mathbb{R}^{n} : f(x)=0\} \subseteq Z_{\R}(f)$.
%In the sense of {\DAE}, these real solutions of algebraic constraints are potentially initial points for every component of a non-linear {\DAE}.
%
%\begin{define}
%For a polynomial system $f$, let $\Sing$ be the set of singular points of $V_{\R}(f)$.
%If a finite set $W\subseteq \R^n$ contains at least one uniformly random point on each connected component of $V_{\R}(f) \backslash \Sing$. Then this set is called real witness set of $V_{\R}(f)$ and these points are called real witness points.
%\end{define}
More precisely, to solve a polynomial system $ \bm f = \{f_1,...,f_k\} \subset \R[x_1,...,x_n]$, we first choose a random point $\randpoint \in \R^n$, then there is at least one point on each connected component of $V_{\R}(\bm f)$ with minimal distance to $\randpoint$ satisfying the following problem:

\begin{eqnarray}\label{eq:opt2}
 \min \;  \sum_{i=1}^n (x_i-\randpoint_i)^2/2 \\
 s.t. \hspace{1cm} \bm f(\bm x) = 0   \nonumber.
\end{eqnarray}

This optimization problem can be formulated as a square system by using Lagrange multipliers, i.e.
\begin{equation}
\bm g = \{\bm f,  \sum_{i=1}^k \lambda_i \nabla f_i + x_i - \randpoint_i \} = 0
\end{equation}

When $\bm f$ satisfies the regularity assumptions in \cite{WuReid13}, all the real solutions of $\bm g = \bm 0$ can be obtained by
the homotopy continuation method. These points are called real witness points of $V_{\R}(\bm f)$, where $V_{\R}(\bm f)=\{\bm x\in \mathbb{R}^{n} : \bm f(\bm x)=0\}$.
These real solutions of the constraint equations provide  initial points for every component of a non-linear {\DAE}.

\begin{define}
For a polynomial system $\bm f$, let $\Sing$ be the set of singular points of $V_{\R}(\bm f)$.
If a finite set $W\subseteq \R^n$ contains at least one point on each connected component of $V_{\R}(\bm f) \backslash \Sing$. Then this set is called the \textbf{real witness set} of $V_{\R}(\bm f)$ and these points are called \textbf{real witness points}.
\end{define}

However, if $\bm f$ does not satisfy the regularity assumptions due to high multiplicity or a non-real radical ideal, then we apply a critical point approach \cite{WuChenReid17} based on a penalty factor.

For example if $f = (x^2+y^2-1)^2$, then we choose a random point $\randpoint = (a,b)$ and a large penalty factor $\beta$.
The corresponding regular system is
\begin{equation}\label{eq:cp1}
\left(
  \begin{array}{c}
    x \\
    y \\
  \end{array}
\right) + \beta \cdot f \cdot \left(
                         \begin{array}{c}
                           4x(x^2+y^2-1) \\
                           4y(x^2+y^2-1) \\
                         \end{array}
                       \right) =  \left(
                         \begin{array}{c}
                           a \\
                           b \\
                         \end{array}
                       \right)\\
\end{equation}

If the random point is $(a,b)=(1,1)$ and a large penalty factor $\beta=100000$ is chosen, then $4$ approximate solutions can be calculated by the homotopy continuation method. Here, two of these solutions are far away from components, which can be verified by substituting them into the equation $f$. The verified solutions $[x = 0.985220, y = 0.172402],[x = -0.652031, y = -0.758442]$ are close the real variety which is the unit circle.
For more details, please see \cite{WuChenReid17}.

\section{Detecting Degeneration by Witness Points}\label{sec:det}
\sloppy{}

To build a solid foundation of our theory, we need some results from the theory of real analytic functions of several variables \cite{KrantzParks02}.

\begin{define}
A function $f$, with domain an open subset $U \subset \R^n$ and range
$\R$, is called real analytic on $U$, if for each $\bm p \in U$ the function
$f$ can be represented by a convergent power series in some neighbourhood of $\bm p$.
\end{define}

\begin{prop}[Proposition 2.2.8 of \cite{KrantzParks02}]\label{prop:composition}
If $f_1, ..., f_m$ are real analytic in some neighbourhood of
the point $\bm p\in \R^n$  and $g$ is real analytic in some neighbourhood of the point
$(f_1(\bm p),  ... , f_m(\bm p)) \in \R^m$, then the composition of functions $g(f_1(\bm x), ...,f_m(\bm x))$ is real analytic
in a neighborhood of $\bm p$.
\end{prop}

\begin{theorem}[Real Analytic Implicit Function Theorem \cite{KrantzParks02}]\label{thm:IFT}
Given a set of equations
$f_i(x_1,...,x_m;y_1,...,y_n) =0, \;\; i=1,2,...,n$, where each $f_i$ is real analytic, suppose that
$(\bm p,\bm q)= (p_1,...,p_m; q_1,...,q_n)$  is a solution with nonsingular Jacobian
$\frac{\partial \bm f}{\partial \bm y} (\bm p,\bm q)$.

Then there exists a neighborhood $U \subset \R^m$ of $\bm p$ and a set of real analytic functions $ \phi_j : U \rightarrow  \R, j = 1, 2, ... , n$, such that $\phi_j(\bm p) = q_j, j = 1,2,...,n$, and
$$f_i(\bm x; \phi_1(\bm x),...,\phi_n(\bm x)) =0, i = l,2,...,n, $$
hold for $\bm x \in U$.
\end{theorem}

\begin{theorem}[Identity Theorem for Real Analytic Functions]\label{thm:IT}
Given two real analytic functions $f$ and $g$ on an open and connected set $U\subset \R^n$, if $f = g$ on a nonempty open subset $S  \subseteq U$, then $f = g$ on the whole set $U$.
\end{theorem}
\proof \hskip4pt Define a set where $f$ and $g$ have the same power series:
$$ D = \left\{\bm c \in U: \frac{\partial ^{|\bm \mu|}}{\partial \bm x^{\bm \mu}} f(\bm c) =  \frac{\partial ^{|\bm \mu|}}{\partial \bm x^{\bm \mu}} g(\bm c) \mbox{ for all } \bm \mu\right\} $$ where $\bm \mu$ is a multi-index in $(\mathbb{Z}^{+})^n$.
Firstly, since $f = g$ on a nonempty open subset $S$, we have $D \supseteq S$ and consequently $D\neq \emptyset$.
Secondly, $f$ and $g$ are real analytic on $U$, then for any $\bm c\in D$ the power series have a non-zero radius of convergence. It implies that
$D$ is open.
Meanwhile,
$$D = \bigcap_{\mu}  \left\{\bm c \in U:  \frac{\partial ^{|\bm \mu|}}{\partial \bm x^{\mu}} f(\bm c) =  \frac{\partial ^{|\bm \mu|}}{\partial \bm x^{\bm \mu}} g(\bm c) \right\}$$
which is an intersection of closed sets, so it is closed.

Therefore, by the connectedness of $U$, $D$ must be equal to the whole set $U$.  \foorp \\

\begin{lemma}\label{lem:measure0}
Let $C$ be a connected real analytic manifold in $\R^{m+n}$ of dimension $m$ and let $f$ be a real analytic function on $\R^{m+n}$.
Then the intersection $C\cap Z_{\R}(f)$ is equal to $C$ or has measure zero over $C$.
\end{lemma}
\proof \hskip4pt
Since $C$ is a smooth manifold described by an analytic system implicitly,
by the Implicit Function Theorem \ref{thm:IFT}, locally $C$ can be parameterized by $m$ free coordinates.
To be rigorous, we need an atlas over $C$ which is a collection of charts depending on the free coordinates.
Without loss of generality, we assume $y_j = \phi_j(x_1,...,x_m),  j=1,...,n$ where $\phi_j$ is real analytic.

Suppose the intersection $C\cap Z_{\R}(f)$ has non-zero measure. Then there is a nonempty open subset $S$ of $C$ where
$g(\bm x) = f(\bm x,\phi_1(\bm x),...,\phi_n(\bm x)) = 0$. By Proposition \ref{prop:composition}, $ g$ is real analytic.
Due to the Identity Theorem \ref{thm:IT}, $g = 0$ on the whole component, and thus $C \subseteq Z_{\R}(f)$.
\foorp \\

The real zero set of an analytic system $\bm f$ is denoted by $Z_{\R}(\bm f)$, whereas the real zero set of a polynomial system $\bm f$ is often denoted by
 $V_{\R}(\bm f)$.

\begin{theorem}\label{thm:vanish}
Let $\bm g$ be a polynomial system and $W$ be a real witness set of $V_{\R}(\bm g)$. If another polynomial system $\bm f(\bm p) = 0$ for any $\bm p\in W$, then $V_{\R}(\bm g) \subseteq V_{\R}(\bm f)$ with probability one.
\end{theorem}
\proof \hskip4pt
First $V_{\R}(\bm g) \backslash \Sing$ consists of finitely many smooth connected components and isolated real points.
By Lemma \ref{lem:measure0}, for each smooth connected component $C$, the intersect $C\cap  V_{\R}(\bm f)$ has measure zero over
$C$, unless $C \subseteq V_{\R}(\bm f)$. Since $\bm p$ is a random point on $C$, it belongs to a measure zero set with probability zero. For the isolated points, $\bm f(\bm p) =\bm 0 \Rightarrow \bm p \in V_{\R}(f)$. Therefore, $V_{\R}(\bm g) \backslash \Sing \subseteq V_{\R}(\bm f)$ with probability one.

Since $V_{\R}(\bm f)$ is a closed set, the closure of $V_{\R}(\bm g) \backslash \Sing$, which is $V_{\R}(\bm g)$, must be contained in $V_{\R}(\bm f)$.
\foorp \\

This theorem actually gives a probabilistic method to check if  $\bm f \in \sqrt[\R]{\langle \bm g \rangle }$ without using Gr\"{o}bner bases.

\begin{remark}
After we obtain a witness set $W$ of the constraints, it is unnecessary to compute the determinant of $\bm{\Jac}$ during the detection of degenerated cases by Theorem \ref{thm:vanish}.
We can simply substitute a real witness point into the Jacobian matrix and compute its smallest singular value by numerical methods. If the smallest singular values at all witness points are close to zero, then the Jacobian is degenerated with probability close to one. If some of them are almost zero and the remaining singular values are non-zero, then the determinant vanishes on some components of the constraints, and further work is needed for the {\DAE} on these components.
\end{remark}

\section{Index Reduction by Embedding for Degenerated Systems}\label{sec:implictit method}
\sloppy{}

Consider a smooth connected component $C$ of $Z_{\R}(\bm{F}^{(\bm{c})})$ with a real point $\bm p \in \R^n$.
Suppose $\rank \bm{\Jac}(\bm p) = r < n$. Without loss of generality, we assume that the sub-matrix $\bm{\Jac}(\bm p)[1:r,1:r]$ has full rank.
In this section we will show that the rank is almost a constant over the whole component.

\begin{lemma}\label{lem:whole}
Let $C$ be a smooth connected component. If $\bm{\Jac}[1:r,1:r]$ has full rank at a random point $\bm p$ on $C$. Then it is non-singular over the whole component except some set with measure zero. Moreover, if a minor of $\bm{\Jac}$ at this point is singular, then  it is singular over the whole component with probability one.
\end{lemma}
\proof \hskip4pt
Let $f$ be the determinant of $\bm{\Jac}[1:r,1:r]$.  If $f(\bm p)=0$, then by Lemma \ref{lem:measure0} we have $C \subseteq Z_{\R}(f)$ with probability one, implying that the Jacobian is singular over the whole component.

If $f(\bm p) \neq 0$, then $C \nsubseteq Z_{\R}(f)$ and  Lemma \ref{lem:measure0} implies that $C\cap Z_{\R}(f)$ has measure zero over $C$.
\foorp \\

Jacobians with constant rank enable us to embed the zero set into a higher dimensional space.

\begin{lemma}[Constant Rank Embedding]\label{lem:proj}
Let
$$\bm f= \{f_1(x,y,z),...,f_r(x,y,z)\} \hspace{0.4cm} \mbox{and} \hspace{0.4cm} \bm g= \{g_1(x,y,z),...,g_{n-r}(x,y,z)\} $$
be two sets of analytic functions, where
$\bm x=(x_1,...,x_r)$, $\bm y=(y_1,...,y_{n-r})$ and $\bm z=(z_1,...,z_m)$. Let $C$ be a smooth connected component in $\R^{m+n}$.
If the Jacobian matrices $\frac{\partial (\bm f,\bm g)}{\partial (\bm x, \bm y)}$ and $\frac{\partial \bm f}{\partial \bm x}$ have constant rank $r$ on $C$.
Then $$Z_{\R}(\bm f,\bm g) \cap C = \pi Z_{\R}(\bm f(\bm x,\bm y,\bm z), \bm f(\bm u,\bm \xi,\bm z),\bm g(\bm u,\bm \xi,\bm z)) \cap C$$  where $\bm u=(u_1,...,u_r)$ and $\bm \xi$ is a constant vector and $\pi$ is the projection from $(\bm x,\bm y,\bm z,\bm u)$-space to $(\bm x,\bm y,\bm z)$-space.
\end{lemma}
\proof \hskip4pt
Since  $\frac{\partial \bm f}{\partial \bm x}$ has constant rank $r$ on $C$, by the Implicit Function Theorem \ref{thm:IFT} and the Identity Theorem \ref{thm:IT}, there exist a set of real analytic functions $\bm \phi = \{\phi_1,...,\phi_r\}$  such that
$\bm f(\bm \phi(\bm y,\bm z),\bm y,\bm z)=\bm 0$ for any $(\bm y,\bm z) \in \pi_{\bm{yz}}(C)$. Thus, $$\frac{\partial \bm f}{\partial \bm x} \frac{\partial \bm \phi}{\partial \bm y} + \frac{\partial \bm f}{\partial \bm y} =\bm 0\,. $$

Since $\frac{\partial (\bm f,\bm g)}{\partial (\bm x, \bm y)}$ also has constant rank $r$, $(\frac{\partial \bm \phi}{\partial \bm y}, I)^t$  is in the
null-space of $\frac{\partial (\bm f,\bm g)}{\partial (\bm x, \bm y)}$. So $$\frac{\partial \bm g}{\partial \bm x} \frac{\partial \bm \phi}{\partial \bm y} + \frac{\partial \bm g}{\partial \bm y} =\bm 0\,. $$

Let $\bm G(\bm y,\bm z) = \bm g(\bm \phi(\bm y,\bm z),\bm y,\bm z)$. We have $\frac{\partial \bm G}{\partial \bm y} = \frac{\partial \bm g}{\partial \bm x} \frac{\partial \bm \phi}{\partial \bm y} + \frac{\partial \bm g}{\partial \bm y} =\bm 0$, which implies that $\bm G(\bm y,\bm z) = \bm G(\bm \xi,\bm z)$ for any constant $\bm \xi$ on $C$.

If $\bm p=(\bm p_{\bm{x}},\bm p_{\bm{y}},\bm p_{\bm{z}}) \in Z_{\R}(\bm f,\bm g) \cap C$, then $\bm p_{\bm{x}}= \bm \phi(\bm p_{\bm{y}},\bm p_{\bm{z}})$. Let $\bm p_{\bm{u}}= \bm \phi(\bm \xi,\bm p_{\bm{z}})$ for some constant vector $\bm \xi$, and let $\hat{\bm p}= (\bm p_{\bm{x}},\bm p_{\bm{y}},\bm p_{\bm{z}},\bm p_{\bm{u}})$. It is straightforward to verify that $\bm f(\bm p_{\bm{u}},\bm \xi,\bm p_{\bm{z}})=\bm 0$ and $\bm g(\bm p_{\bm{u}},\bm \xi,\bm p_{\bm{z}})= \bm G(\bm \xi,\bm p_{\bm{z}}) =  \bm G(\bm p_{\bm{y}},\bm p_{\bm{z}}) = \bm g(\bm p_{\bm{x}},\bm p_{\bm{y}},\bm p_{\bm{z}}) = \bm 0$.  Therefore, $\hat{\bm p} \in Z_{\R}(\bm f(\bm x,\bm y,\bm z), \bm f(\bm u,\bm \xi,\bm z),\bm g(\bm u,\bm \xi,\bm z))$. Thus $\bm p \in \pi Z_{\R}(\bm f(\bm x,\bm y,\bm z), \bm f(\bm u,\bm \xi,\bm z),\bm g(\bm u,\bm \xi,\bm z)) \cap C$.

For any  $\bm p=(\bm p_{\bm{x}},\bm p_{\bm{y}},\bm p_{\bm{z}}) \in \pi Z_{\R}(\bm f(\bm x,\bm y,\bm z), \bm f(\bm u,\bm \xi,\bm z),\bm g(\bm u,\bm \xi,\bm z)) \cap C$, we have $\bm p_{\bm{x}}= \bm \phi(\bm p_{\bm{y}},\bm p_{\bm{z}})$ and ${\bm{u}}= \bm \phi(\bm \xi,\bm p_{\bm{z}})$.
Also $\bm g(\bm u,\bm \xi,\bm p_{\bm{z}})=\bm 0 \Rightarrow \bm 0=\bm G(\bm \xi,\bm p_{\bm{z}})=\bm G(\bm p_{\bm{y}},\bm p_{\bm{z}})= \bm g(\bm p_{\bm{x}},\bm p_{\bm{y}},\bm p_{\bm{z}})$.
So $\bm p\in Z_{\R}(\bm f,\bm g)\cap C$.

\foorp \\

If we have the witness set, then according to Lemma {\ref{lem:whole}}, then the rank of Jacobian matrix
of the {\DAE} on whole component can be calculated by singular value decomposition (SVD) given by Algorithm $3$.

Suppose a prolonged system $Z_{\R}(\bm{F}^{(\bm{c})})$ has constant rank i.e.
\begin{equation}\label{eq:rank}
\rank \bm{\Jac} = r =\rank \bm{\Jac}[1:r,1:r]  < n
\end{equation}
over a smooth component $C$ of $Z_{\R}(\bm{F}^{(\bm{c-1})})$. To simplify our description, we specify the full rank submatrix to be $\bm{\Jac}[1:r,1:r]$, which always can be done by proper permutations of variables and equations given by Algorithm $4$.

\begin{define}\label{define_IRE}
Index Reduction by Embedding (IRE): Suppose $(\bm{c},\bm{d})$ is the optimal solution of Problem (\ref{LPP}) for a given {\DAE} $\bm{F}$, and then prolonged {\DAE} $\bm{F}^{(\bm{c})} = \{  \bm{B}_{k_c},\bm{F}^{(\bm{c}-1)}\}$ has constant rank $\rank \bm{\Jac} = r < n$. Let $\bm{s}= (x_1^{d_1}, ..., x_r^{d_r})$, $\bm{y}= (x_{r+1}^{d_{r+1}},...,x_n^{d_n})$ and $\bm{z}= (t,\bm{X},\bm{X}^{(1)},...,\bm{X}^{(k_d-1)})$, then $\bm{B}_{k_c}=\{ \bm{f(s,y,z)},\bm{g(s,y,z)}\}$, where $\bm{f(s,y,z)}=\{F_1^{(c_1)},..., F_r^{(c_r)}\}$ and  $\bm{g(s,y,z)}=\{F_{r+1}^{(c_{r+1})},..., F_n^{(c_n)}\}$.
We can construct $\bm{G}= \{\bm{F}^{aug} ,\bm{F}^{(\bm{c}-1)}\}$ in which $\bm{F}^{aug}=\{ \bm{f(s,y,z)},\bm{f(u,\xi,z)},\bm{g(u,\xi,z)}\}$.  Then $\bm{F}^{aug}$ is constructed by the following steps:

\begin{enumerate}
  \item  Introduce $n$ new equations $\hat{\bm{F}}=\{\bm{f(u,\xi,z)},\bm{g(u,\xi,z)}\}$:
to replace $\bm{s}$ in the top block $\bm{B}_{k_c}$ by $r$ new dependent variables $\bm{u}=(u_1,...,u_r)$  respectively, and simultaneously replace $\bm{y}$ in the top block $\bm{B}_{k_c}$ by $n-r$ random constants $\bm{\xi}=(\xi_1,...,\xi_{n-r})$ respectively.
  \item Construct a new square subsystem
\begin{equation}\label{eq:newsys}
\bm{F}^{aug} = \{ \bm{f(s,y,z)},\hat{\bm{F}}\},
\end{equation}
\end{enumerate}
\end{define}

where $\bm{F}^{aug}$ has $n+r$ equations with $n+r$ leading variables $\{\bm{X}^{(k_d)}, \bm{u}\}$ and $\bm{X}^{(k_d)}=\{\bm{s},\bm{y}\}$.

Since this reduction step introduces a new variable $\bm{u}$, the corresponding lifting of the consistent initial values must be addressed.
One approach to this problem is to solve the new system $\bm{F}^{aug}$ to obtain lifted consistent initial values.
But this approach is unnecessary and expensive.
According to Definition \ref{define_IRE}, the consistent initial values of the new variables $\bm{u}$ can simply be taken as the initial values of their replaced variables $\bm{s}$. Then $\bm{\xi}$ takes the same initial value as was assigned to $\bm{y}$.

%Further more, after IRE, signature matrix of $\bm{G}$ is comprised of not only $(n+r)\times (n+r)$ square signature matrix of $\bm{F}^{aug}$ but also non-square signature matrix of additional constraint equations $\bm{F}^{(\bm{c}-1)}$. So we need to extend the definition of $\delta$ to non-square signature matrix.
%\begin{define}\label{define_delta1}
%Let $p\times q $ signature matrix $\bm{C}$ of {\DAE} consists of $\bm{A}$ and $\bm{B}$,  where $p\geq q$, $\bm{A}$ is $q\times q$ square sub-matrix and $\bm{B}$ is $(p-q)\times q$ non-square sub-matrix. To assume $\delta(\bm{A})$ be the optimal solution of $\bm{A}$, as to variants of $\bm{B}$ is subset of variants of $\bm{A}$, we define the optimal solution of $\bm{C}$ be $\delta= \delta(\bm{A}) - codim(B)$, where $codim$ is codimension and $codim(\bm{B})=\dim V -\dim \bm{B}$, $\bm{B}$ is subspace of $V$, here $\dim V$ equals the number of equations of $\bm{B}$.
%\end{define}

%Especially, for square signature matrix of $\bm{F}$, as to $codim(B)=0$, the extended definition of $\delta$ is equivalent to the original definition. For a prolonged {\DAE} system $\bm{F}^{(\bm{c})}$, the signature matrix of top block $\bm{B}_{k_c}$ is square, and $\delta(\bm{B}_{k_c})= \sum d_j$, associating with $codim(\bm{F}^{(\bm{c}-1)})=\sum d_j-n -\dim \bm{F}^{(\bm{c}-1)}$ and $\dim \bm{F}^{(\bm{c}-1)}=\sum c_i-n$, we can deduce $\delta (\bm{F})=\delta (\bm{F}^{(\bm{c})})$.

\begin{theorem}\label{thm:result}
Let $(\bm{c},\bm{d})$ be the optimal solution of Problem (\ref{LPP}) for a given {\DAE} $\bm{F}$.
Let $\bm{F}^{(\bm{c})} = \{ \bm{B}_{k_c}, \bm{F}^{(\bm{c}-1)} \}$ as defined in Equation (\ref{eq:B_i}). If $\bm{F}^{(\bm{c})}$ satisfies (\ref{eq:rank}), and $C$ is a smooth connected component in $\R^{\sum d_j+n}$, then
$$Z_{\R}(\bm{F}^{(\bm{c})})\cap C = \pi Z_{\R}(\bm{G}) \cap C$$
where $\bm{G}= \{\bm{F}^{aug},\bm{F}^{(\bm{c}-1)} \}$ as defined in Definition \ref{define_IRE}.   %when the Jacobian matrix of the top block of prolongation
%of $\bm{G}$ is nonsingular,
 Moreover, we have $\delta(\bm{G}) \leq  \delta(\bm{F}) - (n-r)$.
\end{theorem}
\proof \hskip4pt
 %Corresponding to Lemma \ref{lem:proj}, let $\bm{X}^{(k_d)}=\{\bm{s},\bm{y}\}$ and $\bm{z}= (t,\bm{X},\bm{X}^{(1)},...,\bm{X}^{(k_d-1)})$, where $\bm{s}= (x_1^{d_1}, ..., x_r^{d_r})$, $\bm{y}= (x_{r+1}^{d_{r+1}},...,x_n^{d_n})$. Then we can divide $\bm{B}_{k_c}$ into two parts $\bm{f(s,y,z)}=\{F_1^{(c_1)},..., F_r^{(c_r)}\}$ and $\bm{g(s,y,z)}=\{F_1^{(c_r+1)},..., F_r^{(c_n)}\}$. Similarly, we can divide $\bm{F}^{aug}$ into three parts $\bm{f(s,y,z)}=\{F_1^{(c_1)},..., F_r^{(c_r)}\}$, $\bm{f(u,\xi,z)}=\{\hat{F}_1,...,\hat{F}_r\}$ and $\bm{g(u,\xi,z)}=\{\hat{F}_{r+1},...,\hat{F}_n\}$.
By the Constant Rank Embedding Lemma \ref{lem:proj} and Definition \ref{define_IRE}, since the random constants involved can be arbitrarily ascribed, we easily get $Z_{\R}(\bm{B}_{k_c}) \cap C =  \pi Z_{\R}(\bm{F}^{aug}) \cap C$.  Further, since
$\bm{F}^{(\bm{c}-1)}$ is common to
both ${\bm{F}^{(\bm{c})}}$ and ${\bm{G}}$ we have $Z_{\R}(\bm{F}^{(\bm{c})})\cap C = \pi Z_{\R}(\bm{G}) \cap C$.

\begin{table}
\caption{signature matrix of $\bm{F}^{aug}$ and a feasible solution of $(\bar{\bm{c}},\bar{\bm{d}})$}\label{table:tritop}
\renewcommand\arraystretch{1.2}
\begin{center}

\begin{tabular}{cccc}

 & &\multicolumn{1}{c}{ $\bar{\bm{d}}(1,...,n)=\bm{d}\quad $}  &   \multicolumn{1}{c}{ $\quad \bar{\bm{d}}(n+1,...,n+r)=\bm{1}$ } \cr\cline{2-4}
  &\multicolumn{1}{|c|}{$\sigma_{i,j}(\bm{F}^{aug})$}& \multicolumn{1}{|c|}{$\bm{x}$ }& \multicolumn{1}{|c|}{$\bm{u}$} \cr\cline{2-4}
   \multicolumn{1}{l|}{$\bar{\bm{c}}(1,...,r)= \bm{0}$} &\multicolumn{1}{|c|}{$\bm{f(s,y,z)}$} & \multicolumn{1}{|c|}{$\sigma_{i,j}\leq d_{j}$} & \multicolumn{1}{|c|}{$\sigma_{i,j}=-\infty$} \cr\cline{2-4}
   \multicolumn{1}{l|}{$\bar{\bm{c}}(r+1,...,n+r)= \bm{1}$ } & \multicolumn{1}{|c|}{$\hat{\bm{F}}(\bm{z},\bm{\xi},\bm{u})\}$} &  \multicolumn{1}{|c|}{$\sigma_{i,j}\leq d_{j}-1$}     &  \multicolumn{1}{|c|}{$\sigma_{i,j}=0$ or $-\infty$} \cr\cline{2-4}
\end{tabular}
\end{center}
\end{table}

According to Table \ref{table:tritop}, we construct a pair  $(\bar{\bm{c}},\bar{\bm{d})}$:
\begin{equation}\label{opt_soln}
 \bar{c}_{i}= \left\{%
\begin{array}{ll}
    0,& i=1,\cdots,r \\
    1,& i=(r+1),\cdots,(n+r) \\
\end{array}%
\right., \bar{d}_{j}= \left\{%
\begin{array}{ll}
    d_{j},& j=1,\cdots,n \\
    1,& j=(n+1),\cdots,(n+r) \\
\end{array}%
\right.
\end{equation}

For $1\leq i \leq r$ and $ 1\leq j \leq n$, the signature matrix of $\bm{F}^{aug}$ is the same as $\bm{B}_{k_c}[1:r,1:n]$, implying that $\sigma_{i,j}(\bm{f})\leq d_{j}-0=\bar{d}_{j}-\bar{c}_{i}$.

For $1\leq i \leq r$ and  $ (n+1)\leq j \leq (n+r)$, $\sigma_{i,j}(\bm{f})=-\infty<1-0=\bar{d}_{j}-\bar{c}_{i}$.

 For $(r+1)\leq i \leq (n+r)$ and $ 1\leq j \leq n$, since $\bm{s}$ and $\bm{y}$ in $\hat{\bm{F}}$ have been replaced with dummy variables and constants, we have:  \\
 $\sigma_{i,j}(\hat{\bm{F}})\leq \sigma_{i,j}(\bm{B}_{k_c})-1 \leq d_{j}-1= \bar{d}_{j}-\bar{c}_{i}$

  For $(r+1)\leq i \leq (n+r)$ and $(n+1)\leq j \leq (n+r)$,  $\sigma_{i,j}(\hat{\bm{F}})\leq 0 = \bar{d}_{j}-\bar{c}_{i}$.

 To sum up, $(\bar{\bm{c}},\bar{\bm{d})}$ is a pair of feasible solutions of the {\ILP} (\ref{LPP}) for $\bm{F}^{aug}$. Thus, $\delta(\bm{F}^{aug})\leq \sum\limits^{n+r}_{j=1}\bar{d}_{j}-\sum\limits^{n+r}_{j=1} \bar{c}_{i} =\sum\limits^{n}_{j=1}{d}_{j}-(n-r) = \delta(\bm{B}_{k_c}) -(n-r) $.

 %Similarly,  we also construct a pair $(\hat{\bm{c}},\hat{\bm{d})}$ , for $i=1,\cdots,n$ and $ j=1,\cdots,n$, $\hat{c}_{i}=0$ and $\hat{d}_{j}= d_{j}$. Since $(\bm{c},\bm{d})$ be the optimal solution of $\bm{F}$ and $\bm{B}_{k_c}$ is the top block of $\bm{F}^{(\bm{c})}$, it holds $(\hat{\bm{c}},\hat{\bm{d})}$ is optimal solution of $\bm{B}_{k_c}$, $\delta(\bm{B}_{k_c})= \sum\limits^{n}_{j=1} d_j$.

% As for $ \bm{F}^{aug}(\bm{s},\bm{y},\bm{z},\bm{u}))= \{\bm{f(s,y,z)}, \hat{\bm{F}}(\bm{z},\bm{\xi},\bm{u})\} $, in which $\hat{\bm{F}}$ actually does not involve variable $\bm{s}$ and $\bm{y}$, there must be $\Pro\hat{\bm{F}} \subset \R[t,\bm{X},...,\bm{X}^{(k_d-1)},\bm{X}^{k_d},\bm{u},\bm{u}^{(1)}]$. Obviously, the top block of   the prolongation of $\bm{F}^{aug}$ up to order $\bar{\bm{c}}$ must be $\{\bm{f},\Pro\hat{\bm{F}}\}$, and the highest order variables in the top block of Jacobian matrix of $\bm{F}^{aug}$ are $\bm{s}$, $\bm{y}$ and $\bm{u}^{(1)}$.

Obviously, since both ${\bm{F}^{(\bm{c})}}$ and ${\bm{G}}$ have the same
block of constraints $\bm{F}^{(\bm{c}-1)}$, according to Definition {\ref{define_delta1}}, it follows that $\delta({\bm{G}})-\delta({\bm{F}^{(\bm{c})}})=\delta(\bm{F}^{aug})-\delta(\bm{B}_{k_c})\leq -(n-r)$. Finally,  $\delta(\bm{G}) \leq \delta(\bm{F}^{(\bm{c})}) - (n-r)= \delta(\bm{F}) - (n-r)$, since $\delta(\bm{F})=\delta(\bm{F}^{(\bm{c})})$ by Proposition \ref{prop:extend_delta}. \foorp \\

%From the definition of $DOF$, the $\bar{\bm{G}}$  consisting of prolongation of $\bm{F}^{(\bm{c}-1)}$ and $\bm{F}^{aug}$ basing on $(\bar{\bm{c}},\bar{\bm{d})}$, has $n+r$ more equations than $\bm{F}^{(\bm{c})}$, as to the number of equation $\{\bm{f},\hat{\bm{F}},\hat{\bm{F}}^{(1)}\}$ , {\ie} $(2n+r)$ equations, minus the number of equation $\{\bm{f},\bm{g}\}$, {\ie} $n$ equations. And $\bar{\bm{G}}$ has $2r$ more variables than $\bm{F}^{(\bm{c})}$, as to additional variables $\{\bm{u},\bm{u}^{(1)}\}$. Finally, we can deduce $\delta(\bm{G}) \leq \delta(\bar{\bm{G}})= \delta(\bm{F}) - (n-r)$

Since $\rank \bm{\Jac} = r$, restoring regularity is equivalent is some sense to finding $n-r$ hidden constraints by elimination.

\begin{remark}
Actually, most of the results in the paper can be generalized to real analytic functions. We only consider polynomially nonlinear {\DAE}s in this paper, because the homotopy continuation methods can provide all solutions of a square polynomial system and we lack of such a global solver for analytic systems.
\end{remark}

Although there are more dependent variables in $\bm{G}$, the computational cost is much lower than explicit symbolic elimination, since $\bm{G}$ and the corresponding lifted witness points can be easily constructed.
Moreover, in the IRE method, the feasible solution $(\bar{\bm{c}},\bar{\bm{d})}$  given in Equation (\ref{opt_soln}) without {\ILP} solving is an optimal solution in all examples in Section \ref{sec:ex}. Theoretically, Lemma {\ref{lem:lifting}} below
shows that the feasible solution $(\bar{\bm{c}},\bar{\bm{d})}$ is optimal under some reasonable assumptions.

\begin{lemma}\label{lem:lifting}
Suppose each equation $F_{i}$ in the top block $\bm{B}_{k_c}$ of a {\DAE} $\bm F$
contains at least one variable  $x_{j}\in \bm{X}^{(k_d)-1}$.  If $\bm F$ is also a perfect match, then $(\bar{\bm{c}},\bar{\bm{d})}$ in Equation (\ref{opt_soln}) is an optimal solution and $\delta(\bm{G}) =  \delta(\bm{F}) - (n-r)$.
\end{lemma}
\proof \hskip4pt
According to the Table \ref{table:tritop}, since $\bm{f(s,y,z)}$ is a part of {\DAE} $\bm{F}$, its corresponding $(\bar{\bm{c}}[1:r],\bar{\bm{d}}[1:n])$ is optimal. If $(\bar{\bm{c}},\bar{\bm{d}})$ is not an optimal solution, then there must be a feasible solution $({\bm{c}},{\bm{d})}$ satisfies one of the following four cases, such that $(\sum \bm{d}- \sum {\bm{c}})\leq \sum (\bar{\bm{d}}- \sum \bar{\bm{c}})$. The Lemma is now proved by contradiction.

(1) $\bm{c}=\bar{\bm{c}}$ and at least one element in ${\bm{d}}[(n + 1):(n + r)]$ is $0$. It is easy to prove it does not satisfy $d_j-c_i\geq \sigma_{i,j}$.

(2) $\bm{d}=\bar{\bm{d}}$ and at least one element in ${\bm{c}}[(r + 1):(n + r)]$ is more than $1$. It also does not satisfy $d_j-c_i\geq \sigma_{i,j}$.

(3) Some elements in ${\bm{c}}[(r + 1):(n + r)]$ are zeros and more elements in ${\bm{d}}[(n + 1):(n + r)]$ are also zeros. That implies, at least $2$ of the highest derivatives of variables in $\bm{u}$ only occur in $1$ equation.  This contradicts the perfect match condition.

(4) Some elements in ${\bm{c}}[(r + 1):(n + r)]$ are $>1$, and some elements in ${\bm{d}}[(n + 1):(n + r)]$ are also $>1$ , such that $\sum {\bm{d}}[(n + 1):(n + r)]-\sum {\bm{c}}[(r + 1):(n + r)]< (r-n)$. Since at least one of $\bm{X}^{(k_d)-1}$ occurs in every equation of $\hat{\bm{F}}(\bm{z},\bm{\xi},\bm{u})\}$, all elements in ${\bm{c}}[(r + 1):(n + r)]$ must be $\leq 1$, contrary to our assumption.

Since $(\bar{\bm{c}},\bar{\bm{d})}$ is an optimal solution of {\ILP} (\ref{LPP}) for $\bm{F}^{aug}$ it follows that,
$\delta(\bm{F}^{aug}) =  \delta(\bm{B}_{k_c}) -(n-r) \Rightarrow \delta(\bm{G}) =  \delta(\bm{F}) - (n-r)$.

\foorp \\
%Of course, there exist some cases where $(\bar{\bm{c}},\bar{\bm{d})}$ is not an optimal solution, {\ie} the $j$-th highest derivative variables does not appear in $\bm{f}$, if $d_j\geq 1$ and $(r+1)\leq j \leq n$, then $\bar{d}_{j}=d_j-1 < d_j$ is also a feasible solution. Please see Example ??

\section{Algorithms}\label{sec:alg}
\sloppy{}

This section provides a global structural differentiation method (Algorithm $7$) for solving a polynomially non-linear
{\DAE}, based on the IRE method (Algorithm $6$) --- the key algorithm to restore the regularity and to reduce index. Also, we need to recall some existing subroutines given in Algorithms $1-5$.
%For a more intuitive understanding, the flow chart of the whole algorithm is given in Figure {\ref{fig:flow}}.

Algorithm $1$ is used to find an optimal solution $(\bm{c},\bm{d})$ of {\ILP} (\ref{LPP}) of a {\DAE} ${\bm F}$ with variables $\bm{x}$, which helps to prolong {\DAE} in a special pattern to reduce its differential index.

Algorithm $2$ is used to find a real witness set $\bm{W}=\{\bm{p}_{i}|i=1,...,m\}$ by the homotopy continuation method. Here, the input $\bm{f}$ is considered as a polynomial system by taking all derivatives of $\bm x$ as new variables. For constraints of a {\DAE}, the obtained real witness points can be considered as candidate initial points. Crucially, this algorithm can find all constraint components of a {\DAE}.

Algorithm $3$  is the Singular Value Decomposition (SVD). The purpose of this algorithm is to find the numerical rank of Jacobian matrix $\bm{\Jac}$ at a real witness point $\bm{p}$ with absolute tolerance $AbsTol$.

Algorithm $4$ is a sorting method to find a sub-matrix with constant rank
by swapping the equations of the top block $\bm{B}_{k_{c}}$ and the highest derivative variables $\bm{X}^{k_d}$.
The output is a new $\bm{B}_{k_{c}}$ whose Jacobian matrix at a given real witness point $\bm{p}$ has a full rank sub-matrix $\bm{\Jac}(\bm{p})[1:r,1:r]$, where $r$ is determined by Algorithm $3$. Firstly, calculate permutation vectors of rows and columns for $\bm{\Jac}(\bm{p})$ respectively by householder QR (HQR). Then, swap equations and variables according to permutation vectors respectively. Before returning the sorted matrix, we will verify the rank of $\bm{\Jac}(\bm{p})[1:r,1:r]$ by SVD.

Algorithm $5$ is a low index {\DAE} solver implemented by one-step projection and one-step prediction. Obviously a low index {\DAE} $\bm{F}^{(\bm{c})}$ can be divided into two parts --- constraints $\bm{F}^{(\bm{c-1})}$ and a square {\ODE} $\bm{B}_{k_{c}}$. Firstly, since a initial value may be not a consistent initial value of the {\ODE}, the initial value point needs to  be projected back onto the constraints by Newton iteration to find a nearby consistent initial value point satisfying the constraints. Secondly, an {\ODE} solver,
such as the Runge-Kutta method or the Euler method, is used to make a one-step prediction from the previous consistent initial value point. Through step-by-step iteration, the {\DAE} can be solved numerically, where the tolerance can be set as needed.

\begin{breakablealgorithm}
\caption{}
 \begin{algorithmic}[1]
  \State  $(\bm{c},\bm{d})=Structure(\bm{F}, \bm{x})$, such as Pryce method see \cite{Nedialkov08}.
   \end{algorithmic}
\end{breakablealgorithm}

\vspace{0.2cm}

\begin{breakablealgorithm}
\caption{}
 \begin{algorithmic}[1]
  \State  $\bm{W}=\{\bm{p}_{i}|i=1,...,m\}=witness(\bm{f})$, see \cite{WRF17}. // $m$ is the number of real witness points.
   \end{algorithmic}
\end{breakablealgorithm}

\vspace{0.2cm}

\begin{breakablealgorithm}
\caption{}
 \begin{algorithmic}[1]
  \State  $r=Rank(\bm{\Jac}(\bm{p}), AbsTol)$, see Section $2.5$ \cite{Golub13}.
   \end{algorithmic}
\end{breakablealgorithm}

\vspace{0.2cm}

\begin{breakablealgorithm}
\caption{}
 \begin{algorithmic}[1]
 \Require the top block equations $\bm{B}_{k_{c}}$, Jacobian matrix $\bm{\Jac}$ with witness point $\bm{p}$, the constant rank $r$,  absolute tolerance $AbsTol$
 \Ensure  recombination of the top block equations $\bm{B}_{k_{c}}$
 \Function {SORT}{$\bm{B}_{k_{c}}, \bm{\Jac}(\bm{p}), r, AbsTol$}
  \State $\bm{piv_{row}}=HQR(\bm{\Jac}(\bm{p}),AbsTol)$, see Section $5.2$  \cite{Golub13};
  \State $\bm{piv_{col}}=HQR(\bm{\Jac}^{T}(\bm{p}),AbsTol)$; // $\bm{piv_{row}}$ and $\bm{piv_{col}}$ are the permutation vector of rows and columns, respectively;
  \State $\bm{B}_{k_{c}}=\bm{B}_{k_{c}}[\bm{piv_{row}}]$, // swap equations;
  \State $\bm{B}_{k_{c}}=\bm{B}_{k_{c}}(\bm{X}_{k_d}[\bm{piv_{col}}])$, // swap the highest derivative variables;
  \State verify the rank of $\bm{\Jac}(\bm{p})[1:r,1:r]$ by SVD.
   \EndFunction
   \end{algorithmic}
\end{breakablealgorithm}

\vspace{0.2cm}

\begin{breakablealgorithm}
\caption{}
 \begin{algorithmic}[1]
 \Require low index {\DAE} equations $\bm{F}^{(\bm{c})}$ and dependent variables $\bm{x}$ with independent variable $t\in [t_0,t_{end}]$, initial point $\bm{p}$,  absolute tolerance $AbsTol$ and relative tolerance $RelTol$
 \Ensure  numerical solutions of {\DAE} $\bm{x}(t)$
 \Function {{DAESOLVER}}{$\bm{F}^{(\bm{c})}, \bm{x}, \bm{p}, [t_0, t_{end}], AbsTol, RelTol$} //
  \State  $j=0$, $\bm{x}(t_{0})=\bm{p}$, set step $h$ and the maximum number of iterations $N$;
  \While {$t_{j}<=t_{end}$}
   \State $\bm{x}(t_{j}) = Newton(\bm{F}^{(\bm{c-1})},\bm{x}(t_{j}), AbsTol, N)$ // Refinement, see \cite{Mathews2004};
  \State $\bm{x}(t_{j+1}) = OdeSolver(\bm{B}_{k_c}, \bm{x}(t_{j}), AbsTol, RelTol, h)$ // such as $ode45$, Euler method, $ode15i$ {\etc};
  \State $j=j+1$;
   \EndWhile
   \EndFunction
   \end{algorithmic}
\end{breakablealgorithm}

\vspace{0.2cm}

\renewcommand{\thealgorithm}{6}\label{alg:alg1}
\begin{breakablealgorithm}
\caption{Index Reduction by Embedding}
 \begin{algorithmic}[1]
    \Require {\DAE} equations $\bm{F}$ and dependent variables $\bm{x}$ with independent variable $t\in [t_0,t_{end}]$, real witness point $\bm{p}$,  absolute tolerance $AbsTol$
    \Ensure  modified {\DAE} new equations $\bm{F}^{(\bm{c})}$ and new real witness point $\bm{p}$

    \Function {IRE}{$\bm{F},\bm{x},\bm{p}_{i}, AbsTol$}
    \While {true}
    \State Structural Analysis: $(\bm{c},\bm{d})=Structure(\bm{F},\bm{x})$
    \State $n=length(\bm{c})$, $k_d = \max(d_j)$, $k_c = \max c_i$, $\delta=\sum d_j - \sum c_i$ \label{alg1goto}
    \State Construct: $\bm{F}^{(\bm{c})}$  , $\bm{B}_{k_{c}}$,  $\bm{\Jac}$ by Equation (\ref{eq:DefFc},\ref{eq:B_i},\ref{Jac})
    \State $r=Rank(\bm{\Jac}(\bm{p}), AbsTol)$
    \If{$r=n$}
    \State \Return  $\bm{F}^{(\bm{c})}$, $\bm{p}$
    \ElsIf {$\delta-(n-r) \leq 0$}
    \State \Return Error // this {\DAE} does not have a solution.
    \EndIf
    \State $\{\bm{f(s,y,z)},\bm{g(s,y,z)}\}=$SORT$(\bm{B}_{k_{c}}, \bm{\Jac}(\bm{p}), r, AbsTol)$
    \State Introduce $n$ new equations $\hat{\bm{F}}=\{\bm{f(s,y,z)},\bm{g(s,y,z)}\}$
    \State Replace $\bm{s}$  by $\bm{u}$ in $\hat{\bm{F}}$ // refer to Definition {\ref{define_IRE}}
     \State Replace $\bm{y}$ by random constants $\bm{\xi}$ in $\hat{\bm{F}}$
     \State Substitute $\{t_0, \bm{p}, \bm{\xi}\}$ into $\bm{f(s,y,z)}$ to calculate $\bm{s}$, note as $\hat{\bm{u}}$
    \State  $\bm{p}\leftarrow(\bm{p}, \hat{\bm{u}})$ //  corresponding lifting
of consistent initial value
     \State $\bm{F}^{aug} = \{\bm{f(s,y,z)},\hat{\bm{F}}\}$
     \State $\bm{F} \leftarrow \{ \bm{F}^{(\bm{c}-1)}, \bm{F}^{aug}\}$, $\bm{x} \leftarrow (\bm{x},\bm{u})$ //extend equations and variables
     \If {$\bm{F}^{aug}$ satisfies Lemma {\ref{lem:lifting}}}
     \State $\bar{\bm{c}}= [\bm{0}_{r}, \bm{1}_{n}]$, $\bar{\bm{d}}= [\bm{d},\bm{1}_{r}]$ by Equation (\ref{opt_soln})
     \State  $\bm{c}=[\bm{0}_{(\sum c_{j})},\bar{\bm{c}}]$, $\bm{d}=\bar{\bm{d}}$
      \State Goto \ref{alg1goto}
     \EndIf
    \EndWhile
    \EndFunction
 \end{algorithmic}
\end{breakablealgorithm}

\renewcommand{\thealgorithm}{7}\label{alg:alg2}
\begin{breakablealgorithm}
\caption{Global Structural Differentiation Method}
 \begin{algorithmic}[1]
    \Require {\DAE} equations $\bm{F}$ and dependent variables $\bm{x}$ with independent variable $t\in [t_0,t_{end}]$,  absolute tolerance $AbsTol$ and relative tolerance $RelTol$
    \Ensure numerical solutions of {\DAE} $\bm{x}^{*}(t)$
    \State Initialization:  check the number of equations \#eqns and dependent variables \#dvars
    \If {\#eqns $\neq$ \#dvars} \State \Return False \EndIf
    \State Set $\bm{x}^{*}(t)=\{\}$
    \State Structural Analysis: $(\bm{c},\bm{d})=Structure(\bm{F},\bm{x})$
    \State Construct: prolonged system $\bm{F}^{(\bm{c})}$,  Jacobian matrix $\bm{\Jac}$
    \State Find real witness points: $ \bm{W}=\{ \bm{p}_{i}=witness(\bm{F}^{(\bm{c-1})}(t_0))|i=1,...,m\}$  //$m$ is number of real witness points.
    \For {$\bm{p}_i\in \bm{P}$}
    \State $\{\tilde{\bm{F}},\tilde{\bm{p}}_{i}\}:=IRE(\bm{F},\bm{x},\bm{p}_{i}, AbsTol)$
    \State $\tilde{\bm{x}}(t)=$DAESOLVER$(\tilde{\bm{F}},\bm{x},\tilde{\bm{p}}_{i}, [t_0, t_{end}], AbsTol, RelTol)$
    \State $\bm{x}^{*}(t)=\{\bm{x}^{*}(t), \tilde{\bm{x}}(t)[1,...,n]\}$
    \EndFor
    \State \Return $\bm{x}^{*}(t)$
 \end{algorithmic}
\end{breakablealgorithm}

\section{Examples}\label{sec:ex}
\sloppy{}

In this section, we use five examples. These included three symbolic cancellation examples: transistor amplifier, modified pendulum and ring modulator.
Also included are two numerical degeneration examples: Example \ref{ex:4} and the bending deformation of a beam.

In a similar manner to that described in \cite{Taihei19}, we compare several methods on the {\DAE} for the above examples.
In particular, we apply the following four methods to the above $5$ {\DAE}s: (a) Pryce method, (b) the substitution method, (c) the augmentation method, (d) the IRE method.
We use Matlab R$2021$a for the numerical computations with the error settings AbsTol =$10^{-6}$ and RelTol = $10^{-3}$.

\subsection{Transistor Amplifier (index-$1$)}\label{ssec:exam1}
%Let's consider a transistor amplifier problem arising from an electrical network \cite{Mazzia08}. It's a linear {\ODE} system with an identically singular Jacobian matrix.
%
%\[\left\{\begin{array}{rcc}
%C_1\cdot(\dot{x}_{1}-\dot{x}_{2})+(x_{1}-U_{e})/R_{0}&=&0\\
%C_1\cdot(\dot{x}_{1}-\dot{x}_{2})-(1-\alpha)\cdot f(x_{2}-x_{3})+U_b/R_2-x_{2}\cdot(1/R_1+1/R_2)&=&0\\
%C_2\cdot\dot{x}_{3}+x_{3}/R_3-f(x_{2}-x_{3})&=&0\\
%C_3\cdot(\dot{x}_{4}-\dot{x}_{5})+x_{4}/R_4-U_b/R_4+\alpha\cdot f(x_{2}-x_{3})&=&0\\
%C_3\cdot(\dot{x}_{4}-\dot{x}_{5})-x_{5}\cdot(1/R_5+1/R_6)+U_b/R_6-(1-\alpha)\cdot f(x_{5}-x_{6})&=&0\\
%C_4\cdot\dot{x}_{6}+x_6/R_7-f(x_{5}-x_{6})&=&0\\
%C_5\cdot(\dot{x}_{7}-\dot{x}_{8})+x_{7}/R_8-U_b/R_8+\alpha\cdot f(x_{5}-x_{6})&=&0\\
%C_5\cdot(\dot{x}_{7}-\dot{x}_{8})-x_{8}/R9&=&0
%\end{array}\right.\]
%
%where $f(x)=\beta(exp(x/U_F)-1)$ and $U_{e}= 0.1\cdot sin(200\pi t)$ with constant parameters $U_b$, $U_F$, $\alpha$, $\beta$, $R_0, R_1, \ldots, R_9$, and $C_1, \ldots, C_5$. For more details, such as initial values, please refer to \cite{Taihei19}.

First, we discuss a transistor amplifier example existing in electrical network \cite{Mazzia08}. It's a linear {\ODE} system with an identically singular Jacobian matrix. For more details, see \cite{Taihei19}. The structural information obtained by the Pryce method is that the dual optimal solution is $\bm{c}=\bm{0}_{1\times 8}$ and $\bm{d}=\bm{1}_{1\times 8}$, such that $n=\delta=8$. For the Jacobian matrix,  we have $\rank \bm{\Jac} = r =\rank \bm{\Jac}[(4,6,3,1,7),(4,6,3,1,7)]= 5$.

Obviously, we still cannot solve the system directly after the Pryce method. Fortunately, as it is a linear {\DAE}, almost all existing improved structural methods can be used to regularize it.

It is easy to get $\bm{F}=\bm{F}^{(\bm{c})}$ since $\bm{c}$ is a zero vector.
By the IRE method, according to Definition {\ref{define_IRE}}, we have $\bm{s}=\{\dot{x}_{4}, \dot{x}_{6}, \dot{x}_{3}, \dot{x}_{1}, \dot{x}_{7}\}$, $\bm{y}=\{\dot{x}_{2}, \dot{x}_{5}, \dot{x}_{8}\}$, $\bm{f(s,y,z)}=\{F_{4}, F_6, F_3, F_1, F_7\}$ and $\bm{g(s,y,z)}=\{F_{2}, F_5, F_8\}$. Thus, $\hat{\bm{F}}=\{\bm{f(u,\xi,z)},\bm{g(u,\xi,z)}\}$,
where $\bm{s}$ and $\bm{y}$ are replaced by $(u_1, u_2, u_3, u_4, u_5)$ and some random constants $(\xi_1, \xi_2, \xi_3)$ respectively.  Finally, we construct a new top block $\bm{F}^{aug} = \{\bm{f(s,y,z)},\hat{\bm{F}}\}$ of the prolonged {\DAE}, where $\hat{\bm{F}}$ is given below.

%\[\bm{F}^{aug}=\left\{\begin{array}{rcc}
%C_1\cdot(u_4-\xi_1)+(x_{1}-U_{e})/R_{0}&=&0\\
%C_1\cdot(u_4-\xi_1)-(1-\alpha)\cdot f(x_{2}-x_{3})+U_b/R_2-x_{2}\cdot(1/R_1+1/R_2)&=&0\\
%C_2\cdot u_3+x_{3}/R_3-f(x_{2}-x_{3})&=&0\\
%C_3\cdot(u_1-\xi_2)+x_{4}/R_4-U_b/R_4+\alpha\cdot f(x_{2}-x_{3})&=&0\\
%C_3\cdot(u_1-\xi_2)-x_{5}\cdot(1/R_5+1/R_6)+U_b/R_6-(1-\alpha)\cdot f(x_{5}-x_{6})&=&0\\
%C_4\cdot u_2+x_6/R_7-f(x_{5}-x_{6})&=&0\\
%C_5\cdot(u_5-\xi_3)+x_{7}/R_8-U_b/R_8+\alpha\cdot f(x_{5}-x_{6})&=&0\\
%C_5\cdot(u_5-\xi_3)-x_{8}/R9&=&0\\
%C_1\cdot(\dot{x}_{1}-\dot{x}_{2})+(x_{1}-U_{e})/R_{0}&=&0\\
%C_2\cdot\dot{x}_{3}+x_{3}/R_3-f(x_{2}-x_{3})&=&0\\
%C_3\cdot(\dot{x}_{4}-\dot{x}_{5})+x_{4}/R_4-U_b/R_4+\alpha\cdot f(x_{2}-x_{3})&=&0\\
%C_4\cdot\dot{x}_{6}+x_6/R_7-f(x_{5}-x_{6})&=&0\\
%C_5\cdot(\dot{x}_{7}-\dot{x}_{8})+x_{7}/R_8-U_b/R_8+\alpha\cdot f(x_{5}-x_{6})&=&0\\
%\end{array}\right.\]

\[\hat{\bm{F}}=\left\{\begin{array}{rcc}
C_1\cdot(u_4-\xi_1)+(x_{1}-U_{e})/R_{0}&=&0\\
C_1\cdot(u_4-\xi_1)-(1-\alpha)\cdot f(x_{2}-x_{3})+U_b/R_2-x_{2}\cdot(1/R_1+1/R_2)&=&0\\
C_2\cdot u_3+x_{3}/R_3-f(x_{2}-x_{3})&=&0\\
C_3\cdot(u_1-\xi_2)+x_{4}/R_4-U_b/R_4+\alpha\cdot f(x_{2}-x_{3})&=&0\\
C_3\cdot(u_1-\xi_2)-x_{5}\cdot(1/R_5+1/R_6)+U_b/R_6-(1-\alpha)\cdot f(x_{5}-x_{6})&=&0\\
C_4\cdot u_2+x_6/R_7-f(x_{5}-x_{6})&=&0\\
C_5\cdot(u_5-\xi_3)+x_{7}/R_8-U_b/R_8+\alpha\cdot f(x_{5}-x_{6})&=&0\\
C_5\cdot(u_5-\xi_3)-x_{8}/R9&=&0
\end{array}\right.\]

\begin{figure}[htbp]
\centering
\subfigure[Pryce method]{
\includegraphics[width=0.45\textwidth,height=0.30\textwidth]{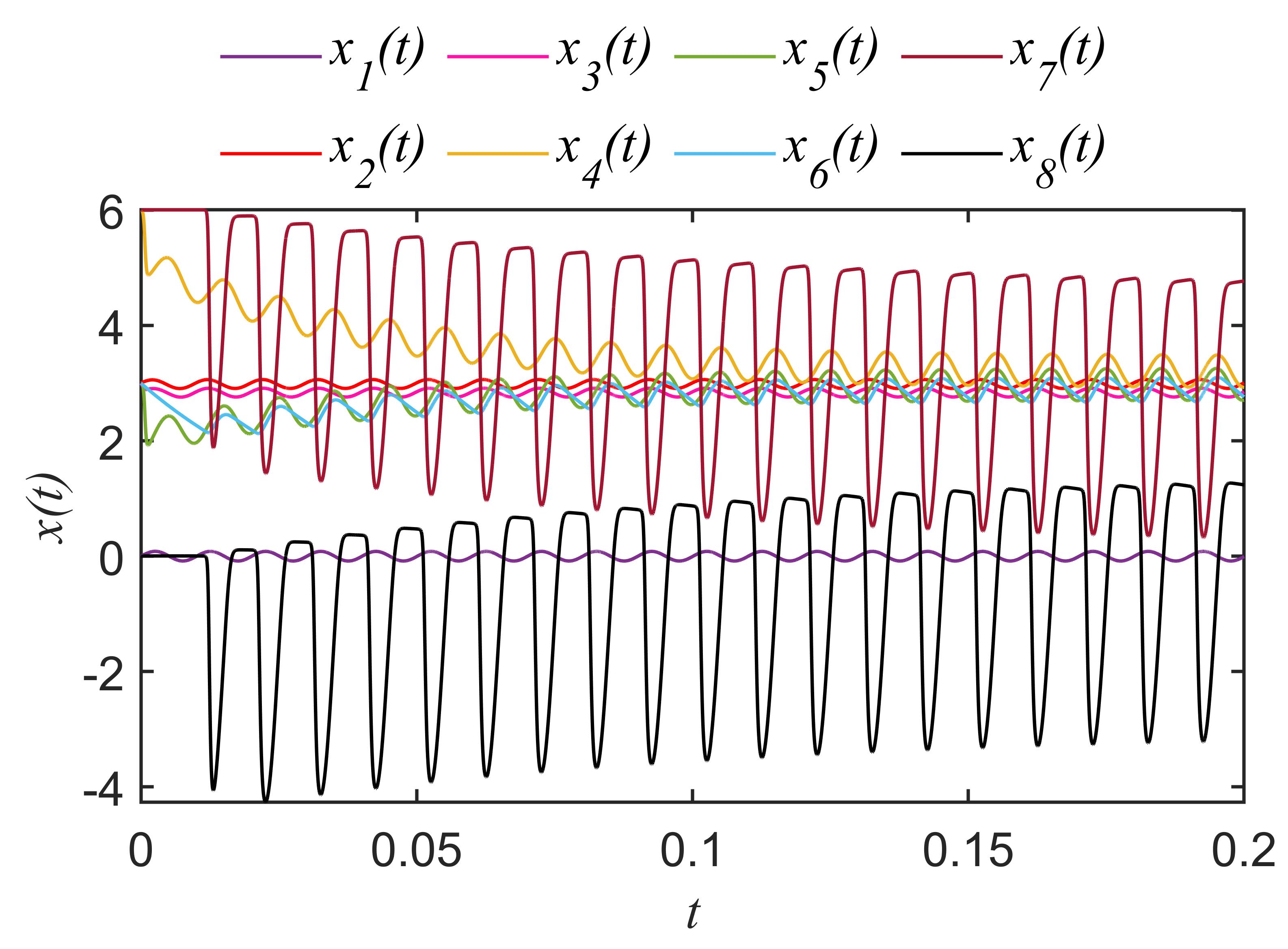}
%\caption{fig1}
}
\quad
\subfigure[substitution]{
\includegraphics[width=0.45\textwidth,height=0.30\textwidth]{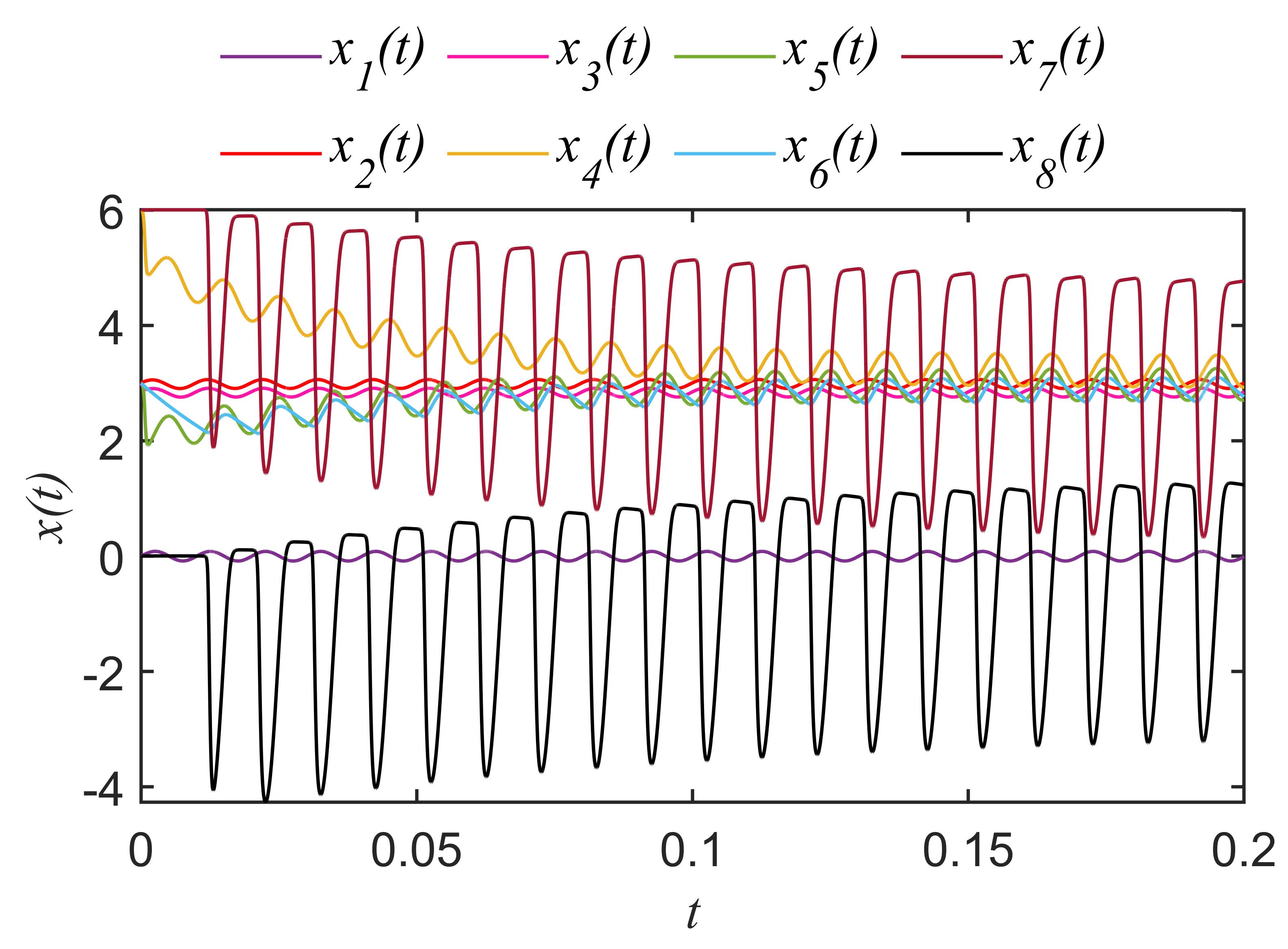}
}
\quad
\subfigure[augmentation]{
\includegraphics[width=0.45\textwidth,height=0.30\textwidth]{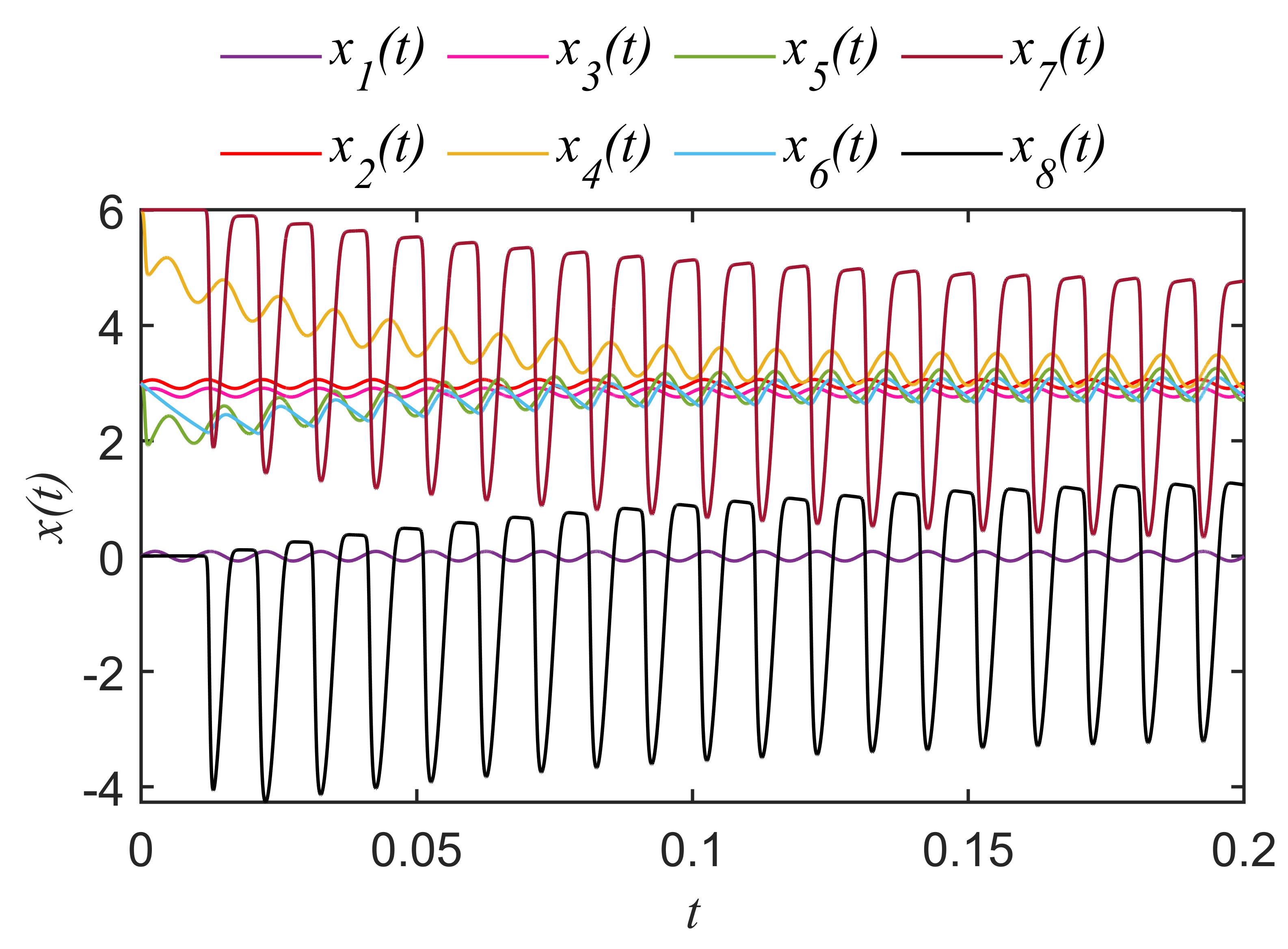}
}
\quad
\subfigure[IRE]{
\includegraphics[width=0.47\textwidth,height=0.31\textwidth]{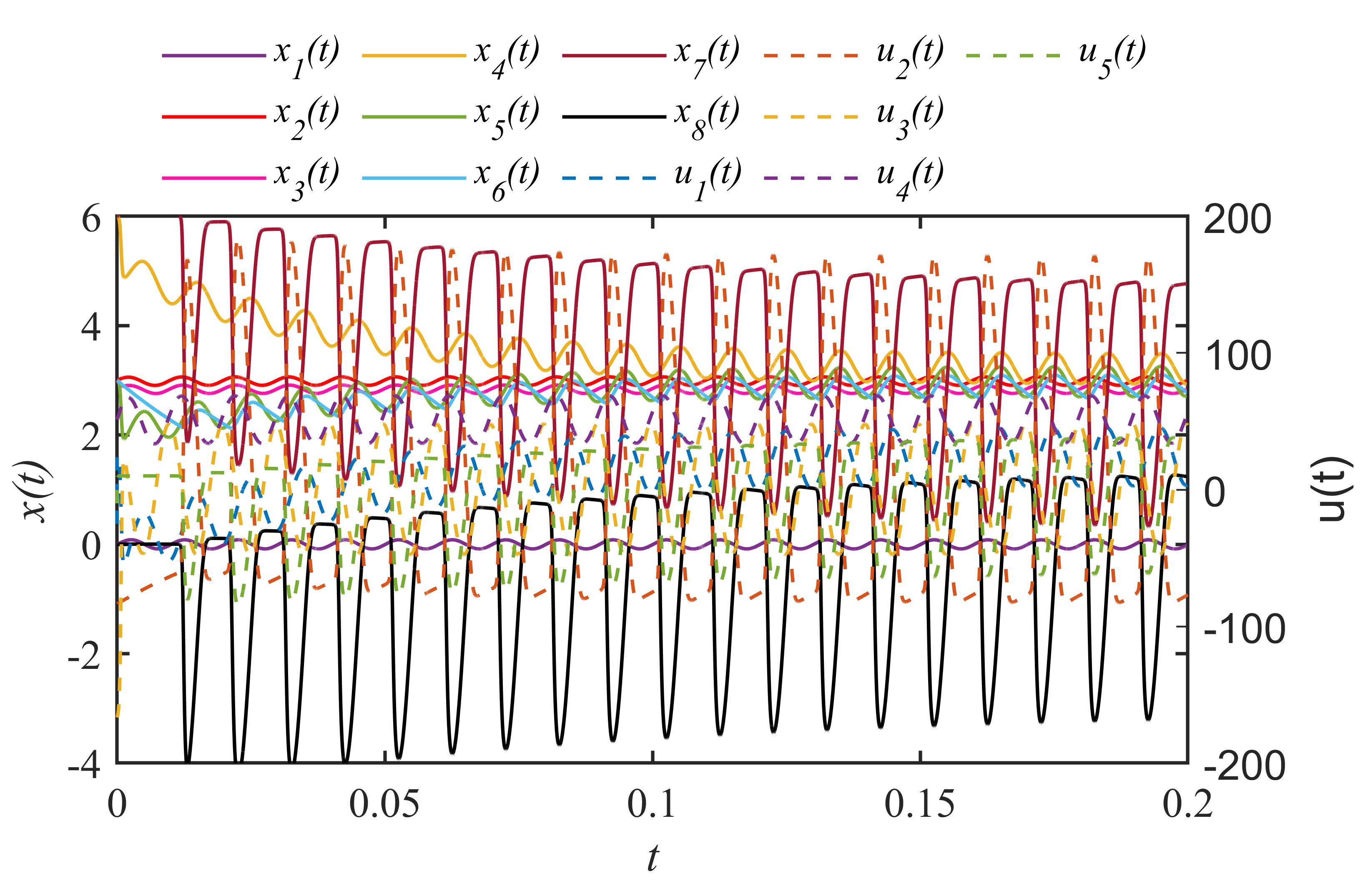}
}
\caption{Numerical Solution of Transistor Amplifier}\label{fig:amp}
\end{figure}

After the IRE method, we can directly construct an optimal solution of {\ILP} with $\bar{\bm{c}}=(\bm{0}_{1\times 5},\bm{1}_{1\times 8})$ and $\bar{ \bm{d}}=(\bm{1}_{1\times 8},\bm{1}_{1\times 5})$ by Lemma \ref{lem:lifting}.
Actually it is equivalent to the optimal solution $\bar{\bm{c}}=(\bm{0}_{1\times 5},1,1,0,1,1,0,1,1)$ and $\bar{ \bm{d}}=(\bm{1}_{1\times 8},1,0,0,1,1)$  calculated by {\ILP} and
both give the same optimal value of the new system $\bar {\delta} =  \delta-n+r = 8-8+5$.

Then we can verify that the determinant of the new Jacobian matrix is a non-zero constant. Furthermore, the IRE method in this example finish the index reduction just by one step, rather than $3$ steps repeatedly by the substitution method or the augmentation method shown in section $6.2$ \cite{Taihei19}. In other words, it shows the IRE method is more efficient for this example.

Specifically, for numerical solution, the initial value of $\bm{u}(0)$ and $\bm{\xi}$ are set corresponding to $\dot{\bm x}(0)$ in Section A.$2$ of \cite{Taihei19}, respectively.

\subsection{Non-linearly Modified Pendulum (index-$3$)}\label{ssec:exam2}
% This {\DAE} is obtained by non-linearly changing the variable of a simple pendulum {\DAE} \cite{Taihei19}. In which, $g=9.8$ is an acceleration constant of gravity. For more details, please refer to A.$2$ in \cite{Taihei19}.
%
%\[\left\{\begin{array}{rcc}
%\dot{x}_{4}-{x}_{1}\cdot{x}_{2}\cdot\cos(x_{3})&=&0\\
%\dot{x}_{5}-{x}_{2}^{2}\cdot\cos({x}_{3})\cdot\sin(x_{3})+g&=&0\\
%{x}_{1}^{2}+{x}_{2}^{2}\cdot\sin({x}_{3})^{2}-1&=&0\\
%\tanh((\dot{x}_{1}-x_{4}))&=&0\\
%\dot{x}_{2}\cdot\sin(x_{3})+x_{2}\cdot\dot{x}_{3}\cdot\cos(x_{3})-x_{5}&=&0
%\end{array}\right.\]

This nonlinear {\DAE} system consisting of $4$
differential equations and $1$ algebraic equation, is obtained by dynamic analysis and modeling of a simple pendulum. See \cite{Taihei19} for more details.

After structural analysis, we get the dual optimal solution is $\bm{c}=(0,0,1,0,0)$ and $\bm{d}=(1,\cdots,1)$, with $\delta=4$ and $n=5$. Moreover, the rank of jacobian matrix is $\rank \bm{\Jac} = r =\rank \bm{\Jac}[(3,2,1,4),(3,1,4,5)]= 4$. Thus, the constraint is $\bm{F}^{(\bm{c}-1)}=\{{x}_{1}^{2}+{x}_{2}^{2}\cdot\sin({x}_{3})^{2}-1=0\}$.

From Section {\ref{ssec:exam1}}, by the IRE method, let $\bm{s}=\{\dot{x}_{3}, \dot{x}_{1}, \dot{x}_{4}, \dot{x}_{5}\}$, $\bm{y}=\{\dot{x}_{2}\}$, $\bm{f(s,y,z)}=\{F_{3}, F_2, F_1, F_4\}$ and $\bm{g(s,y,z)}=\{ F_5\}$.  Then we need to replace $\bm{s}$  by $\{u_1, u_2, u_3, u_4\}$ and $\bm{y}$ by a random constant $\xi$ in $\hat{\bm{F}}$, respectively. Finally, we can get a modified {\DAE} $\{\bm{F}^{(\bm{c}-1)}, \bm{F}^{aug}\}$, in which $\bm{F}^{aug} = \{\bm{f(s,y,z)},\hat{\bm{F}}\}$.

\[\hat{\bm{F}}=\left\{\begin{array}{rcc}
{u}_{3}-{x}_{1}\cdot{x}_{2}\cdot\cos(x_{3})&=&0\\
{u}_{4}-{x}_{2}^{2}\cdot\cos({x}_{3})\cdot\sin(x_{3})+g&=&0\\
2\cdot{x}_{1}\cdot u_{1}+2\cdot{x}_{2}\cdot \xi\cdot\sin({x}_{3})^{2}+2\cdot{x}_{2}^{2}\cdot\sin({x}_{3})\cdot\cos({x}_{3})\cdot u_{2}&=&0\\
\tanh((u_{1}-x_{4}))&=&0\\
\xi\cdot\sin(x_{3})+x_{2}\cdot u_{2}\cdot\cos(x_{3})-x_{5}&=&0
\end{array}\right.\]

We can construct a new optimal solution $(\bar{\bm{c}},\bar{\bm{d}})$ by {\ILP} for $\bm{F}^{aug}$ by Lemma {\ref{lem:lifting}} directly, which yields ${\bm{c}}=(\bm{0}_{1\times 4},0,0,\bm{1}_{1\times 3})$ and ${\bm{d}}=(\bm{1}_{1\times 5},1,1,0,0)$ with the same optimal value $\bar {\delta} = \sum{\bar{d}_{j}}-\sum{\bar{c}_{i}}-\#eqns(\bm{F}^{(\bm{c}-1)})=9-5-1=\delta-n+r=4-5+4$. %And additional constraints can be obtained.

%\[\bm{F}^{aug}=\left\{\begin{array}{rcc}
%{u}_{3}-{x}_{1}\cdot{x}_{2}\cdot\cos(x_{3})&=&0\\
%{u}_{4}-{x}_{2}^{2}\cdot\cos({x}_{3})\cdot\sin(x_{3})+g&=&0\\
%2\cdot{x}_{1}\cdot u_{1}+2\cdot{x}_{2}\cdot \xi\cdot\sin({x}_{3})^{2}+2\cdot{x}_{2}^{2}\cdot\sin({x}_{3})\cdot\cos({x}_{3})\cdot u_{2}&=&0\\
%\tanh((u_{1}-x_{4}))&=&0\\
%\xi\cdot\sin(x_{3})+x_{2}\cdot u_{2}\cdot\cos(x_{3})-x_{5}&=&0\\
%\dot{x}_{4}-{x}_{1}\cdot{x}_{2}\cdot\cos(x_{3})&=&0\\
%\dot{x}_{5}-{x}_{2}^{2}\cdot\cos({x}_{3})\cdot\sin(x_{3})+g&=&0\\
%2\cdot\dot{x}_{1}\cdot{x}_{1}+2\cdot\dot{x}_{2}\cdot{x}_{2}\cdot\sin({x}_{3})^{2}+2\cdot{x}_{2}^{2}\cdot\sin({x}_{3})\cdot\cos({x}_{x})\cdot \dot{x}_{3}&=&0\\
%\tanh((\dot{x}_{1}-x_{4}))&=&0
%\end{array}\right.\]

Unfortunately, the Jacobian matrix of the new top block $ \bm{F}^{aug} $ is also singular, with $\rank \bm{\Jac}(\bm{F}^{aug}) =\rank \bm{\Jac}[(1:6,8:9),(1,3:9)] = 8$. Similarly, we need another modification of $\bm{F}^{aug}$ by the IRE method. Finally, this {\DAE} system has been regularized.  The final optimal value is $2 = \bar{\delta}-9+8$. The numerical results are shown in Figure \ref{fig:pen}.

Compared with the one additional equation of the augmentation method, the IRE method in this example will introduce more equations which will affect the efficiency of the numerical solution, although both methods can be successful after two steps of regularization. However, this adverse effect only exists when the Jacobian matrix is very close to being full rank, {\ie} $r=n-1$.

%Similar to Section \ref{ssec:exam1}, we set the initial value of the numerical solution of this {\DAE}. And numerical results are shown in Figure \ref{fig:pen}.

\begin{figure}[htbp]
\centering
\subfigure[Pryce method]{
\includegraphics[width=0.45\textwidth,height=0.30\textwidth]{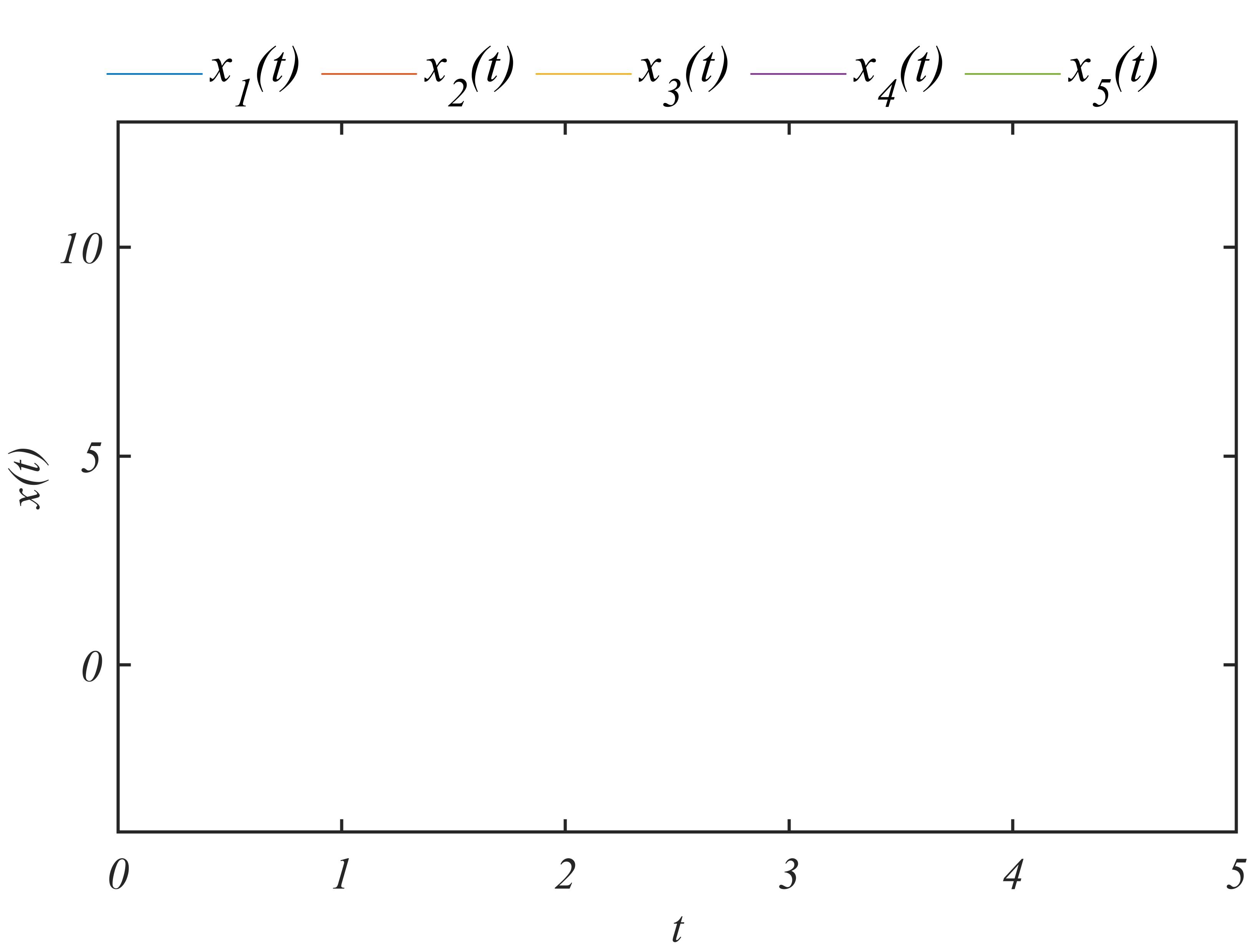}
%\caption{fig1}
}
\quad
\subfigure[substitution]{
\includegraphics[width=0.45\textwidth,height=0.30\textwidth]{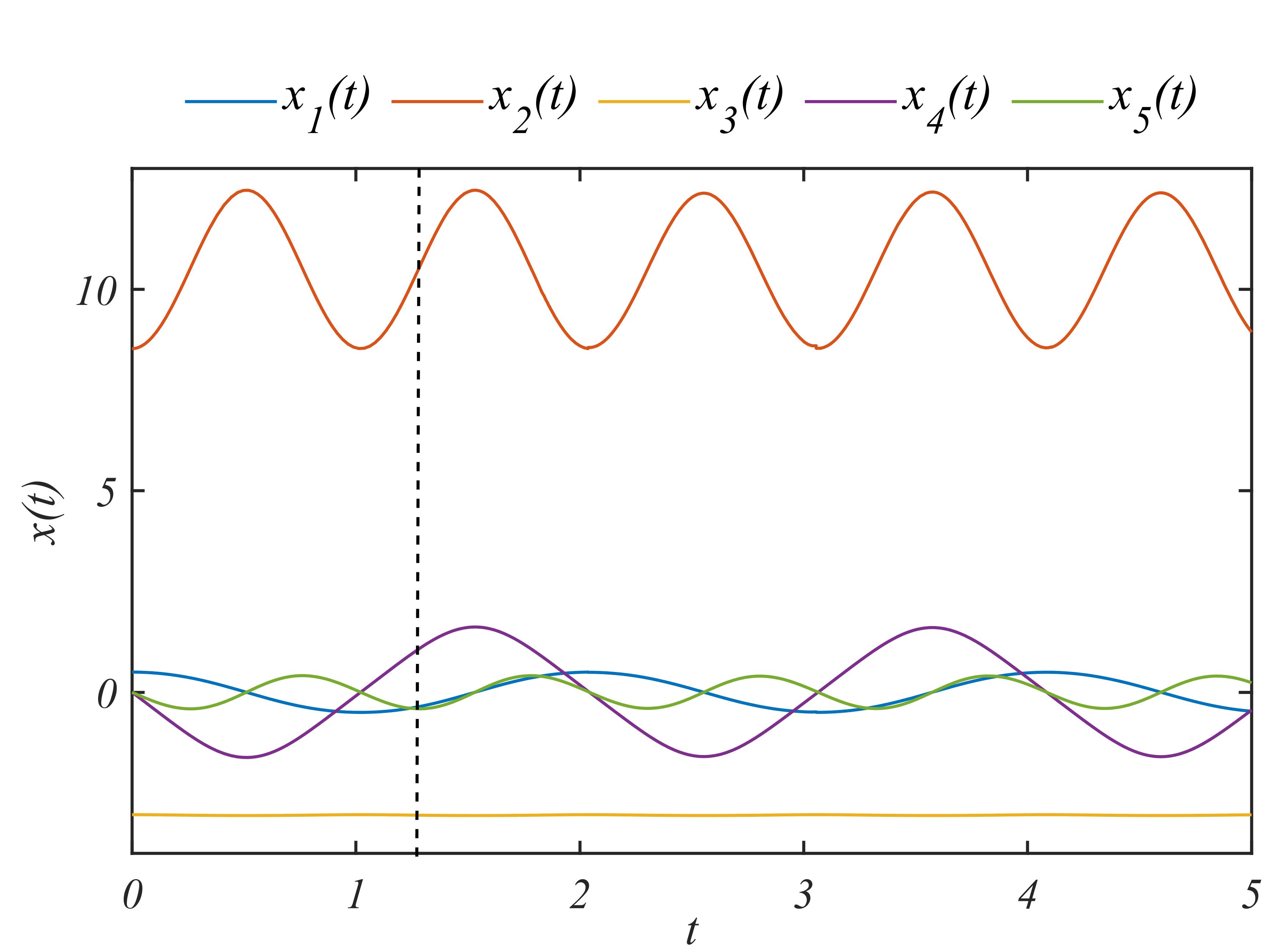}
}
\quad
\subfigure[augmentation]{
\includegraphics[width=0.45\textwidth,height=0.30\textwidth]{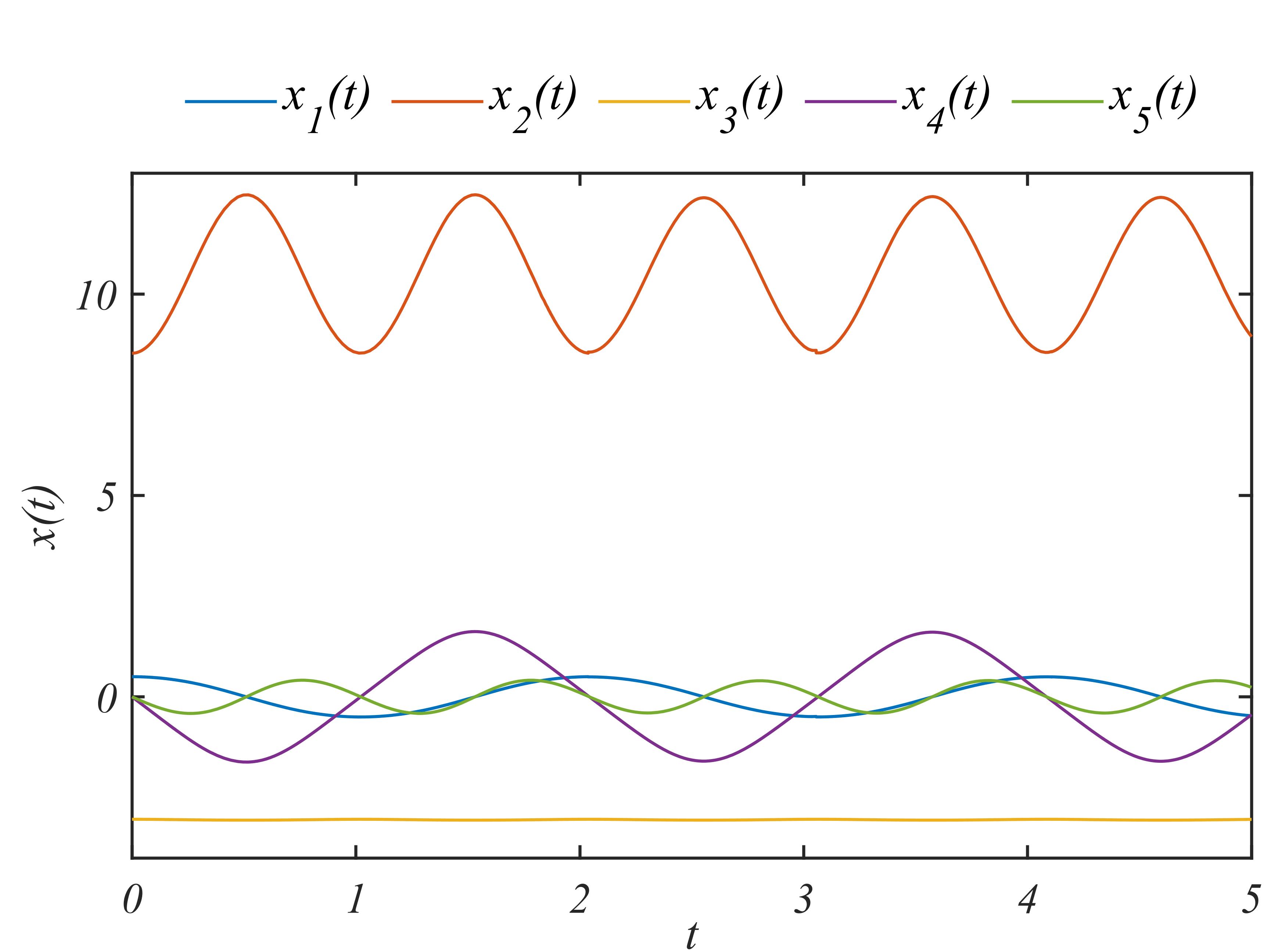}
}
\quad
\subfigure[IRE]{
\includegraphics[width=0.45\textwidth,height=0.30\textwidth]{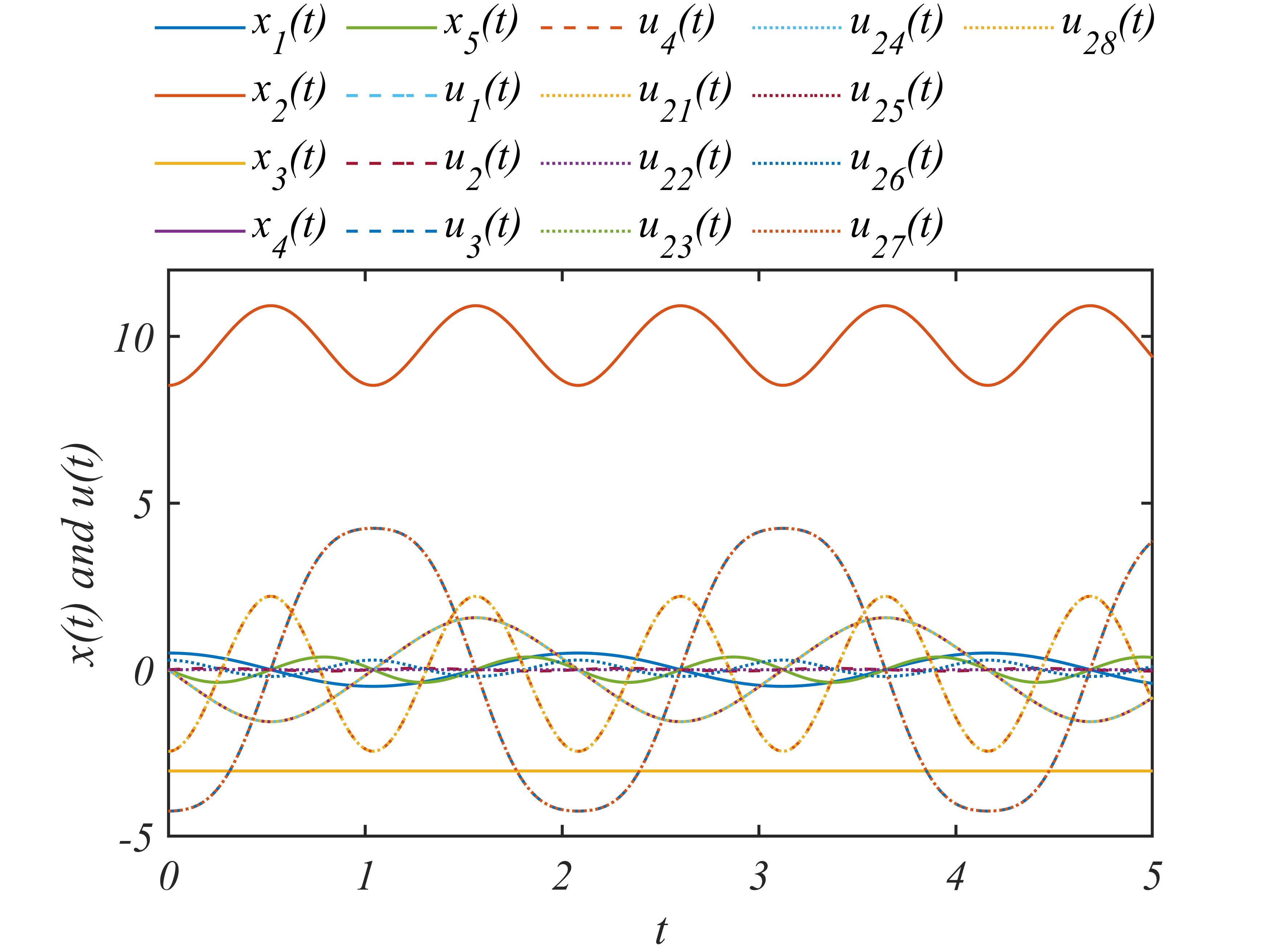}
}
\caption{Numerical Solution of Non-linearly Modified Pendulum}\label{fig:pen}
\end{figure}

\subsection{Ring Modulator (index-$2$)}
%The problem originates from electrical circuit analysis, describing the behaviour of the ring modulator. In \cite{Taihei19}, it's a {\DAE} of index 2, consisting of
%$11$ differential equations and $4$ algebraic equations.
%
%\[\left\{\begin{array}{rcc}  {C}\dot{x}_{1}-x_{8}-0.5{x_{10}}+0.5{x_{11}}+x_{14}-x_{1}/R&=&0\\ {C}\dot{x}_{2}-x_{9}-0.5x_{12}+0.5x_{13}+x_{15}-x_{2}/R&=&0\\
% x_{10}-G(U_{D1})+G(U_{D4})&=&0\\
% x_{11}-G(U_{D2})+G(U_{D3})&=&0\\
% x_{12}+G(U_{D1})-G(U_{D3})&=&0\\
% x_{13}+G(U_{D2})-G(U_{D4})&=&0\\
% \dot{x}_{7}+({x}_{7}/R_p-G(U_{D1})-G(U_{D2})+G(U_{D3})+G(U_{D4}))/C_p&=&0\\
%  \dot{x}_{8}+x_{1}/L_h&=&0\\
%  \dot{x}_{9}+x_{2}/L_h&=&0\\
%  \dot{x}_{10}+(x_{3}-0.5x_{1}+R_{g2}x_{10})/L_{s2}&=&0\\ \dot{x}_{11}+(-x_{4}+0.5x_{1}+R_{g2}x_{11})/L_{s3}&=&0\\
%  \dot{x}_{12}+(x_{5}-0.5x_{2}+R_{g2}x_{12})/L_{s2}&=&0\\ \dot{x}_{13}+(-x_{6}+0.5x_{2}+R_{g2}x_{13})/L_{s3}&=&0\\
%  \dot{x}_{14}+(x_1+(R_{g1}+R_{i})x_{14}-U_{in}(t))/L_{s1}&=&0\\
%  \dot{x}_{15}+(x_2+(R_{c}+R_{g1})x_{15})/L_{s1}&=&0 \end{array}\right.\]
%
%Where $U_{D1}=x_3-x_5-x_7-U_{in2}(t)$, $U_{D2}=-x_4+x_6-x_7-U_{in2}(t)$, $U_{D3}=x_4+x_5+x_7+U_{in2}(t)$, $U_{D4}=-x_3-x_6+x_7+U_{in2}(t)$, $G(U)=\gamma (e^{\delta U}-1)$, $U_{in1}(t)=0.5\cdot sin(2000\pi t)$, $U_{in2}(t)=2\cdot sin(20000\pi t)$.For more details, please refer to A.$2$ in \cite{Taihei19}.

This {\DAE} of index 2, consists of $11$ differential equations and $4$ algebraic equations and originates from electrical circuit analysis, describing the behavior of a ring modulator. For more details, see \cite{Taihei19}.

In this {\DAE}, the prolongation order $\bm{c}$ is a zero vector, and the highest derivative of variables $\bm{d}=(1,1,\bm 0_{1\times 4},\bm 1_{1\times 9})$, which means $\bm{F}=\bm{F}^{(\bm{c})}$  and $\delta=11$, $n=15$. The system's Jacobian matrix, has constant rank $\rank \bm{\Jac}= r =\rank \bm{\Jac}[(1:2,4:15),(1:10,12:15)]= 14$. That means $\bm{f(s,y,z)}=\{F_{1}, F_2, F_4,\cdots F_{15}\}$, and $\hat{\bm F}$ consists of $15$ new equations in which the highest derivative of variables $\{{x}_{1},\cdots,{x}_{10},{x}_{12},\cdots,{x}_{15}\}$ are replaced by $\{u_1, \cdots, u_{14}\}$ respectively, and $\dot{x}_{11}$ is replaced by a random constant $\xi$. Further, we can get $\bm{F}^{aug}=\{\bm{f(s,y,z)},\hat{\bm F}\}$ with $29$ variables and a non-singular Jacobian matrix.  By Lemma {\ref{lem:lifting}}, the optimal value of the new {\DAE} is $\bar{\delta} = \delta-n+r = 11-15+14$.

Note that although some highest derivatives of variables are zero, such as the highest derivatives of $x_{3},x_{4},x_{5},x_{6}$, they also need to be replaced. Numerical results are shown in Figure \ref{fig:ring}.

\begin{figure}[htbp]
\centering
\subfigure[Pryce method]{
\includegraphics[width=0.45\textwidth,height=0.30\textwidth]{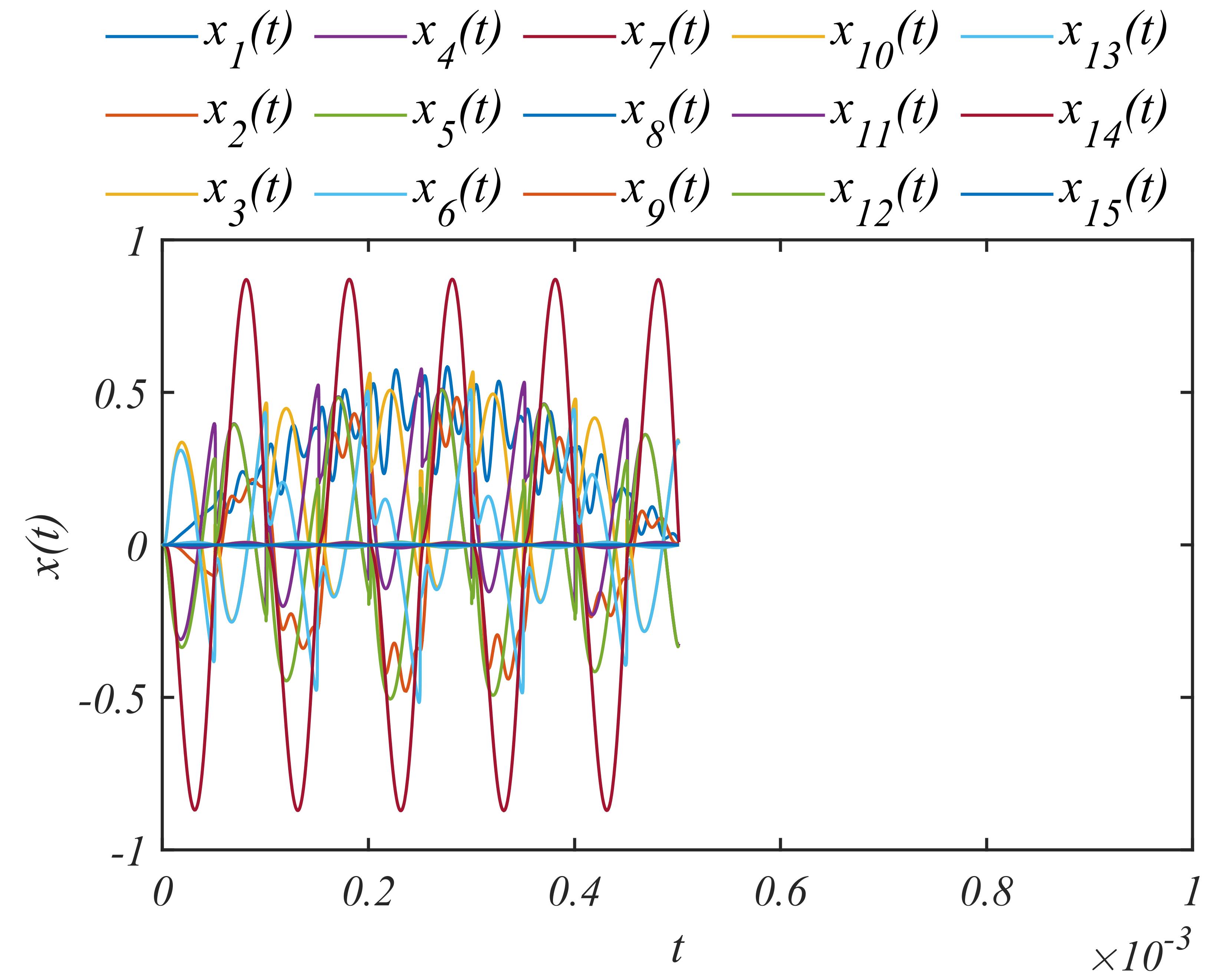}
%\caption{fig1}
}
\quad
\subfigure[substitution]{
\includegraphics[width=0.45\textwidth,height=0.30\textwidth]{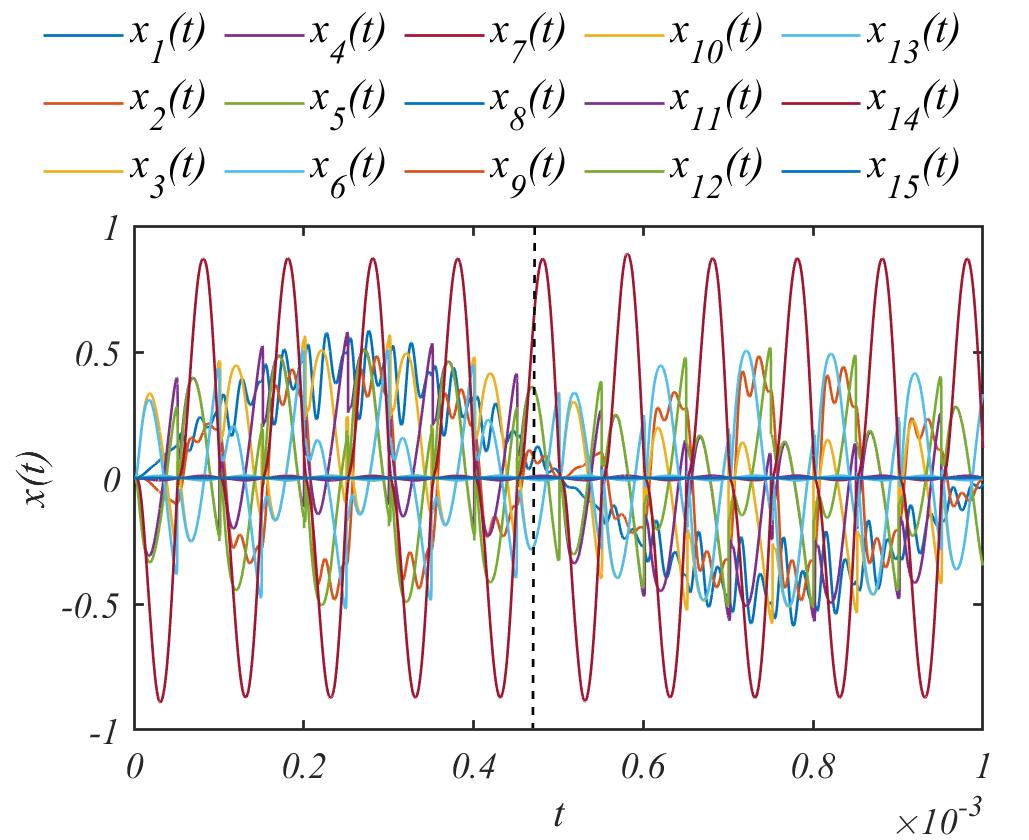}
}
\quad
\subfigure[augmentation]{
\includegraphics[width=0.45\textwidth,height=0.30\textwidth]{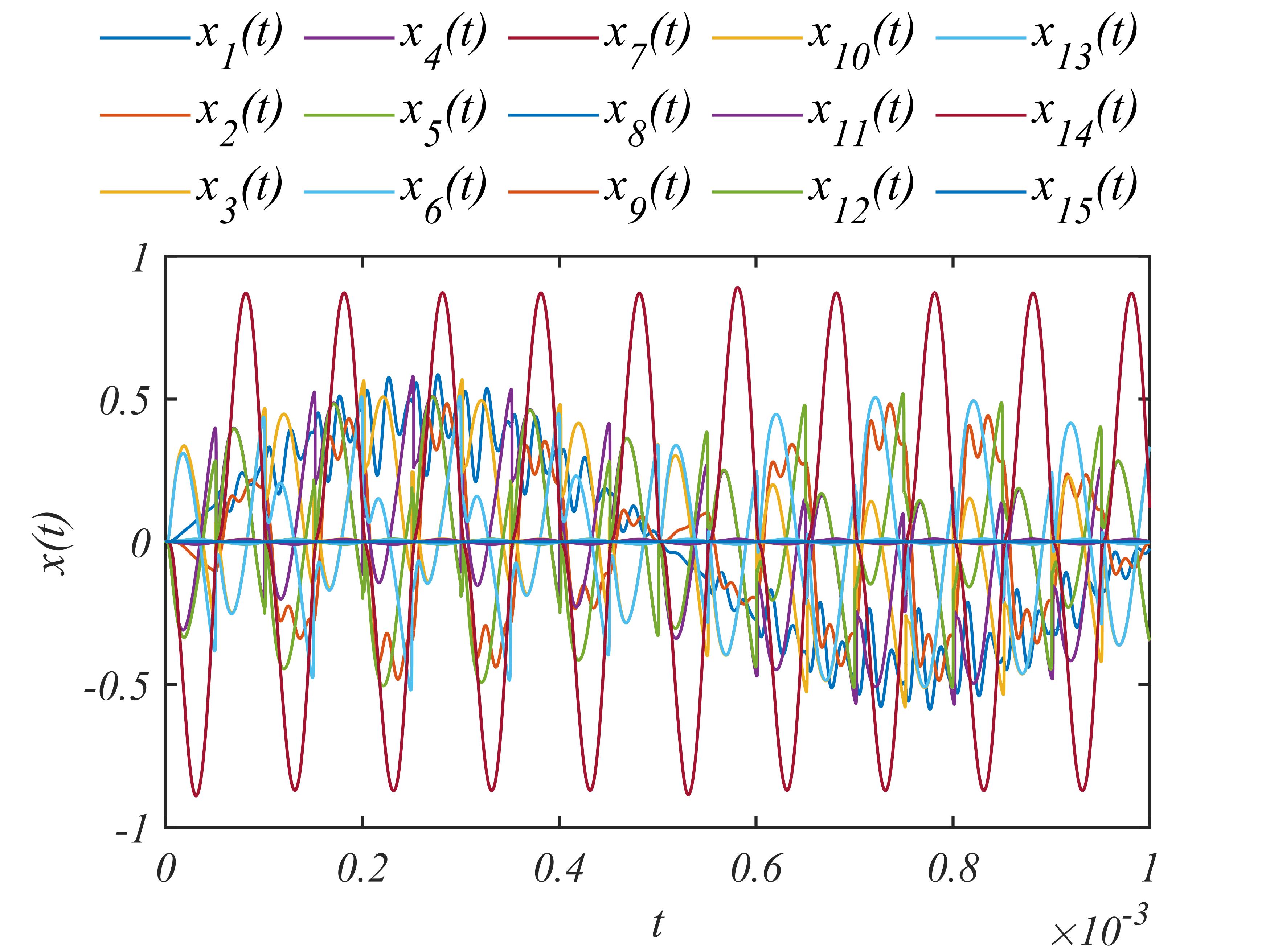}
}
\quad
\subfigure[IRE]{
\includegraphics[width=0.45\textwidth,height=0.30\textwidth]{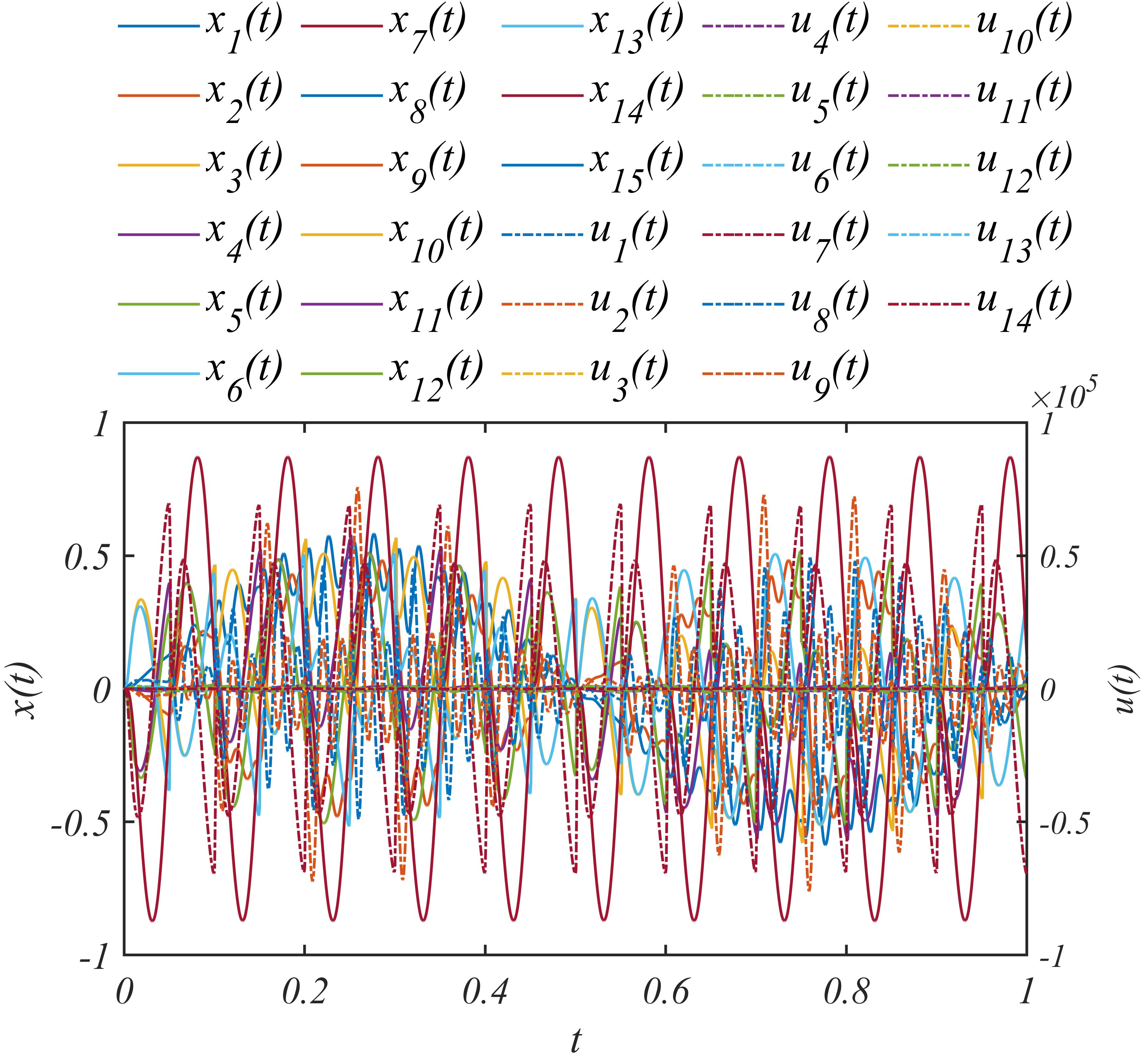}
}
\caption{Numerical Solution of Ring Modulator}\label{fig:ring}
\end{figure}

\subsection{Example \ref{ex:4} (index-$2$)}

As shown in Example \ref{ex:4}, this {\DAE} is an example of numerical degeneration.  The exact solution of this {\DAE} is $x(t)=C-\cos(t)$ and $y(t)=x(t)^2$. The IRE method is essential to address the difficulties posed by numerical degeneration
for this example.

\[\bm{F}^{aug}=\left\{\begin{array}{rcc}
2\cdot u_{1}\cdot y-\xi\cdot x+2x
\left({\frac {{\rm d}x}{{\rm d}t}} \right)^{2} - {\frac {{\rm d}x}{{\rm d}t}} +\sin(t)&=&0\\
\xi-2\cdot  u_{1}\cdot  x-2\cdot\left({\frac {{\rm d}x}{{\rm d}t}} \right)^{2}&=&0\\
2\,y {\frac {{\rm d^{2}}x}{{\rm d}t^{2}}}  - x
  {\frac {{\rm d^{2}}y }{{\rm d}t^{2}}} +2x
 \left({\frac {{\rm d}x}{{\rm d}t}} \right)^{2} - {\frac {{\rm d} x }{{\rm d}t}} +\sin \left( t \right)&=&0
\end{array}\right.\]
After the IRE method, the new Jacobian matrix is
$$\bm{\Jac} =
      \left(
        \begin{array}{ccc}
          4x\cdot x_t -1 & 0 & 2y \\
          -4x_t & 0& -2x \\
         2y  &-1&0\\
        \end{array}
      \right)$$
It is obvious that the determinant of the new Jacobian matrix will not degenerate to a singular matrix by virtue of the constraints.  Numerical results for $C=2$ are shown in Figure \ref{fig:ex}.

 \begin{figure}[htbp]
\centering
\subfigure[Pryce method]{
\includegraphics[width=0.45\textwidth,height=0.30\textwidth]{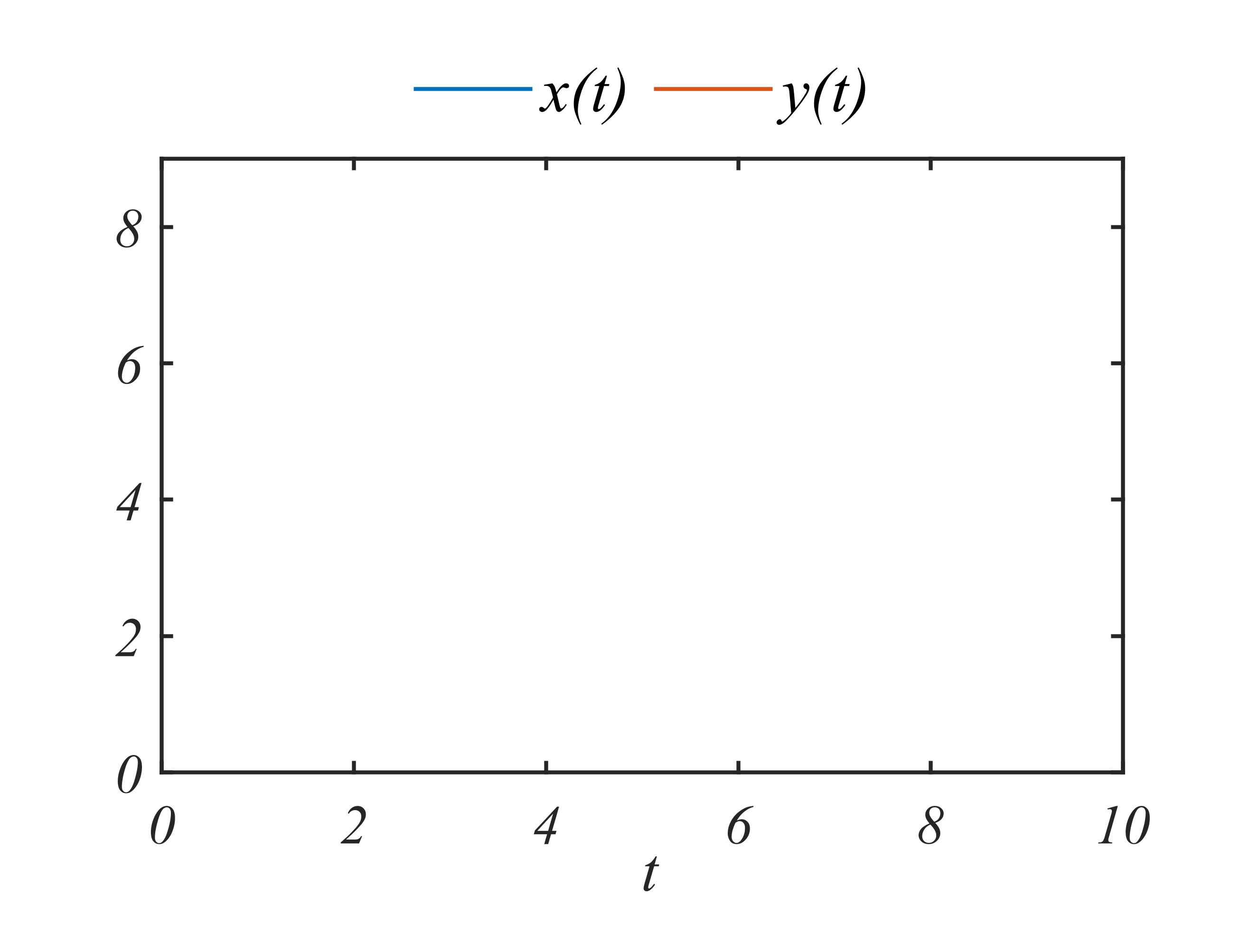}
%\caption{fig1}
}
\quad
\subfigure[substitution]{
\includegraphics[width=0.45\textwidth,height=0.30\textwidth]{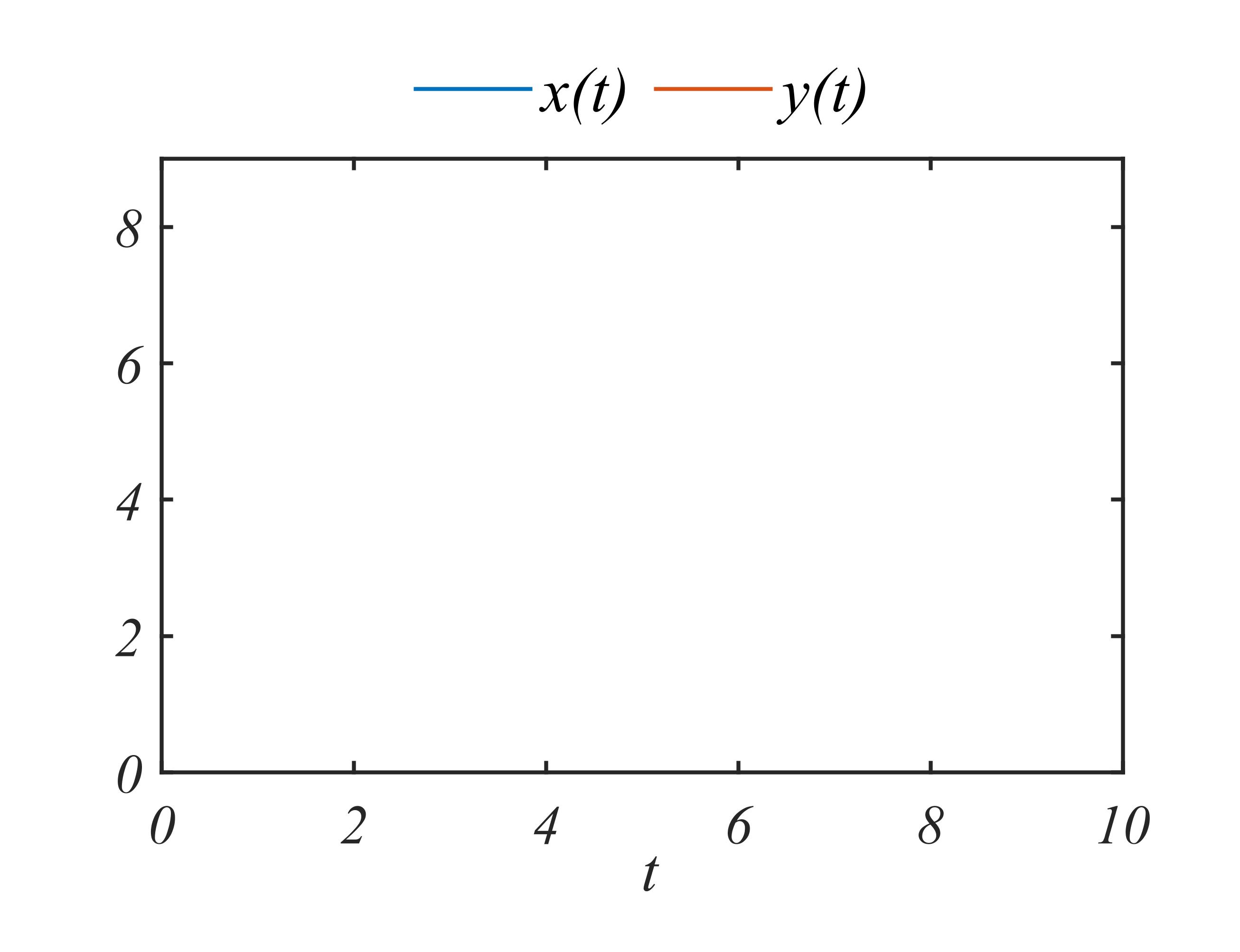}
}
\quad
\subfigure[augmentation]{
\includegraphics[width=0.45\textwidth,height=0.30\textwidth]{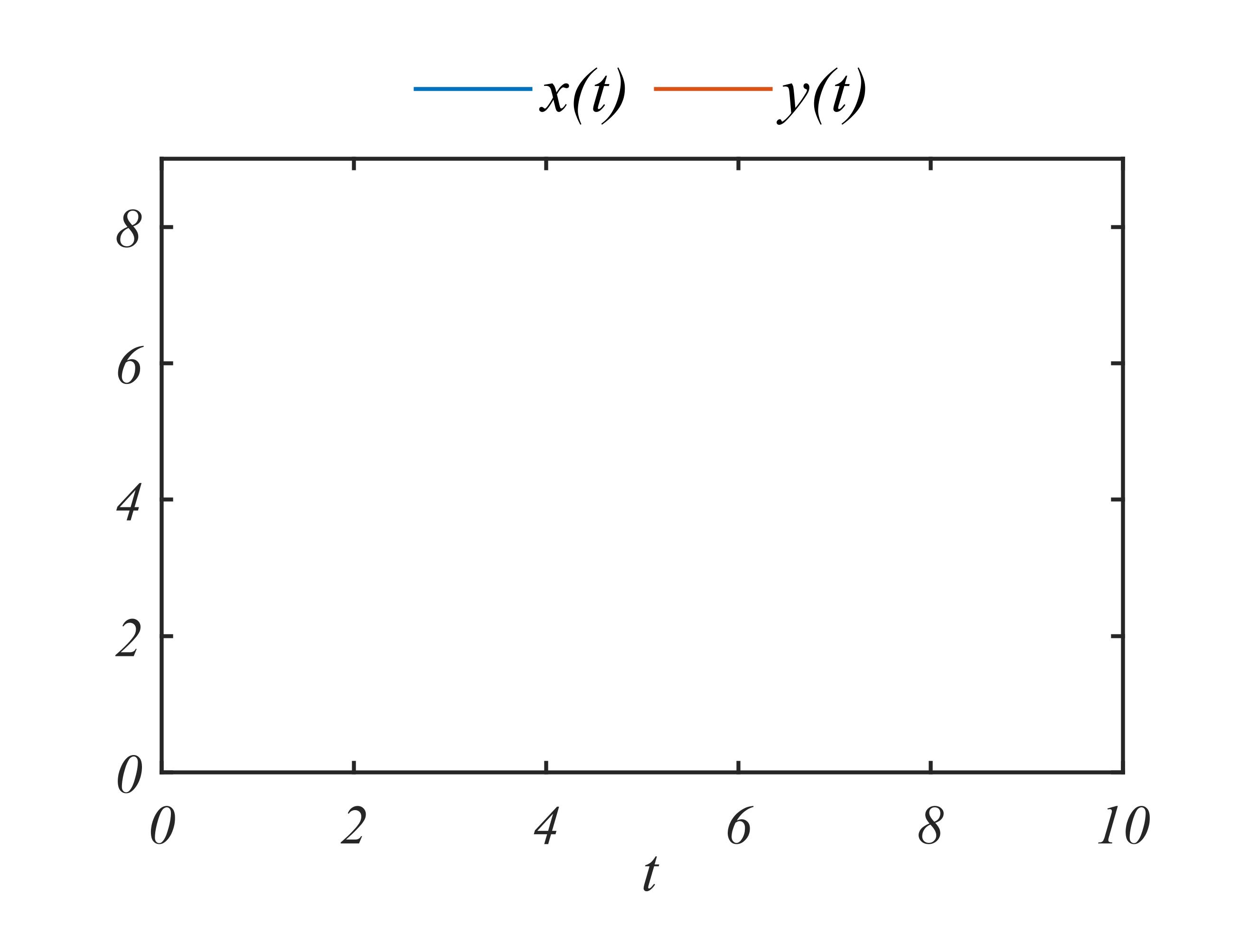}
}
\quad
\subfigure[IRE]{
\includegraphics[width=0.40\textwidth,height=0.30\textwidth]{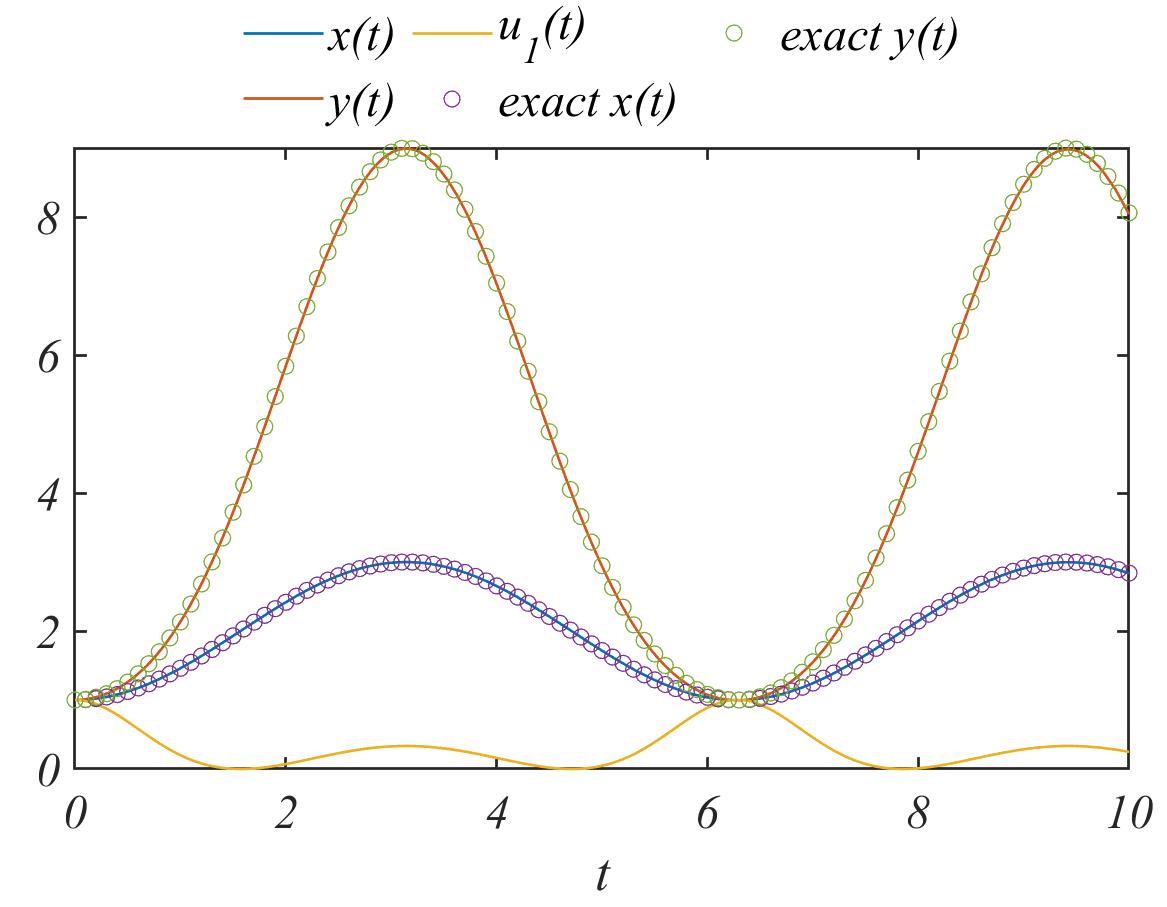}
}
\caption{Numerical Solution of Example \ref{ex:4}}\label{fig:ex}
\end{figure}

\subsection{Analysis of Bending Deformation of Beam (index-$2$)}
The specific description is given in Example \ref{ex:3}. In this example, when the elastic deformation energies of forces  are the same, we can set $\lambda = 1$. By structural analysis, the optimal solutions is $\bm{c}=(0,2)$ and $\bm{d}=(2,2)$.

 This non-linear {\DAE} has two components resulting from its constraints: one component results from $y_{1}=y_{2}$, the other component results from $y_{1}=-y_{2}$. In detail, two witness points
 are computed by the Homotopy continuation method \cite{WWX2017} where each point has
 coordinates $(y_{1}, y_{2}, \dot{y}_{1},\dot{y}_{2})$:
\[
\begin{array}{rrrrrr}
  ( & -0.43092053722 &  -0.43092060160 &  -0.27565041470 & -0.27565030340 & ) \\
  ( & -0.19993949748 &  +0.19993723792 &  +0.64332968577 & -0.64333747822 & ) \\
\end{array}
\]
 % {\ie} one is \{$y_{1}=-0.43092053722$, $y_{2}=-0.43092060160$, $\dot{y}_{1}=-0.27565041470$, $\dot{y}_{2}=-0.27565030340$\} and the other is\{$y_{1}=-0.19993949748$, $y_{2}=0.19993723792$, $\dot{y}_{1}=0.64332968577$, $\dot{y}_{2}=-0.64333747822$\}.
 Because the Jacobian matrix of the polynomial constraints is singular here, a large penalty factor should be introduced in order to improve convergence.  These witness points are approximate points near the consistent initial value points, which and need to be refined by Newton iteration.

 By symbolic computation, we can get two exact solutions of above {\DAE} as
 \begin{eqnarray*}
  y_{1}(x) &=& +y_{2}(x)=C_1\cdot\sin(\frac{\sqrt{2}x}{2})+C_2\cdot\cos(\frac{\sqrt{2}x}{2})-\frac{1}{5}\cdot\sin(x)-\frac{1}{5} \\
   y_{1}(x)&=& -y_{2}(x)=-\frac{1}{5}\cdot(1-\sin(x))
 \end{eqnarray*}
 Here $C_1$ and $C_2$ are constants depending on consistent initial conditions. These exact solutions can be used to check the correctness of our numerical solution of the global structural differentiation method.

 Obviously, the Jacobian matrix is non-singular for any witness point from the component with $y_{1}=y_{2}$. This case can be solved directly after applying the Pryce method as shown in Figure \ref{fig:beam}. On the contrary, for any witness point on the component with $y_{1}=-y_{2}$, the Jacobian matrix will degenerate to a singular matrix. For this case, we have to construct its
equivalent {\DAE}, and its numerical results are shown in Figure \ref{fig:beam1}.

  \begin{figure}[htbp]
\centering
\subfigure[Pryce method]{
\includegraphics[width=0.45\textwidth,height=0.30\textwidth]{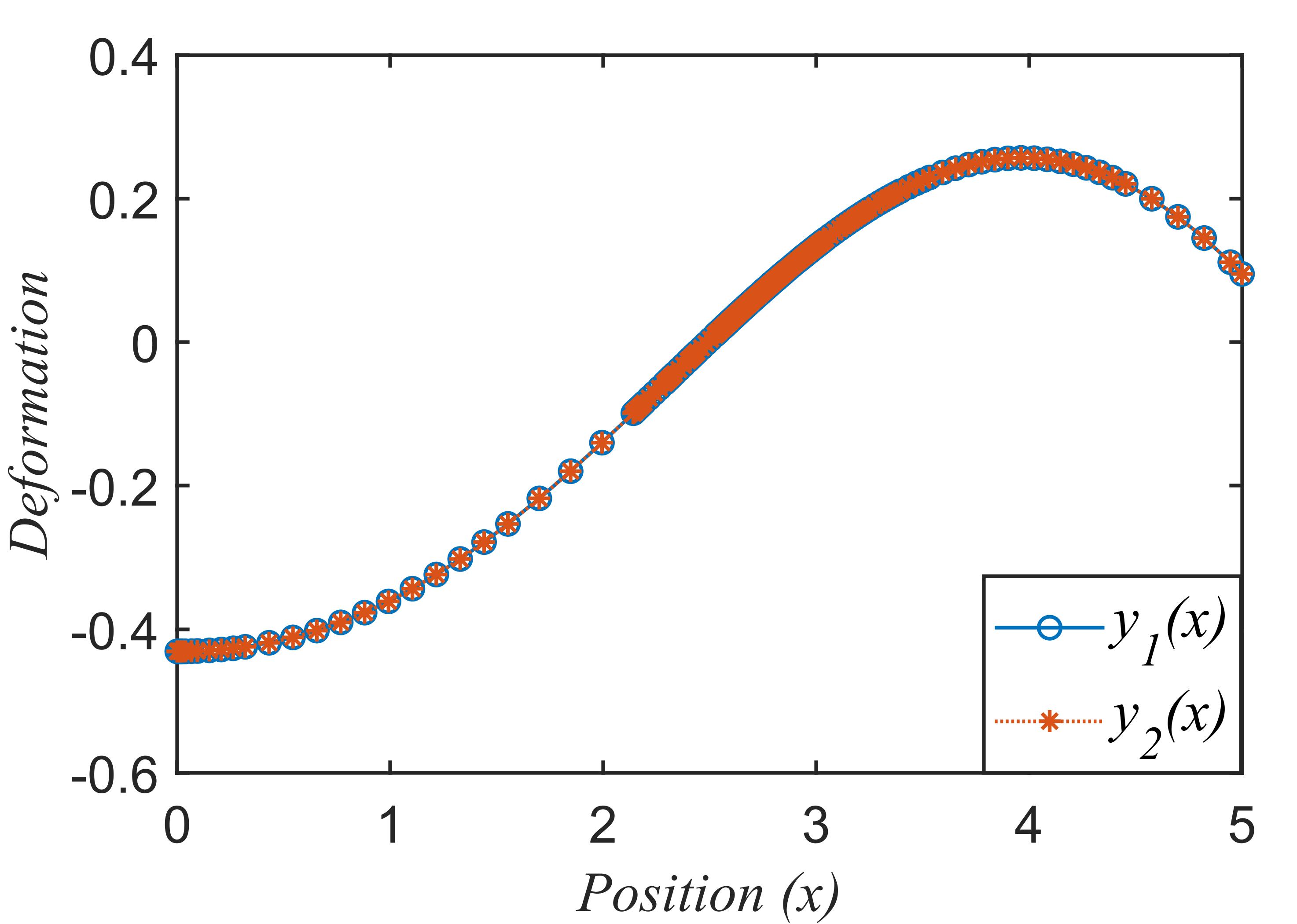}
%\caption{fig1}
}
\quad
\subfigure[substitution]{
\includegraphics[width=0.45\textwidth,height=0.30\textwidth]{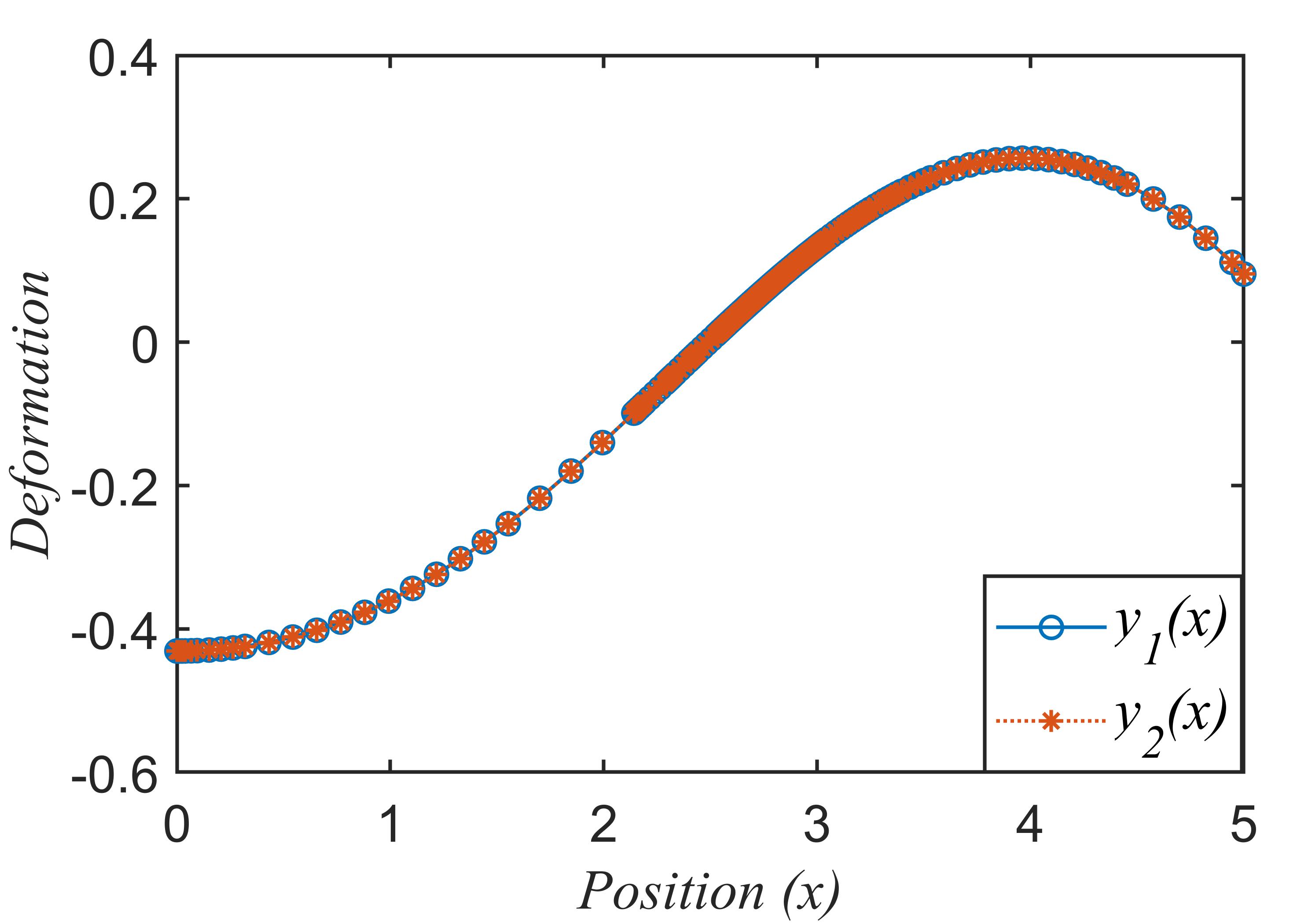}
}
\quad
\subfigure[augmentation]{
\includegraphics[width=0.45\textwidth,height=0.30\textwidth]{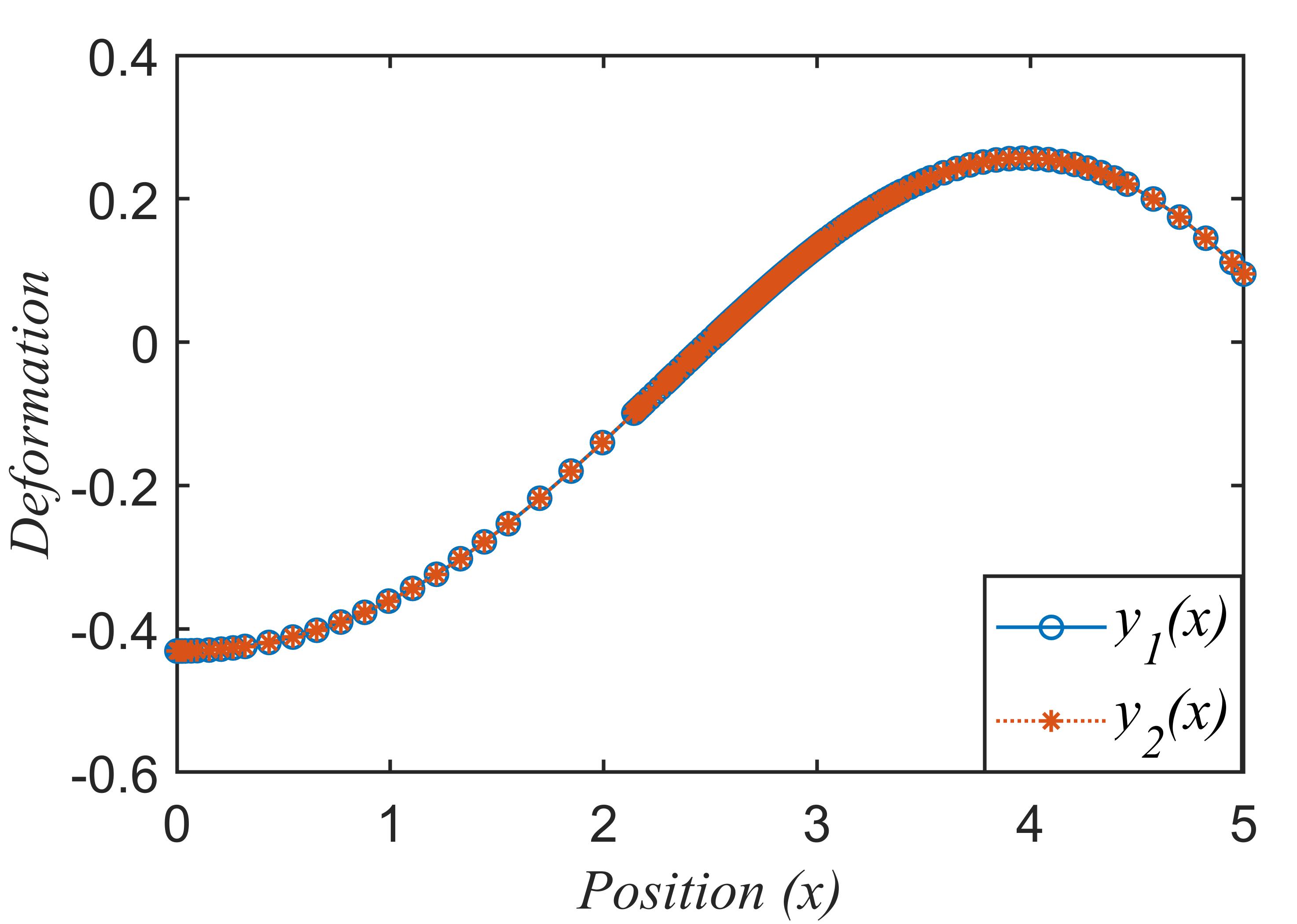}
}
\quad
\subfigure[IRE]{
\includegraphics[width=0.45\textwidth,height=0.30\textwidth]{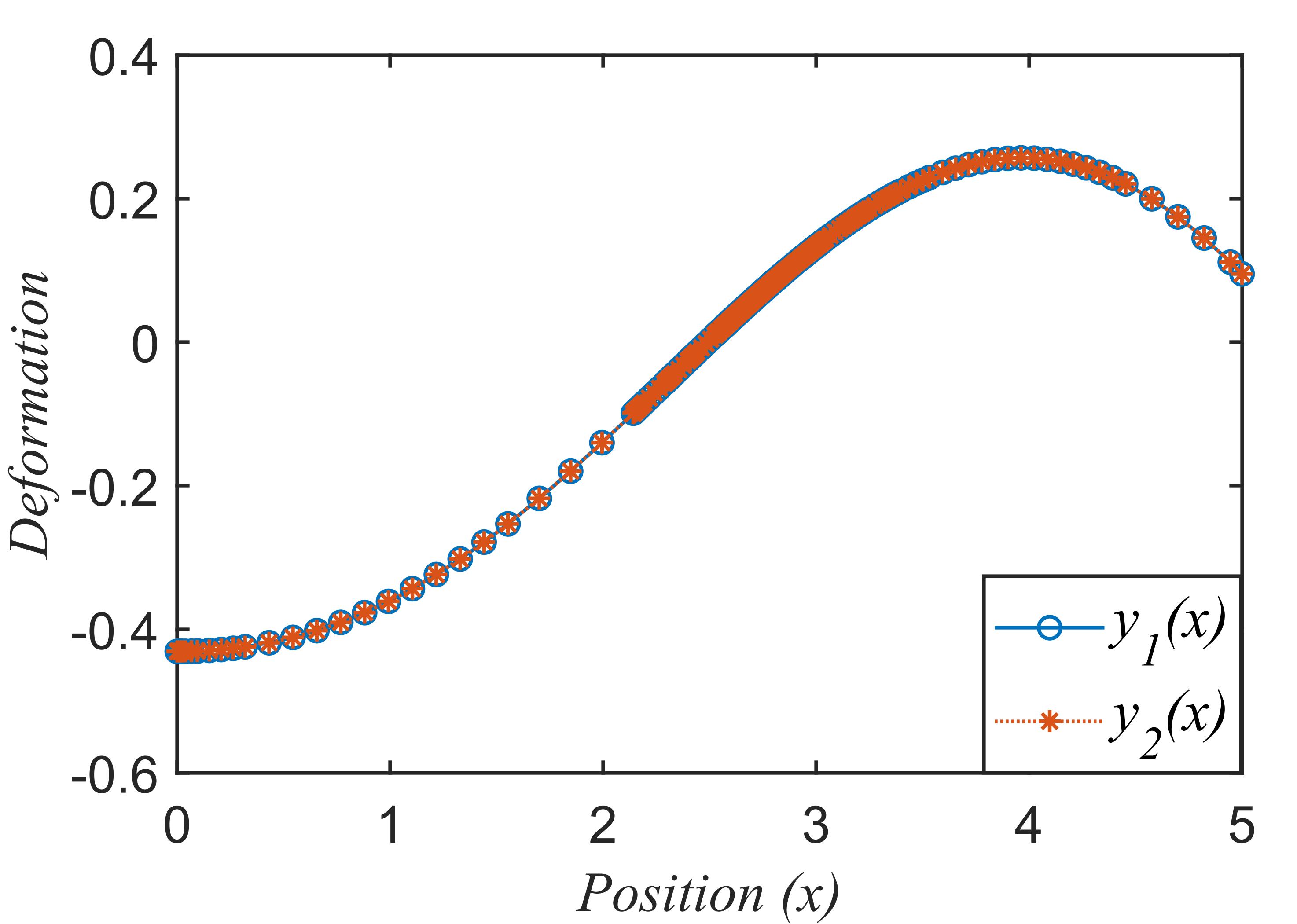}
}
\caption{Numerical Solution of Beam (Nonsingular Component)}\label{fig:beam}
\end{figure}

%Further more, for numerical solution, two witness points are taken on the different path-tracking by the Homotopy continuation method \cite{WWX2017}, {\ie} one is \{$y_{1}=-0.43092053722$, $y_{2}=-0.43092060160$, $\dot{y}_{1}=-0.27565041470$, $\dot{y}_{2}=-0.27565030340$\} and the other is\{$y_{1}=-0.19993949748$, $y_{2}=0.19993723792$, $\dot{y}_{1}=0.64332968577$, $\dot{y}_{2}=-0.64333747822$\}. Here, because the Jacobian matrix of the initial value polynomial system is singular, a huge penalty factor should be introduced in order to improve the correction effect. And these solutions are approximate solutions near the consistent initial values, which needs to be refined by Newton iteration.

   \begin{figure}[htbp]
\centering
\subfigure[Pryce method]{
\includegraphics[width=0.45\textwidth,height=0.30\textwidth]{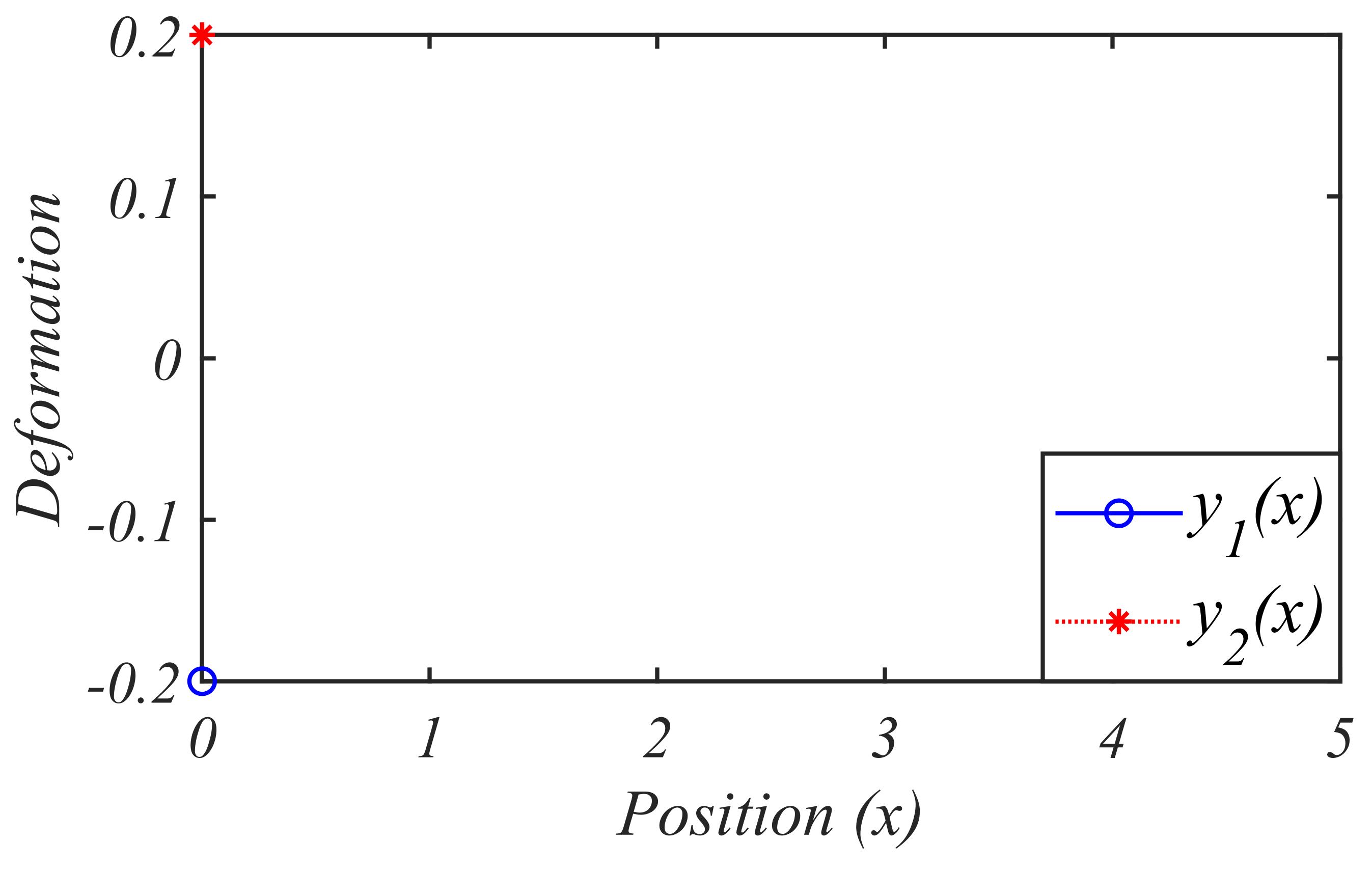}
%\caption{fig1}
}
\quad
\subfigure[substitution]{
\includegraphics[width=0.45\textwidth,height=0.30\textwidth]{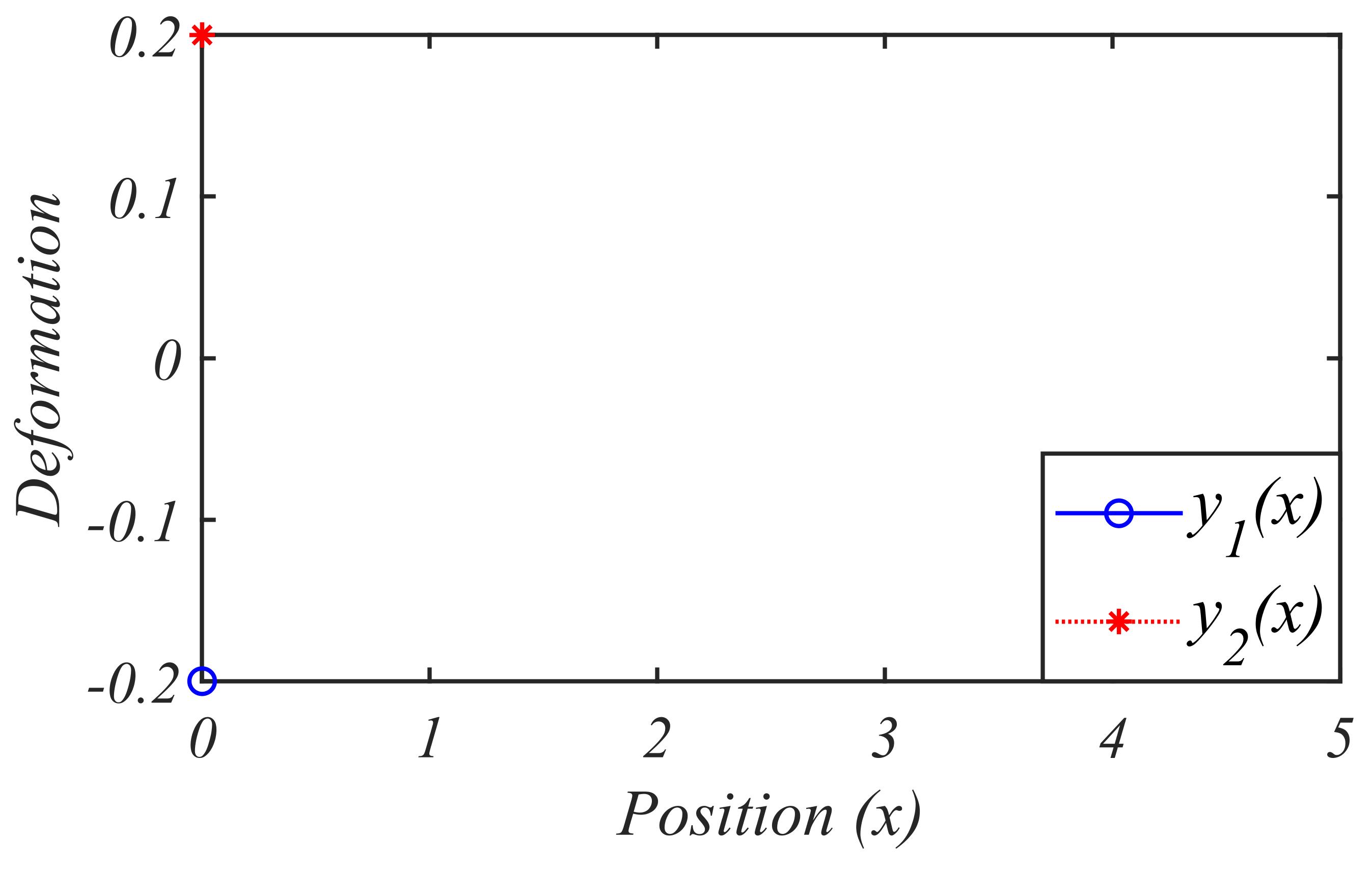}
}
\quad
\subfigure[augmentation]{
\includegraphics[width=0.45\textwidth,height=0.30\textwidth]{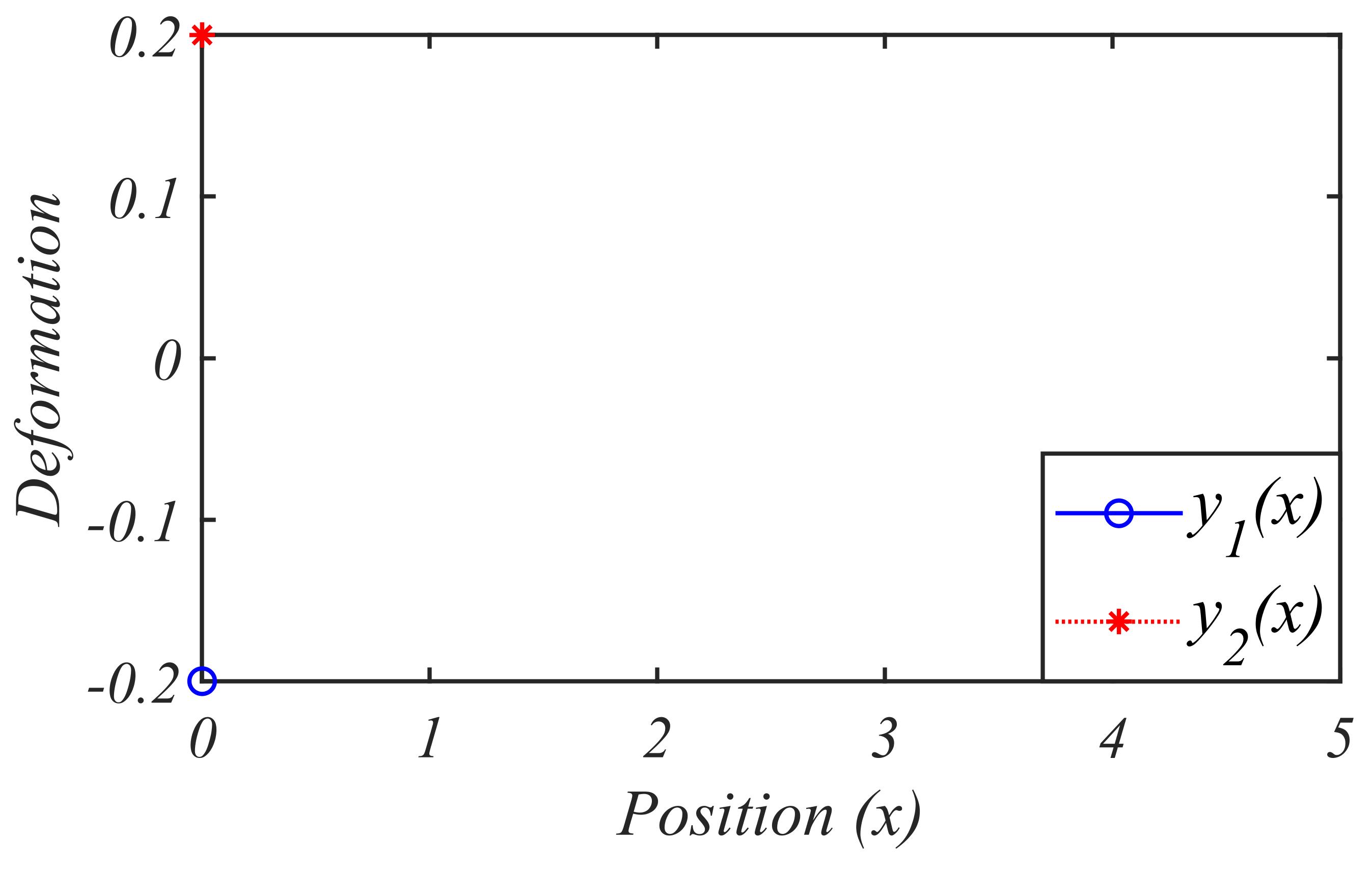}
}
\quad
\subfigure[IRE]{
\includegraphics[width=0.45\textwidth,height=0.30\textwidth]{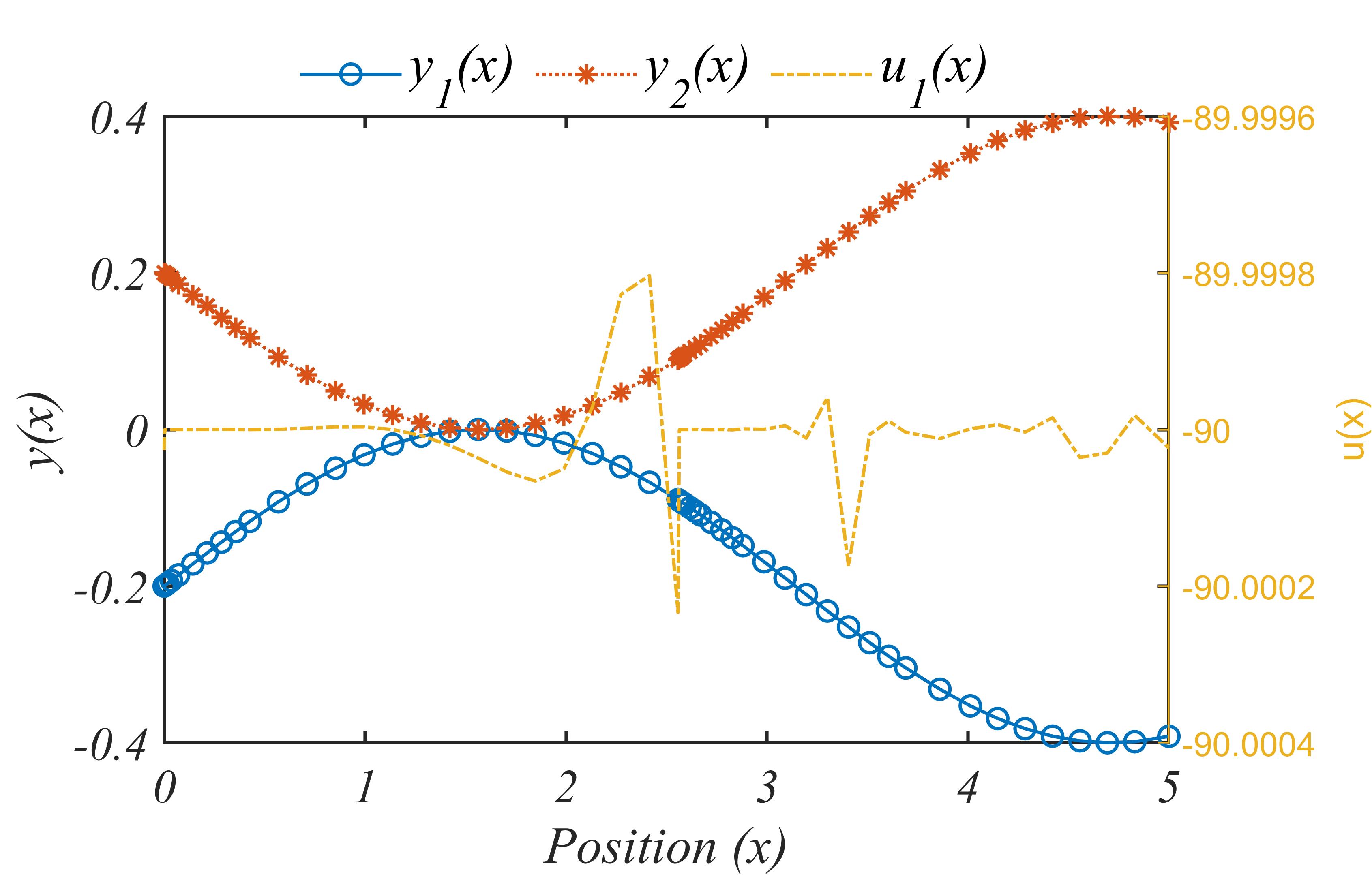}
}
\caption{Numerical Solution of Beam (Singular Component)}\label{fig:beam1}
\end{figure}

\subsection{Result Analysis}\label{ssec:numer}

%Since the IRE method introduces new variables $\bm{u}$ and constants $\bm{\xi}$, in order to reduce iterations and to avoid the occurrence of dynamic pivoting, it is crucial to set reasonable initial values and constants based on the existing equation relationship. %Although random constants $\bm{\xi}$ are still correct in IRE, it will cause a lot of unnecessary calculation of initial value points and violent fluctuation of $\bm{u}$. The best way is to deduce the initial value of the new variable through the relationship of {\DAE}.

We used MATLAB's ode$15$i as the ode solver of Algorithm $5$ for numerical solution of the examples in our paper.

By comparing Figures {\ref{fig:amp}}, {\ref{fig:pen}}, {\ref{fig:ring}}, it can be seen that the substitution method and the augmentation method are effective in dealing with  symbolic cancellation {\DAE}s, as well as the IRE method. Without index reduction, ode$15$i only works well when the index is low and its solutions of high index {\DAE}s become unstable.

In Figure \ref{fig:ex}, although this {\DAE} is low-index, the former three methods all fail at time $t=3.952847\times 10^{-4}$, as they cannot detect that the Jacobian matrix has been constrained to be singular. Furthermore, the Homotopy continuation method used in the IRE method helps to detect numerical degeneration by computing a constant rank of Jacobian matrix at witness points.

By the Homotopy continuation method, all possible consistent initial value paths can be tracked, for the two components in the bending deformation of beam (see Figures {\ref{fig:beam}}, {\ref{fig:beam1}}, respectively).
Thus, by the IRE method, structural information of each path can be obtained separately, and all real solutions of the {\DAE} can be approximated.

\section{Two Types of Challenge {\DAE}s for Structural s}\label{sec:spec}
\sloppy{}

Beside the degradation of the Jacobian matrix in the previous sections, the Pryce
method will also fail in dealing with unreduced models. Such
models may be due to unreduced descriptions of {\DAE}s in the modeling process,  e.g.  {\DAE}s with mixed
signature matrix or high multiplicity. Next, we will discuss how to apply the IRE method to solve
such {\DAE}s.

\subsection{Linear Recombination}\label{ssec:linear}
Here mixed signature matrix means that all rows of the signature matrix are exactly same. To produce this case, the original {\DAE} is multiplied by a non-singular constant matrix. Its structural information is hidden and it causes trouble for the Pryce method.
This type of {\DAE} belongs to the case of symbolic cancellation. Obviously, in theory, the IRE method can deal with this kind of case well. However, due to the missing structural information, it is necessary to call the IRE method several times.

Consider example \ref{ex:4}.  Suppose there is a matrix $\bm{A}=(1, 0; 1, 1)$, and the new {\DAE} is $\bm{A}\cdot \bm{F}$. Its structural information by the Pryce method is $\bm{c}=(0,0)$, $\bm{d}=(2,2)$. Compared with Section \ref{ssec:exam1} where the {\DAE} is missing $1$ hidden constraint equation, the new {\DAE} is missing $2$ additional hidden constraints caused by structural method failure.

Then we need to make additional calls of the IRE method to find missing hidden constraints.
In the first call, the size $n=2$ and rank $r=1$, so we can only find $n-r=1$ additional hidden constraint equations.
Then we need a second call, which yields size $n=3$ and rank $r=2$.  So the remaining $n-r=1$ additional hidden constraint equation has also been found.  However, because the original {\DAE} is missing $1$ hidden constraint equation, as in Section \ref{ssec:exam1}, a third call is necessary.  That yields size $n=5$ and rank $r=4$, and $n-r=1$ hidden constraint equation is found, yielding a full rank Jacobian.

Finally, the numerical solution in Figure \ref{fig:relinear} shows that the structure information of the new {\DAE} by the IRE method is correct and reliable.
Futhermore, it is easy to deduce that linear recombination cases can be handled well by IRE method.

\begin{figure}[htpb]
	\centering
	\includegraphics[height=6cm,width=8cm]{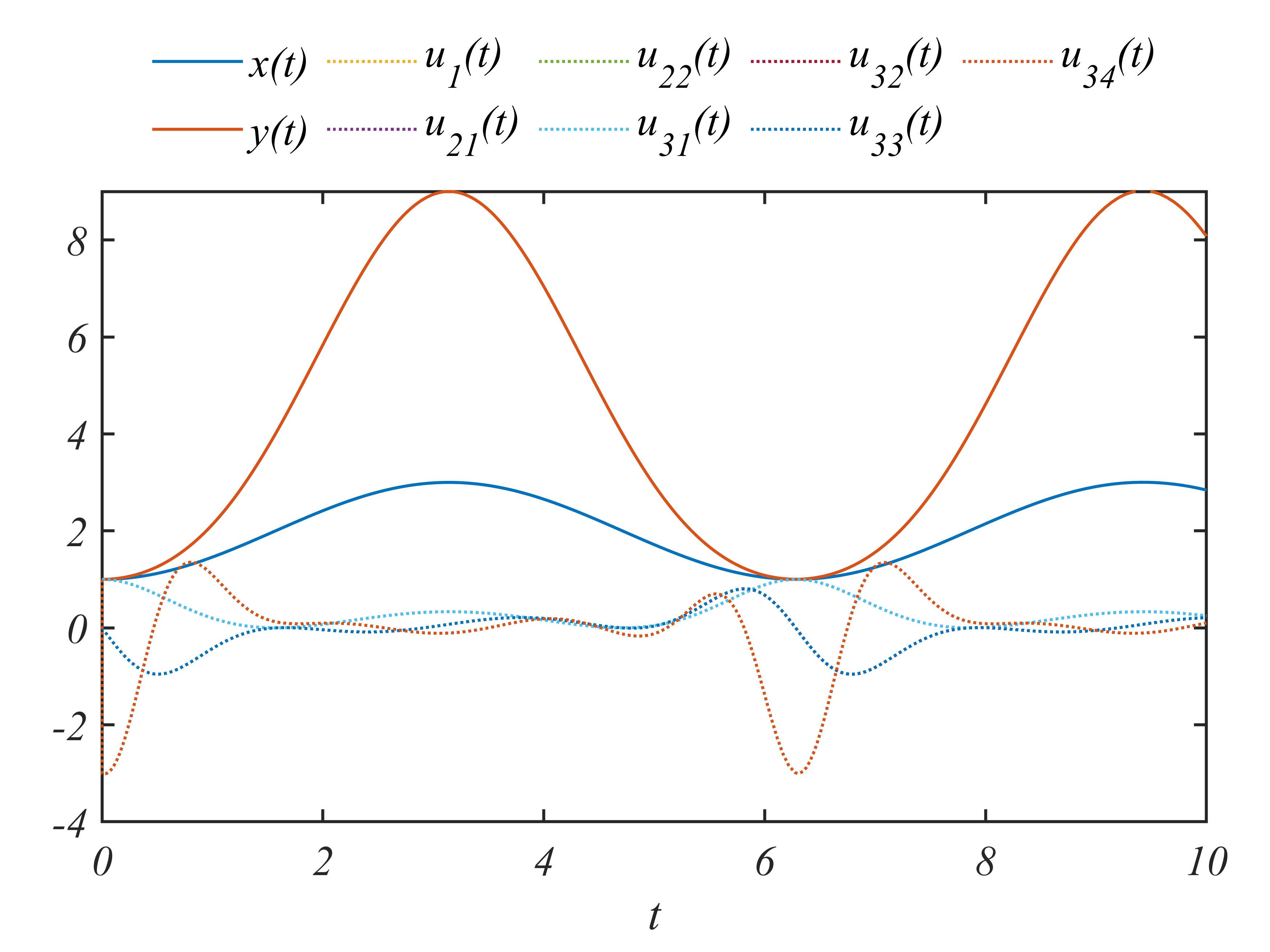}
	\caption{Numerical Solution of Linear Recombination Case from Example \ref{ex:4} }
	\label{fig:relinear}
\end{figure}

\subsection{{\DAE}s with High Multiplicity}\label{ssec:sqare}
A simple way to produce this case is that some equations of a {\DAE} are squared. It will also lead to the singularity of Jacobian matrix.
The determinant of the Jacobian is nonzero with a factor which is the constraint.  This is the case of numerical degeneration, and common factors lead to
redundancy.

Actually, due to the existence of redundant equations, Problem (\ref{LPP}) and Definition \ref{define_delta1} must be reformulated, which will lead to the invalidation of $\delta(\bm{G}) \leq  \delta(\bm{F}) - (n-r)$ in Theorem \ref{thm:result}.  But this is beyond the scope of this paper.  Instead we will present some interesting observations concerning the IRE method.

Note that for high multiplicity case, Theorem \ref{thm:result} can still guarantee the equivalence of {\DAE} before and after application of the IRE method.

  %Furthermore, $\delta(\bm{G}) \leq  \delta(\bm{F}) - (n-r)$ also is equivalent to finding at least $n-r$ hidden constraints by calling the IRE method once.

Consider Example \ref{ex:4}, and suppose the constraint equation is replaced by $(y \left( t \right) -x^2 \left( t \right))^2$, whose  structural information is $\bm{c}=(0,2)$, $\bm{D}=(2,2)$.  Apparently the number of equations is $4$, the number of variables is $6$, and the optimal value is $6-4=2$.  But in fact, the rank of its equations is $2$.  Compared with Section \ref{ssec:exam1} where \ref{ex:4} is missing $1$ hidden constraint equation, the new {\DAE} is missing $4-2=2$ additional hidden constraint equations.

In a similar manner to Section \ref{ssec:linear}, we make multiple calls of IRE method to find hidden constraints of the new {\DAE}.  The first call of IRE yields $1$ hidden constraint equation of the new {\DAE}, with size $n=2$ and rank $r=1$.
In the second call, we also found $1$ hidden constraint equation of the new {\DAE}, with size $n=3$ and rank $r=2$.
However, because the original {\DAE} is missing $1$ hidden constraint equation, as in Section (\ref{ssec:exam1}), a third call is necessary for the new {\DAE}, and yields size $n=5$ and rank $r=4$. Finally, the last $1$ hidden constraint equation is found, and the final Jacobian is non-singular. After the structural method is applied, the numerical solution is shown in Figure \ref{fig:square}.

 \begin{figure}[htpb]
 	\centering
 	\includegraphics[height=6cm,width=8cm]{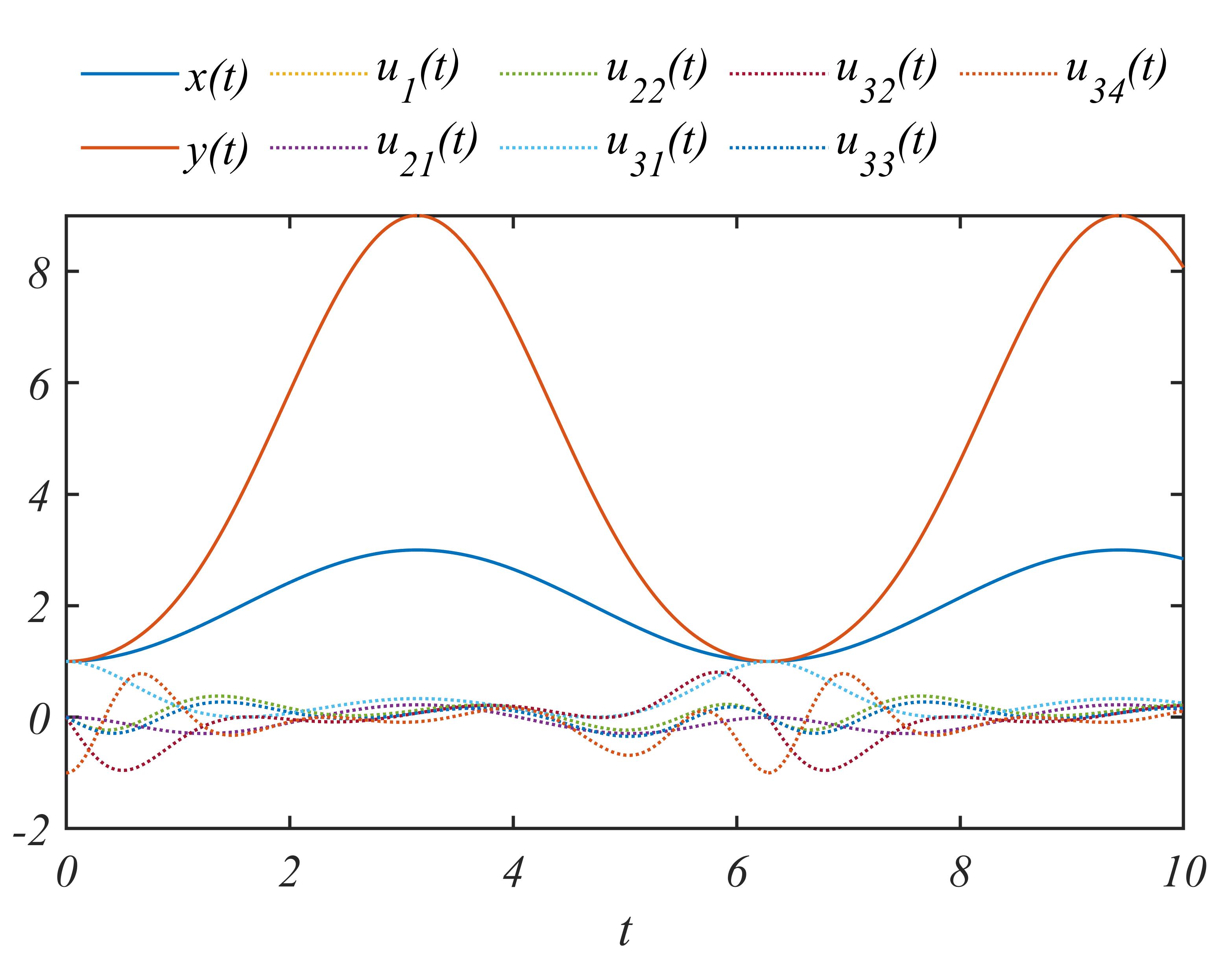}
 	\caption{Numerical Solution of Square Case from Example \ref{ex:4} }
 	\label{fig:square}
 \end{figure}

In this example, the IRE method can also find $n-r$ hidden constraint equations at once. In other words, IRE method has the possibility of regularizing the structure method, but further research is needed. From the perspective of hidden  constraints, we can redefine the optimal value as $\#var-rank(\bm{F})$, where $\#var$ is the number of variables. In Theorem \ref{thm:result}, we conjecture that a more general theorem can be obtained by replacing $\delta(\bm{G}) \leq  \delta(\bm{F}) - (n-r)$ by finding at least $n-r$ hidden constraints.

\section{Conclusions}\label{sec:con}
\sloppy{}

\begin{table}[htpb]
	\caption{Comparison of Experimental Results  }\label{tab:result}
 \#C = \#Components, \hskip2pt $\surd$ = success, \hskip2pt $\times$ = failure, \hskip2pt * = possible failure depending on Jacobian matrix, \hskip2pt Sub = Substitution, \hskip2pt Aug = Augmentation, \hskip2pt DE = Degeneration, \hskip2pt SC = symbolic cancellation \hskip2pt ND = numerical degeneration.
\centering
\begin{tabular}{|c|c|c|p{50pt}|c|c|c|c|p{20pt}|}
  \hline
  Index & Structure & \#C & Examples & ode$15$i & Sub & Aug & IRE & DE\\
  \hline
  low & linear& $1$ & Transistor Amplifier {\cite{Taihei19}} &$\surd$ & $\surd$ &   $\surd$ &  $\surd$ & SC\\
   \hline
  high & non-linear& $1$ &Modified  Pendulum {\cite{Taihei19}} & $\times$& $\surd$ &  $\surd$ &   $\surd$& SC\\
   \hline
  high & linear& $1$ & Ring Modulator {\cite{Taihei19}} & $\times$&   $\surd$ &   $\surd$ &   $\surd$& SC\\ \hline
  low & linear& $1$ &Example 1  &  $\times$& $\times$ & $\times$ &   $\surd$ & ND\\ \hline
  high & non-linear& $2$ &Beam & * & * & * &  $\surd$ & ND\\
  \hline
\end{tabular}
\end{table}

In this paper, we first gave a framework for improved structural methods in Section \ref{ssec:improved}.
In Section \ref{sec:implictit method} we proposed an improved structural method --- the IRE method ---  based on witness point techniques described in Section \ref{sec:det}.

The IRE method avoids the direct elimination of non-linear {\DAE}s in other improved structural methods by introducing new variables and equations to increase the dimensions of space in which the {\DAE} resides. The IRE method is efficient and intuitive, and enables the simultaneous regularization of all the equations of a  {\DAE}, rather than one specific equation at a time. The more rank deficiency, the higher efficiency, but the scale of the equation will also increase.  A strong feature of our approach, is that Homotopy continuation methods can be naturally and efficiently combined with the IRE method, which can help to deal with almost all degeneration issues for {\DAE}s.  Unlike the local equivalence methods, such as the substitution method and the augmentation method, the IRE method is proved to be a global equivalence method in Lemma \ref{lem:proj}.

To better demonstrate our methods, we describe specific algorithms in Section \ref{sec:alg}, and give $5$ numerical examples in Section \ref{sec:ex}.  The experimental results are summarized in Table \ref{tab:result}, which show that global structural differentiation method based on the IRE method can deal with symbolic cancellation {\DAE}s and numerically degenerated {\DAE}s whether these {\DAE}s are high-index or not.

In Section {\ref{sec:spec}} two challenging special cases are discussed whose structural information is wrong after application of structural methods.  In particular Section \ref{ssec:linear} considers the problem of linear
recommbinations and a modification of the IRE method is given which addresses this case.
Section \ref{ssec:sqare} considers the problem of {\DAE}s with high multiplicity.
 For this case, the IRE method still works on examples, but the approach lacks a theoretical justification, which is a problem for future research.

Although the IRE method performs well, it may fail when dealing with {\DAE}s with transcendental equations or strong non-linearity in applications. This is due to the limitation of the Homotopy method in solving the constraints of these {\DAE}s in which we can not find witness points on every component.  Indeed such problems may have infinitely many components, unlike polynomially nonlinear {\DAE}. Here, if a consistent initial point can be obtained by numerical iteration, the global structural differentiation method based on the IRE method can still give some solutions of the {\DAE}.

\section*{Acknowledgments}
We would like to acknowledge the assistance of  Taihei Oki in program codes.

\bibliographystyle{siamplain}
\bibliography{references}

\begin{thebibliography}{10}

\bibitem{Awawdeh09}
{\sc F.~Awawdeh, H.~Jaradat, and O.~Alsayyed}, {\em {Solving System of DAEs by
  Homotopy Analysis Method}}, Chaos, Solitons and Fractals, 42 (2009),
  pp.~1422--1427.

\bibitem{BHSW13}
{\sc D.~J. Bates, A.~J. Sommese, J.~D. Hauenstein, and C.~W. Wampler}, {\em
  Numerically Solving Polynomial Systems with Bertini}, Society for Industrial
  and Applied Mathematics, Philadelphia, PA, 2013.

\bibitem{Brenan95}
{\sc K.~E. Brenan, S.~L. Campbell, and L.~R. Petzold}, {\em Numerical Solution
  of Initial-Value Problems in Differential-Algebraic Equations}, Society for
  Industrial and Applied Mathematics, 1995.

\bibitem{Brown98}
{\sc P.~N. Brown, A.~C. Hindmarsh, and L.~R. Petzold}, {\em Consistent initial
  condition calculation for differential-algebraic systems}, SIAM Journal on
  Scientific Computing, 19 (1998), pp.~1495--1512.

\bibitem{Collins1982}
{\sc G.~E. Buchberger, B.~Collins and R.~Loos}, {\em Computer Algebra: Symbolic
  and Algebraic Computation}, Springer-Verlag Vienna, Vienna, 1982.

\bibitem{Caillaud20}
{\sc B.~Caillaud, M.~Malandain, and J.~Thibault}, {\em {Demo: IsamDAE, an
  Implicit Structural Analysis Tool for Multimode DAE Systems}}.
\newblock {HSCC 2020 - 23rd ACM International Conference on Hybrid Systems:
  Computation and Control}, Apr 2020.

\bibitem{Campbell93}
{\sc S.~L. Campbell}, {\em Least squares completions for nonlinear differential
  algebraic equations}, Numer. Math., 65 (1993), p.~77–94.

\bibitem{Campbell1995}
{\sc S.~L. Campbell}, {\em {High}-{Index Differential Algebraic Equations}},
  Mechanics of Structures and Machines, 23 (1995), pp.~199--222.

\bibitem{Campbell95}
{\sc S.~L. Campbell and C.~W. Gear}, {\em {The Index of General Nonlinear
  DAEs}}, Numerische Mathematik, 72 (1995), pp.~173--196.

\bibitem{CoxLittleOshea07}
{\sc D.~O. David A.~Cox, John~Little}, {\em Ideals, Varieties, and Algorithms},
  Springer-Verlag Vienna, 2007.

\bibitem{Hairer91}
{\sc G.~W. Ernst~Hairer}, {\em Solving Ordinary Differential Equations II},
  vol.~14, Springer-Verlag, Berlin Heidelberg, 1991.

\bibitem{Fritzson14}
{\sc P.~Fritzson}, {\em Principles of Object Oriented Modeling and Simulation
  with Modelica 3.3 (A Cyber-Physical Approach)}, Wiley-IEEE Press, Hoboken,
  2014, ch.~17, pp.~977--991.

\bibitem{Gear88}
{\sc C.~W. Gear}, {\em {Differential}-{Algebraic Equation Index
  Transformations}}, SIAM Journal on Scientific and Statistical Computing, 9
  (1988), pp.~39--47.

\bibitem{Gear83}
{\sc C.~W. Gear and L.~R. Petzold}, {\em {Differential}/{Algebraic Systems and
  Matrix Pencils}}, in Matrix Pencils, B.~K{\aa}gstr{\"o}m and A.~Ruhe, eds.,
  Berlin, Heidelberg, 1983, Springer Berlin Heidelberg, pp.~75--89.

\bibitem{Geddes1992}
{\sc K.~O. Geddes, S.~R. Czapor, and G.~Labahn}, {\em Algorithms for Computer
  Algebra}, Springer US, Boston, MA, 1992, ch.~Gr{\"o}bner Bases for Polynomial
  Ideals, pp.~429--471.

\bibitem{Gerdts11}
{\sc M.~Gerdts}, {\em Optimal Control of ODEs and DAEs}, De Gruyter, 2011.

\bibitem{Golub13}
{\sc G.~H. Golub and C.~F. Van~Loan}, {\em Matrix Computations (4rd Ed.)},
  Johns Hopkins University Press, USA, 2013.

\bibitem{GUZEL06}
{\sc N.~Guzel and M.~Bayram}, {\em {On the Numerical Solution of
  Differential}-{Algebraic Equations with Index}-$3$}, Applied Mathematics and
  Computation, 175 (2006), pp.~1320--1331.

\bibitem{Hauenstein12}
{\sc J.~D. Hauenstein}, {\em {Numerically Computing Real Points on Algebraic
  Sets}}, Acta Applicandae Mathematicae, 125 (2012), pp.~105--119.

\bibitem{Hauenstein2017}
{\sc J.~D. Hauenstein and A.~J. Sommese}, {\em {What is Numerical Algebraic
  Geometry?}}, Journal of Symbolic Computation, 79 (2017), pp.~499--507.
\newblock Numerical Algebraic Geometry.

\bibitem{Ilchmann15}
{\sc A.~Ilchmann and T.~Reis}, {\em Surveys in Differential-Algebraic Equations
  I}, Springer, Berlin, Heidelberg, 2013.

\bibitem{Iwata03}
{\sc S.~Iwata}, {\em {Computing the Maximum Degree of Minors in Matrix Pencils
  via Combinatorial Relaxation}}, Algorithmica, 36 (2003), pp.~331--341.

\bibitem{KrantzParks02}
{\sc S.~G. Krantz and H.~R. Parks}, {\em Some Questions of Hard Analysis},
  Birkh{\"a}user Boston, Boston, MA, 2002.

\bibitem{KUNKEL94}
{\sc P.~Kunkel and V.~Mehrmann}, {\em {Canonical Forms for Linear
  Differential}-{Algebraic Equations with Variable Coefficients}}, Journal of
  Computational and Applied Mathematics, 56 (1994), pp.~225--251.

\bibitem{Kunkel06}
{\sc P.~Kunkel and V.~Mehrmann}, {\em {Differential}-{Algebraic Equations.
  Analysis and Numerical Solution}}, European Mathematical Society, 01 2006.

\bibitem{Kuranishi57}
{\sc M.~Kuranishi}, {\em On \'{E} {Cartan's Prolongation Theorem of Exterior
  Differential Systems}}, Amer. J. Math, 79 (1957), pp.~1--47.

\bibitem{Lamour13}
{\sc R.~Lamour, R.~M{\"a}rz, and C.~Tischendorf}, {\em
  {Differential}-{Algebraic Equations: A Projector Based Analysis}}, Springer,
  Berlin, Heidelberg, 1~ed., 01 2013.

\bibitem{Lee2008}
{\sc T.-L. Lee, T.~Li, and C.~Tsai}, {\em Hom4ps-2.0: A software package for
  solving polynomial systems by the polyhedral homotopy continuation method},
  Computing, 83 (2008), pp.~109--133.

\bibitem{Leimkuhler91}
{\sc B.~Leimkuhler, L.~R. Petzold, and C.~W. Gear}, {\em {Approximation Methods
  for the Consistent Initialization of Differential}-{Algebraic Equations}},
  SIAM Journal on Numerical Analysis, 28 (1991), pp.~205--226.

\bibitem{LIU201593}
{\sc C.-S. Liu}, {\em {Elastoplastic Models and Oscillators Solved by a
  Lie-group Differential Algebraic Equations Method}}, International Journal of
  Non-Linear Mechanics, 69 (2015), pp.~93--108.

\bibitem{LIU2007748}
{\sc H.~Liu and Y.~Song}, {\em {Differential Transform Method Applied to High
  Index Differential}-{Algebraic Equations}}, Applied Mathematics and
  Computation, 184 (2007), pp.~748--753.

\bibitem{Roswitha02}
{\sc R.~M{\"a}rz}, {\em The index of linear differential algebraic equations
  with properly stated leading terms}, Results in Mathematics, 42 (2002),
  pp.~308--338.

\bibitem{Mathews2004}
{\sc J.~H. Mathews and K.~K. Fink}, {\em Numerical Methods Using Matlab (4th
  Edition)}, Pearson, 4~ed., jan 2004.

\bibitem{Mattsson93}
{\sc S.~E. Mattsson and G.~S$\ddot{o}$derlind}, {\em {Index Reduction in
  Differential}-{Algebraic Equations Using Dummy Derivatives}}, SIAM Journal on
  Scientific Computing, 14 (1993), pp.~677--692.

\bibitem{Mazzia08}
{\sc M.~C. K.~J. Mazzia~F.}, {\em {Test Set for Initial Value Problem
  Solvers}}.
\newblock Department of Mathematics, 2008.

\bibitem{McKenzie17}
{\sc R.~McKenzie and J.~Pryce}, {\em {Structural Analysis Based Dummy
  Derivative Selection for Differential Algebraic Equations}}, BIT Numerical
  Mathematics, 57 (2017), pp.~433--462.

\bibitem{Murota95}
{\sc K.~Murota}, {\em {Computing the Degree of Determinants via Combinatorial
  Relaxation}}, SIAM J. Comput., 24 (1995), pp.~765--796.

\bibitem{Nedialkov08}
{\sc N.~Nedialkov and J.~Pryce}, {\em {Solving Differential Algebraic Equations
  by Taylor Series (III): the DAETS Code}}, European Society of Computational
  Methods in Sciences and Engineering (ESCMSE) Journal of Numerical Analysis,
  Industrial and Applied Mathematics, 3 (2008), pp.~61--80.

\bibitem{Taihei19}
{\sc T.~Oki}, {\em Improved structural methods for nonlinear
  differential-algebraic equations via combinatorial relaxation}, CoRR,
  abs/1907.04511 (2019).

\bibitem{Ollivier09}
{\sc F.~{Ollivier}}, {\em {Jacobi's Bound and Normal Forms Computations. A
  Historical Survey}}, 2009.

\bibitem{Pantelides88}
{\sc C.~C. Pantelides}, {\em The consistent initialization of
  differential-algebraic systems}, SIAM Journal on Scientific and Statistical
  Computing, 9 (1988), pp.~213--231.

\bibitem{Pryce98}
{\sc J.~D. Pryce}, {\em {Solving High}-{index DAEs by Taylor Series}},
  Numerical Algorithms, 19 (1998), pp.~195--211.

\bibitem{Pryce01}
{\sc J.~D. Pryce}, {\em A simple structural analysis method for daes}, BIT
  Numerical Mathematics, 41 (2001), pp.~364--394.

\bibitem{Rans}
{\sc C.~Rans and S.~T.~D. Freitas}, {\em {Bending Deflection} - {Differential
  Equation Method}}.
\newblock Aerospace Structures and Materials, 2016.

\bibitem{Reid01}
{\sc G.~J. Reid, P.~Lin, and A.~D. Wittkopf}, {\em {Differential Elimination} -
  {Completion Algorithms for DAE and PDAE}}, Studies in Applied Mathematics,
  106 (2001).

\bibitem{Seiler2010}
{\sc W.~Seiler}, {\em {Involution} - {The Formal Theory of Differential
  Equations and its Applications in Computer Algebra}}, vol.~24 of Algorithms
  and Computation in Mathematics, Springer, Berlin, Heidelberg, 01 2010.

\bibitem{Shampine02}
{\sc L.~Shampine}, {\em Solving 0=f(t,y(t),y'(t)) in {Matlab}}, Journal of
  Numerical Mathematics, 10 (2002), pp.~291--310.

\bibitem{Skvortsov12}
{\sc L.~M. Skvortsov}, {\em {Runge}-{Kutta Collocation Methods for
  Differential}-{Algebraic Equations of Indices $2$ and $3$}}, Computational
  Mathematics and Mathematical Physics, 52 (2012), pp.~1373--1383.

\bibitem{SommeseWampler05}
{\sc A.~Sommese and C.~Wampler}, {\em The Numerical Solution of Systems of
  Polynomials Arising in Engineering and Science}, World Scientific Pub Co Inc,
  03 2005.

\bibitem{SVW05}
{\sc A.~J. Sommese, J.~Verschelde, and C.~W. Wampler}, {\em Solving Polynomial
  Equations: Foundations, Algorithms, and Applications}, Springer Berlin
  Heidelberg, Berlin, Heidelberg, 2005, ch.~Introduction to Numerical Algebraic
  Geometry, pp.~301--337.

\bibitem{Tan17}
{\sc G.~Tan, N.~Nedialkov, and J.~Pryce}, {\em {Conversion Methods for
  Improving Structural Analysis of Differential}-{Algebraic Equation Systems}},
  BIT Numerical Mathematics, 57 (2017), pp.~845--865.

\bibitem{Vieira20011}
{\sc R.~Vieira and E.~Biscaia}, {\em {Direct Methods for Consistent
  Initialization of DAE Systems}}, Computers and Chemical Engineering, 25
  (2001), pp.~1299--1311.

\bibitem{Tischendorf98}
{\sc C.~von Tischendorf}, {\em {Topological Index Calculation of DAEs in
  Circuit Simulation}}, ZAMM ‐ Journal of Applied Mathematics and Mechanics /
  Zeitschrift f{\"u}r Angewandte Mathematik und Mechanik, 78 (1998).

\bibitem{WWX2017}
{\sc Y.~Wang, W.~Wu, and B.~Xia}, {\em {A Special Homotopy Continuation Method
  for a Class of Polynomial Systems}}, in Computer Algebra in Scientific
  Computing, V.~P. Gerdt, W.~Koepf, W.~M. Seiler, and E.~V. Vorozhtsov, eds.,
  Cham, 2017, Springer International Publishing, pp.~362--376.

\bibitem{WuChenReid17}
{\sc W.~Wu, C.~Chen, and G.~Reid}, {\em {Penalty Function Based Critical Point
  Approach to Compute Real Witness Solution Points of Polynomial Systems}}, in
  Computer Algebra in Scientific Computing, V.~P. Gerdt, W.~Koepf, W.~M.
  Seiler, and E.~V. Vorozhtsov, eds., Cham, 2017, Springer International
  Publishing, pp.~377--391.

\bibitem{WuReid13}
{\sc W.~Wu and G.~Reid}, {\em {Finding Points on Real Solution Components and
  Applications to Differential Polynomial Systems}}, in Proceedings of the 38th
  International Symposium on Symbolic and Algebraic Computation, ISSAC '13, New
  York, NY, USA, 2013, Association for Computing Machinery, pp.~339--346.

\bibitem{WRF17}
{\sc W.~Wu, G.~Reid, and Y.~Feng}, {\em {Computing Real Witness Points of
  Positive Dimensional Polynomial Systems}}, Theoretical Computer Science, 681
  (2017), pp.~217--231.
\newblock Symbolic Numeric Computation.

\bibitem{WRI09}
{\sc W.~Wu, G.~Reid, and S.~Ilie}, {\em {Implicit Riquier Bases for PDAE and
  Their Semi}-{discretizations}}, Journal of Symbolic Computation, 44 (2009),
  pp.~923--941.
\newblock International Symposium on Symbolic and Algebraic Computation.

\bibitem{Wu13}
{\sc X.~{Wu}, Y.~{Zeng}, and J.~{Cao}}, {\em {The Application of the
  Combinatorial Relaxation Theory on the Structural Index Reduction of DAE}},
  in 2013 12th International Symposium on Distributed Computing and
  Applications to Business, Engineering Science, 2013, pp.~162--166.

\bibitem{ZOLF2021}
{\sc R.~Zolfaghari, J.~Taylor, and R.~J. Spiteri}, {\em Structural analysis of
  integro-differential–algebraic equations}, Journal of Computational and
  Applied Mathematics, 394 (2021), p.~113568.

\end{thebibliography}
\end{document}